\documentclass[10pt,a4paper]{article}
\usepackage[margin=3cm]{geometry}

\RequirePackage{fix-cm}
\usepackage{graphicx}
\usepackage{multirow,comment}
\usepackage{amsmath,amssymb,amsfonts,amsthm}
\usepackage{mathrsfs}
\usepackage[title]{appendix}
\usepackage{xcolor}
\usepackage{textcomp}
\usepackage{manyfoot}
\usepackage{booktabs}
\usepackage[ruled,vlined]{algorithm2e}
\usepackage{listings}
\usepackage{bm}
\usepackage{anyfontsize}
\usepackage{cases}
\usepackage{caption}
\usepackage{subcaption}
\usepackage{graphicx,epsfig,float,enumitem}
\usepackage{graphicx}
\usepackage[
    colorlinks=true,  
    linkcolor=blue,    
    citecolor=blue,    
    urlcolor=blue     
]{hyperref}
\def\diam{{\rm diam}}
\def\dist{{\rm dist}}

\newtheorem{lemma}{Lemma}

\newtheorem{definition}{Definition}
\newtheorem{theorem}{Theorem}
\newtheorem{assumption}{Assumption}
\newtheorem{proposition}{Proposition}

\numberwithin{lemma}{section}
\numberwithin{example}{section}
\numberwithin{theorem}{section}
\numberwithin{proposition}{section}
\numberwithin{remark}{section}

\begin{document}

\title{A Convergent Inexact Abedin-Kitagawa Iteration
Method for Monge-Amp{\`e}re Eigenvalue Problems
\thanks{This work was funded by the National Key R \& D Program of China (No. 2021YFA001300), the
National Natural Science Foundation of China (Nos. 12271150, 12471374), the Hunan Provincial Natural Science
Foundation of China (Nos. 2023JJ10001, 2025JJ40001), 
the Science and Technology Innovation Program of Hunan
Province (No. 2022RC1190), and the Key Scientific Research Project of the Education Department of Hunan Province (No. 23A0034).}
}

\date{Received: date / Accepted: date}
\author{
	Liang Chen\thanks{School of Mathematics, Hunan University, Changsha 410082, China, and 
		Hunan Provincial Key Laboratory of Intelligent Information Processing and Applied Mathematics, Changsha 410082, China
		(\url{chl@hnu.edu.cn}).} 
	\quad  
	Youyicun Lin\thanks{School of Mathematics, Hunan University, Changsha, 410082, China
		(\url{linyouyicun@hnu.edu.cn}).}
	\quad 
	Junqi Yang\thanks{School of Mathematics, Hunan University, Changsha, 410082, China
		(\url{junqiyang@hnu.edu.cn}).}
	\quad
	Wenfan Yi\thanks{School of Mathematics, Hunan University, Changsha 410082, China, and 
		Hunan Provincial Key Laboratory of Intelligent Information Processing and Applied Mathematics, Changsha 410082, China
		(\url{wfyi@hnu.edu.cn}).} 
} 

\date{}

\maketitle
\begin{abstract}

This paper proposes an inexact Aleksandrov-solution-based iteration method, formulated by adapting the convergent Rayleigh inverse iterative scheme introduced by Abedin and Kitagawa, to solve real Monge-Amp{\`e}re eigenvalue (MAE) problems. 
The central feature of the proposed approach is the introduction of a flexible error tolerance criterion for computing inexact Aleksandrov solutions to the required subproblems. This allows the inner iteration to be solved approximately without compromising the global convergence properties of the overall scheme, as we established under a ${\cal C}^{2,\alpha}$ boundary condition, and has the potential of achieving reduced computational cost compared to the original algorithm. 
In practice, for both two- and three-dimensional problems, by leveraging the flexibility of the inexact iterative formulation in conjunction with a fixed-point approach for solving the subproblems, the proposed method performs several times faster than its original version of Abedin and Kitagawa, across all tested problem instances in the numerical experiments. 


\medskip
\noindent 
{\bf MSCcodes:}
65H17, 35J96, 65N25
\end{abstract}

\section{Introduction}
The (real) \emph{Monge-Amp{\`e}re eigenvalue} (MAE) problem \cite{abedin2020inverse,le2017eigenvalue,le2020convergence,lions1985two,Tso1990real} considers finding the eigen-pairs $(\lambda,u)$  such that 
\begin{equation}
\label{eq_MAeig}
\begin{cases}
\det D^2u =\lambda |u|^d &
\mbox{ in } \Omega,\\
u=0 &  
\mbox{ on } \partial \Omega,
\end{cases}
\end{equation}
where $\Omega\subset \mathbb{R}^d$ $(d \ge 2)$ is an open and bounded convex domain with $\partial\Omega$ being the boundary, the Monge-Amp{\`e}re eigenfunction $u:\Omega\cup \partial\Omega\to\mathbb{R}$ is convex, the MAE $\lambda\in \mathbb{R}$ is a positive number, and $\det D^2$ denotes the Monge-Amp{\`e}re operator with the determinant $\det$ and the Hessian $D^2$. 
The existence and the regularity of the eigen-pairs for the MAE problem \eqref{eq_MAeig} have been guaranteed when  $\Omega\subset\mathbb{R}^d$ ($d\ge 2$) is smooth and uniformly convex \cite{lions1985two}. 
It was later considered by Tso \cite{Tso1990real} from a variational perspective. More recently, these results were extended by Le \cite{le2017eigenvalue} to the case of general bounded convex open domains.  In addition, Le \cite{le2025global} showed that the Monge-Amp{\`e}re  eigenfunctions  are globally Lipschitz and established global $W^{2,1}$ estimates for the Monge-Amp{\`e}re  eigenfunctions.

Beyond the real MAE problem \eqref{eq_MAeig}, Lions \cite{lions1985two} suggested that his method might apply to the complex Monge-Amp{\`e}re operator. 
Moreover, Badiane et al. \cite{badiane2023eigenvalue} established the existence and uniqueness of the first eigenpair for the complex Monge-Amp{\`e}re operator on bounded strongly pseudoconvex domains, following the strategy of Lions in the real case.  
Subsequently, Zeriahi \cite{zeriahi2025iterative} developed an effective iterative approximation scheme for approximating both the eigenvalue and eigenfunction without prior knowledge of the first eigenvalue. 
Later, Lu et al. \cite{lu2025new} proposed a new iterative approach in bounded hyperconvex domains, establishing uniqueness in the finite energy class introduced by Cegrell \cite{cegrell1998pluricomple} and a Rayleigh-type characterization of the eigenvalue. 
These developments constitute a parallel theoretical development to the real MAE problem.

Based on a prior acknowledgment of the eigenvalue $\lambda$, the MAE problem \eqref{eq_MAeig} can be transformed to solve the corresponding Monge-Amp{\`e}re (MA) equation, which takes the classical form \cite{benamou2010two,dean2003numerical,dean2004numerical,dean2006numerical,feng2009mixed,froese2011convergent,froese2011fast,glowinski2019finite,liu2019finite} given by
\begin{equation}
\label{eq:MA}
\det D^2u = f \quad \mbox{ in } \Omega,\
\end{equation}
where $f$ is a given function on $\Omega$. 
The MA equation \eqref{eq:MA} and related problems, 
especially the MAE problem \eqref{eq_MAeig}, have been considered fundamental models in various geometric problems and applied fields, which are not only theoretically significant but also have profound implications in practical areas \cite{benamou2000computational,engquist2016optimal,haker2004optimal,stojanovic2004optimal}. 
The MAE problem can be regarded as a nonlinear analogue of classical eigenvalue problems and a natural extension of the MA equation, playing a central role in nonlinear partial differential equations.
For instance, in the prescribed Gaussian curvature \cite{bakelman1994convex,kazdan1985prescribing}, the MAE $\lambda$ was used to determine the extinction rate of non-parametric surfaces evolving by the $n$-th root of their Gaussian curvature \cite{li2017nonparametric,oliker1991evolution}.  
Recently, for the unit ball domain, an explicit bound on the MAE has been established by Le \cite{le2025large}. 
More generally, the study and computation of Monge-Amp{\`e}re eigenpairs $(\lambda, u)$ are not only closely linked to fundamental inequalities, including the Brunn-Minkowski, isoperimetric, and reverse isoperimetric inequalities \cite{brandolini2009new,hartenstine2009brunn,le2017eigenvalue,salani2005brunn}, but also reveal intrinsic structures of fully nonlinear elliptic equations. For more background and details on MA equations and related MAE problems, we refer to the monograph \cite{le2024analysis}. 

Due to the highly nonlinear nature of the MAE problem, efficiently solving it numerically poses significant challenges. 
Therefore, special discretization methods need to be designed to handle the determinant structure of the Monge-Amp{\`e}re operator. 
Although various numerical schemes have been developed for the MA equation, efficient and reliable methods specifically tailored for MAE problems remain far from well established. 
To address these challenges and build on extensive experiences with the MA equation, recent research has focused on developing effective numerical methods for both \eqref{eq:MA} and \eqref{eq_MAeig}. 
In particular, the MA equation \eqref{eq:MA} has been formulated as an optimization problem in \cite{caboussat2018least,caboussat2013least,dean2003numerical,dean2004numerical}, leading to augmented Lagrangian and least-squares methods. 
Moreover, finite difference, finite element, and operator-splitting schemes have been introduced for the MA equation \eqref{eq:MA} in \cite{awanou2015standard,feng2009mixed,froese2011convergent,froese2011fast,glowinski2019finite,imen2024convergent,liu2019finite}; see the references therein for details.

For the MAE problem, Glowinski et al.  \cite{glowinski2020numerical} took advantage of an equivalent divergence formulation of the MA equation and reformulated the MAE problem \eqref{eq_MAeig} as
\begin{equation*}
\begin{cases}
-\nabla \cdot({\rm cof}(D^2u)\nabla u) =d\lambda u |u|^{d-1}  &
\mbox{ in } \Omega,\\
u=0 &  
\mbox{ on } \partial \Omega,
\end{cases}
\end{equation*}
where ${\rm cof}(D^2u)$ is the cofactor matrix of $D^2u$. Then they treated the MAE problem \eqref{eq_MAeig} as the optimality condition of a constrained optimization problem, and proposed an operator-splitting method to solve it. 
Around the same time, Abedin and Kitagawa \cite{abedin2020inverse} computed the eigen-pairs $(\lambda, u)$ by introducing an iterative method (named the Abedin-Kitagawa iteration (AKI) method for ease of reference) with a sequence of subproblems taking the form of the MA equation. 
Specifically, the AKI method is given by handling the following iterative problem:
\begin{equation}
\label{eq:MAsub}
\begin{cases}
\det D^2u_{k+1} =R(u_k)|u_k|^d &  \mbox{ in } \Omega,\\
u_{k+1}=0 &\mbox{ on } \partial \Omega, 
\end{cases}
\qquad  k=0,1,\ldots, 
\end{equation}
where $u_{k+1}$ is an unknown convex Aleksandrov solution (c.f. Definition \ref{definitiona: Aleksandrov}) based on the computed $u_{k}$ from the previous iteration, 
and 
$R(\cdot)$ is the Rayleigh quotient defined by 
\begin{equation}
\label{rayli} 
R(u):= \frac{\int_{\Omega}|u|  \det D^2 u\ \mathrm{d} \bm{x} }{\int_{\Omega}|u|^{d+1} \, \mathrm{d} \bm{x}},
\end{equation}
with $\det D^2 u\, \mathrm{d} \bm{x}$ denoting the Monge-Amp{\`e}re measure associated with the given convex function $u$ on $\Omega$ (c.f. Section \ref{sec:notaions} for more details).
In particular, under the assumptions for the initial convex function $u_0 \in {\mathcal{C}}(\overline{\Omega})$ that (i) $u_0\le 0$ on $\partial \Omega$, (ii) $R(u_0) < \infty$,
and (iii) $\det D^2 u_0\ge c_0$ in $\Omega$ with some constant $c_0\ge 0$, the convergence analysis of the AKI method was established in \cite{abedin2020inverse}. 
Moreover, improvements in the convergence analysis of the AKI method were achieved in \cite{le2020convergence} by replacing the above conditions (i)-(iii) with the condition that $0<R(u_0) < \infty$.  
Building on \cite{abedin2020inverse} and \cite{glowinski2020numerical}, Liu et al. \cite{liu2022efficient} proposed to solve the MAE problem \eqref{eq_MAeig} by utilizing the first-order operator-splitting method in solving the subproblem \eqref{eq:MAsub} of the AKI method.

From the theoretical perspective of the AKI method, an exact Aleksandrov solution should be obtained for the subproblem \eqref{eq:MAsub} at the $k$-th iteration to ensure convergence, but this is generally not achievable in practice due to the nonlinearity of the Monge-Amp{\`e}re subproblem \eqref{eq:MAsub}, making it highly challenging to obtain exact Aleksandrov solutions in numerical implementations. 
Moreover, from a numerical perspective, the subproblems in the form of \eqref{eq:MAsub} are typically solved using iterative methods, such as operator-splitting algorithms.
Thus, these require approximate Aleksandrov solutions to all subproblems with high accuracy, which may make the AKI method very slow. 
Therefore, how to verify whether an approximate Aleksandrov solution to the subproblem \eqref{eq:MAsub} can be accepted without violating the convergence of the AKI method constitutes a crucial issue for the efficient implementation of the AKI iteration method.  
 
To properly address the problems within the above exposition, in this paper, we propose an inexact AKI method for solving the MAE problem \eqref{eq_MAeig}, 
aiming at enhancing the computational efficiency, yet not spoiling the convergence.
We achieve this by introducing an error function for the inexact computations of the subproblems that satisfies certain error tolerance criteria controlled by a summable sequence of nonnegative real numbers. This provides a flexible framework for managing numerical errors, particularly when iterative methods are employed to solve the subproblems. 
Moreover, we prove that the proposed inexact AKI method admits the same convergence properties as those of the original AKI method in \cite{abedin2020inverse,le2020convergence}. 
Motivated by the recent advances developed by Benamou et al. \cite{benamou2010two,froese2011convergent} for solving the MA equation in dimensions two and higher with the Dirichlet boundary condition, we solve the subproblems in the inexact AKI method by reformulating them into a fixed-point problem for MAE problems in two and three dimensions.
Furthermore, under the $\mathcal{C}^{2,\alpha}$ boundary condition, we provide a rigorous convergence analysis of the proposed fixed-point method for two- and three-dimensional (2D and 3D) cases.
Finally, extensive numerical experiments are performed to solve both 2D and 3D MAE problems using the proposed inexact fixed-point-based AKI method. The performance of the proposed method is compared with the original AKI method (where the subproblems are solved with sufficiently high accuracy, regarded as solved exactly) and the operator-splitting approach proposed in \cite{liu2022efficient}. 
The numerical results suggest that the proposed algorithm in this paper can significantly reduce computational time (several times faster while maintaining the same precision) compared to the existing methods, achieving much higher accuracy than the operator-splitting approach. 

The remaining parts of this paper are organized as follows. 
In Section \ref{sec:notaions}, we provide the notation and preliminaries used throughout this paper. 
Then, the inexact AKI method is proposed in Section \ref{sec:An Inexact AKI method} to solve the MAE problem \eqref{eq_MAeig}, and two inexact frameworks are analyzed with their corresponding convergence results. 
In addition, we present a fixed-point iteration to solve the corresponding subproblem and establish its convergence under the $\mathcal{C}^{2,\alpha}$ boundary condition for 2D and 3D cases in Section \ref{sec:fixed}.
Numerical experiments and results are presented in Section \ref{sec:experiments} to validate the effectiveness and efficiency of the proposed algorithm. 
Finally, we conclude this paper with some discussion in Section \ref{sec:conclusions}.

\section{Notation and preliminaries}
\label{sec:notaions}
 This section presents the notation used throughout this paper and some preliminary results related to the MA equation \eqref{eq:MA}, following from two monographs  \cite{figalli2017monge,gutierrez2016monge} and two papers \cite{abedin2020inverse,hartenstine2006dirichlet}. Additionally, we will present some results related to the MAE problem \eqref{eq_MAeig}, following the work in \cite{le2017eigenvalue}.

Let $\Omega$ be an open convex domain contained in $\mathbb{R}^{d}$ with $\partial\Omega$ being its boundary and $\overline\Omega:=\Omega\cup\partial\Omega$. 
Unless otherwise specified, bounded open convex domains are assumed to have non-empty interiors throughout this paper. 
For $n \in \mathbb{N}_0: = \{0,1,2,\ldots\}$ and $\alpha\in(0,1]$, the domain $\Omega$ is said to have a $\mathcal{C}^{n,\alpha}$ boundary if, for every point on the boundary of the domain, the boundary hypersurface can be realized, after an appropriate rotation of coordinates, as a $\mathcal{C}^{n,\alpha}$ function.
In addition, denote $\mathcal{C}^{n,\alpha}(\overline{\Omega})$ $ (\mathcal{C}^{n,\alpha}(\Omega))$ as the H{\"o}lder space consisting of all functions that have continuous derivatives up to order $n$ and whose $n$-th partial derivatives are uniformly (locally) H{\"o}lder continuous with exponent $\alpha$ in $\overline{\Omega} $ $(\Omega)$. For any $u \in \mathcal{C}^{n,\alpha}(\overline{\Omega}) $, the H\"older $\mathcal{C}^{n,\alpha}$-norm of $u$ is equipped with
$$
    \|u\|_{\mathcal{C}^{n,\alpha}(\overline{\Omega})}: =\| u\|_{\mathcal{C}^{n}(\overline{\Omega})}+|u|_{n,\alpha;\Omega},
$$
where 
$\| u\|_{\mathcal{C}^{n}(\overline{\Omega})} := \sum\limits_{|\beta|\le n}\sup\limits_{\bm{x}\in \overline{\Omega}}|\partial^{\beta}u(\bm{x})|$ 
and $|u|_{n,\alpha;\Omega}$ is the H{\"o}lder semi-norm defined by
$$
\begin{array}{ll}
|u|_{n,\alpha;\overline{\Omega}}:=\sum\limits_{|\beta|=n}\sup\limits_{\substack{\bm{x}, \bm{y} \in \overline{\Omega}\\ \bm{x} \neq \bm{y}}}\dfrac{|\partial^\beta u(\bm{x})-\partial^\beta u(\bm{y})|}{\|\bm{x}-\bm{y}\|^\alpha}.
\end{array}
$$
Here, $\beta = (\beta_1,\ldots,\beta_d)$ denotes a multi-index with $\beta_i\in \mathbb{N}_0$, and $\partial^{\beta}$ denotes the corresponding partial derivative.
The space $\mathcal{C}^{n,\alpha}_{\rm loc}(\Omega)$ denotes functions that satisfy the $\mathcal{C}^{n,\alpha}$-regularity locally in $\Omega$, i.e., for every compact subset $K \subset \Omega$, the function restricted to $K$ belongs to $\mathcal{C}^{n,\alpha}(K)$.
The subdifferential mapping of a convex function  $u:\Omega \to \mathbb{R}$ is defined by
$$
\partial u (\bm{x}):=\left\{\bm{x}^* \in \mathbb{R}^{d} \mid u(\bm{z})  \ge u(\bm{x})+ \langle \bm{x}^*, \bm{z}-\bm{x} \rangle \quad \forall\; \bm{z}\in \Omega
\right\},\quad \bm{x}\in \Omega.
$$

The Monge-Amp{\`e}re measure and the Aleksandrov solution of convex functions are given in the following definitions. 
\begin{definition}[{Monge-Amp{\`e}re measure; \cite[Definition 2.1]{figalli2017monge}}]
 Let $u:\Omega\rightarrow \mathbb{R}$ be a convex function, the Monge-Amp{\`e}re measure associated to  $u$ is defined by
$$
\mathcal{M}u(E) := \mathcal{L}^d(\partial u(E)),
$$
where $\partial u(E) =\bigcup \limits_{\bm{x}\in E}\partial u(\bm{x})$ for each Borel set $E\subset\Omega$ and $\mathcal{L}^d$ denotes  $d$-dimensional  Lebesgue measure. 
\end{definition} 

\begin{definition}[{Aleksandrov solutions; \cite[Definition 2.5]{figalli2017monge}}]
\label{definitiona: Aleksandrov}
Given a convex open set $\Omega$ and a Borel measure  $\nu$ on $\Omega$, 
let $u$ : $\Omega \to \mathbb{R}$ be a convex function and ${\mathcal M}u$ be the Monge-Amp{\`e}re measure of $u$, we say that $u$ is an Aleksandrov solution of the equation
$$
\det D^2u=\nu,
$$
if $\nu={\mathcal M}u$.  
When $\nu = f \ \mathrm{d} \bm{x}$ we will say for simplicity that $u$ solves 
$$
\det D^2u=f.
$$
\end{definition}

Note that, if $u \in \mathcal{C}^{2}(\Omega) $, the change of variable formula gives
$$
\mathcal{L}^d( \partial u(E)) = \mathcal{L}^d(\nabla u(E)) = \int_{E} \det D^{2}u \ \mathrm{d} \bm{x}\quad     \quad {\rm for \ each\ Borel\  set}\  E \subset \Omega.
$$
Thus one has $\mathcal{M}u=\det D^{2}u(\bm{x})\  \mathrm{d} \bm{x}$, immediately.  
Throughout this paper, by slightly abusing the notation, we denote the Monge-Amp{\`e}re measure  $\mathcal{M}u$ of a general convex function $u$ by $\det D^2 u$. For each Borel set 
$E\subset\Omega$, one has 
$$\int_E \det D^2 u\ \mathrm{d} \bm{x}=\mathcal{M}u(E)
\quad\mbox{and}
\quad \int_E |u|\det D^2 u\ \mathrm{d} \bm{x}=\int_E |u|\ \mathrm{d} \mathcal{M}u.
$$

The existence and uniqueness of the Aleksandrov solution to the Dirichlet problem for the MA equation under general conditions are guaranteed by the following fundamental result.
\begin{lemma}[{Solvability of Dirichlet problem; \cite[Theorem 1.1]{hartenstine2006dirichlet}}]
\label{solvability}
Let $\Omega$ be a bounded convex open domain in $\mathbb{R}^{d}$ and $\nu$ be a nonnegative Borel measure in $\Omega$, 
then there exists a unique convex function $u\in \mathcal{C}(\overline{\Omega})$ constituting an Aleksandrov solution of the Dirichlet problem
\begin{equation*}
\begin{cases}
\det D^2u =\nu  & \text{ in }  \Omega,\\
u=0 & \text{ on } \partial \Omega.
\end{cases}
\end{equation*}
\end{lemma}

The following Aleksandrov maximum principle, the comparison principle, and the continuity property of the Monge-Amp{\`e}re energy are important in our convergence analysis. We  write $\mathcal{M}u \ge \nu$ in $\Omega$ (resp. $\mathcal{M}u \le \nu$ in $\Omega$) 
if $\mathcal{M}u(E) \ge \nu(E)$ (resp. $\mathcal{M}u(E) \le \nu(E)$) for all Borel sets 
$E \subset \Omega$.
\begin{lemma}[{Aleksandrov maximum principle; \cite[Theorem 1.4.2]{gutierrez2016monge}}]
\label{Aleksandrov maximum principle}
Let  $\Omega$ be a bounded convex open domain in $\mathbb{R}^{d}$.
 Suppose $u\in \mathcal{C}(\overline{\Omega})$ is convex and vanishes on  $\partial \Omega$, then there exists a constant $C_d > 0$ depending
only on the dimension $d$ such that
$$
|u(\bm{x})|^d \le C_d \diam(\Omega)^{d-1} 
\dist(\bm{x},\partial\Omega)\int_\Omega \det D^2 u\ \mathrm{d} \bm{x}   \quad \forall\; \bm{x} \in\Omega,
$$
where $\diam(\Omega)$ is the radius of $\Omega$.
\end{lemma}

\begin{lemma}[{Comparison principle; \cite[Theorem 1.4.6]{gutierrez2016monge}}]
\label{Comparison Principle}
Let  $\Omega$ be a bounded convex open domain in $\mathbb{R}^{d}$. Suppose $u, v \in \mathcal{C}(\overline{\Omega})$ are convex functions satisfying $u \ge v$ on $\partial\Omega$ and
$$ 
\mathcal{M}u \le \mathcal{M}v.
$$
Then $u \ge v$  in $\Omega$.
\end{lemma}

\begin{lemma}[{Continuity property of the Monge-Amp{\`e}re energy; \cite[Lemma 2.9]{abedin2020inverse}}]
\label{Continuity property}
Let  $\Omega$ be a bounded convex open domain in $\mathbb{R}^{d}$. Suppose that $v_k \in \mathcal{C}(\overline{\Omega})$ are convex functions converging uniformly on $\overline{\Omega}$ 
to a function $v$, and there exists a constant $b > 0$ such that 
$\mathcal{M}v_k \le b \mathcal{L}^d$ for all $k \ge 0$. 
Then
\[
\lim_{k \to \infty} I(v_k) = I(v),
\]
where
\[
I(u):= \int_\Omega |u|\det D^2 u\ \mathrm{d} \bm{x}.
\]
\end{lemma}
The next Lemma shows that if $u \in \mathcal{C}(\overline{\Omega})$ is convex and vanishes on $\partial \Omega$, then all $L^p$ norms of $u$ are comparable. This lemma plays an important role in establishing the upper bound of $R(u_k)$ for all $k\ge1$.

\begin{lemma}[{Comparability of $L^p$ norms; \cite[Lemma 2.10]{abedin2020inverse}}]
\label{Comparability}
If $u \in \mathcal{C}(\overline{\Omega})$ is convex and vanishes on $\partial \Omega$, then for all $p \ge 1$, 
\[
\frac{\|u\|_{L^{\infty}(\Omega)}}{d+1}
    \le \Big( \dfrac{1}{\mathcal{L}^d(\Omega)} \int_{\Omega} |u|^p \, d\bm{x} \Big)^{\!{1}/{p}}
    \le \|u\|_{L^{\infty}(\Omega)}.
\]
\end{lemma}

At last, we focus on the MAE problem \eqref{eq_MAeig}, for which one can define the constant
\begin{equation}
\label{inf}
\lambda_{\rm MA}:=\inf\limits_{w\in \mathcal{K}} R(w),
\end{equation}
 where  $R$ is the Rayleigh quotient defined by \eqref{rayli}, and 
$$
\mathcal{K} = \{  u \in \mathcal{C}(\overline{\Omega})\mid  u \text{ is convex and non-zero in } \Omega, \ u = 0 \text{ on } \partial \Omega \}.
$$
The constant $\lambda_{\rm MA}$ is called the MAE of $\Omega$. 
If the infimum in \eqref{inf} is attained at a certain function $w$, it is called the corresponding Monge-Amp\`ere eigenfunction.
Moreover, one has the following result regarding the existence and uniqueness of the solution to the MAE problem \eqref{eq_MAeig}, which establishes the mathematical justification for connecting the subproblems \eqref{eq:MAsub} and the MAE problem \eqref{eq_MAeig}.

\begin{lemma}[Existence and uniqueness of solution; {\cite[Theorem 1.1]{le2017eigenvalue}}]
\label{existence and uniqueness}
Let $\Omega$ be a bounded open convex domain in $\mathbb{R}^{d}$. 
Then the following assertions hold:
\begin{enumerate}
\item[\bf (\romannumeral1)]
The infimum in \eqref{inf} is achieved by a non-zero convex solution $w\in \mathcal{C}^{0,\beta}(\overline{\Omega})\cap \mathcal{C}^{\infty}(\Omega)$ for all $\beta\in (0, 1)$ to the MAE problem \eqref{eq_MAeig}.
 
\item [\bf (\romannumeral2)]
If the pair $(\Lambda, \tilde w)$ 
satisfies $\det D^2 \tilde w =\Lambda |\tilde w|^d$ in $\Omega$, where $\Lambda>0$ is a positive constant and $\tilde w\in \mathcal{K}$, 
then $\Lambda=\lambda_{\rm MA}$ and $\tilde w=m w$ for some positive constant $m$.
\end{enumerate}
\end{lemma}

\section{An inexact AKI method for MAE problems}
\label{sec:An Inexact AKI method}
In this section, based on the inexact Aleksandrov solutions to the iterative subproblems \eqref{eq:MAsub}, we propose an inexact AKI method to solve the MAE problem \eqref{eq_MAeig} along with its convergence analysis. 
Note that, from a computational perspective, the subproblem \eqref{eq:MAsub} in the AKI method is to solve an MA equation with the Dirichlet boundary condition, which is a fully nonlinear second-order partial differential equation. 
Solving the subproblems \eqref{eq:MAsub} with high accuracy is generally difficult and computationally expensive. 
So we consider introducing an approximate Aleksandrov solution to the subproblem \eqref{eq:MAsub} with relatively low accuracy, leading to the perturbed subproblem \eqref{app sub} presented in Algorithm \ref{alg:inexact_AK}, and propose the following inexact AKI method.  

\begin{algorithm}[H]
\caption{An inexact AKI method for the MAE problem} 
\label{alg:inexact_AK}
\KwIn{An initial non-zero convex function $u_0$, 
a summable sequence $\{\xi_{k}\}_{k \geq 0}$ of nonnegative real numbers satisfying $\xi_{k} \in [0,1)$ for all $k\ge0$.}
\KwOut{A sequence of  real  numbers $\{R(u_k)\}_{k\ge0}$ and a sequence of functions $\{u_k\}_{k\ge0}$.} 
\For {$k=0,1,\ldots,$}{
  calculate $R(u_k)$  according to \eqref{rayli};

  compute an approximate Aleksandrov solution ${u}_{k+1}$ of the perturbed subproblem 

\begin{equation}
\label{app sub}  
\begin{cases}
\det  D^2  u_{k+1} =R(u_k)|u_k|^d  +\varepsilon_k&  \mbox{ in } \Omega,\\
u_{k+1}=0 & \mbox{ on } \partial \Omega,
\end{cases}
\end{equation}

 such that $\|\varepsilon_k\|_{L^{\infty}(\Omega)}\le  \xi_k $ and $\varepsilon_k \geq -\xi_kR(u_k)|u_k|^d$.}
\end{algorithm}

\medskip 
Termed the \emph{inexact AKI method} for convenience, Algorithm \ref{alg:inexact_AK} modifies the subproblem \eqref{eq:MAsub} of the AKI method to the perturbed subproblem \eqref{app sub}, 
where $\{\varepsilon_k:=\varepsilon_k(\bm{x})\}_{k\ge0}$ is a sequence of functions that quantify errors.
The design of such an error criterion was motivated by the inexact proximal point algorithm to find zero points for maximal monotone operators \cite{rockafellar1976monotone} and its potential in significantly improving computational efficiency, as examined, e.g., in \cite{chen2021equivalence,chen2017efficient}.  
Moreover, since $\varepsilon_k \geq -\xi_kR(u_k)|u_k|^d$ and $\xi_k\in  [0,1)$  for all $k\ge0$, it follows that $R(u_k)|u_k|^d  +\varepsilon_k \ge0 $. Therefore, by Lemma   \ref{solvability}, Algorithm \ref{alg:inexact_AK} is guaranteed to generate an infinite sequence $\{u_k\}_{k\geq0}$, ensuring its well-definiteness.

When analyzing the convergence behavior of Algorithm \ref{alg:inexact_AK}, the signs of 
$\{\varepsilon_k(\bm{x})\}_{k\ge0}$ play a decisive role. 
In the following, we separate the analysis on the uniform convergence of the sequence $\{u_k\}_{k\ge0}$ to a non-zero Monge-Amp{\`e}re eigenfunction into two scenarios. 
For the general case, our analysis in Section \ref{analysis_general} mainly follows the original work of Abedin and Kitagawa \cite{abedin2020inverse}, in which the subproblems are required to be solved accurately, under the same assumptions as it.
For the more special case that $\{\varepsilon_k(\bm{x})\}_{k\ge0}$ are nonnegative functions, our analysis in Section \ref{analysis_special}  can be viewed as an inexact extension of that in \cite{le2020convergence}, 
in which errors in solving subproblems are not allowed, without modifying its assumptions.
Our analysis in this part also extends that in \cite{le2022spectral}, in which the errors are positive constant functions.

\subsection{Convergence analysis for the general case}
\label{analysis_general}
In this subsection, we study the convergence of Algorithm \ref{alg:inexact_AK} for the general setting
in which the perturbation functions $\{\varepsilon_k(\bm{x})\}_{k\ge0}$ are allowed to take negative values. 
Suppose $\{u_k\}_{k\ge0}$ is the sequence generated by Algorithm \ref{alg:inexact_AK} under the error criterion $\|\varepsilon_k\|_{L^{\infty}(\Omega)}\le  \xi_k $ and $\varepsilon_k \geq -\xi_kR(u_k)|u_k|^d$ for all $k\ge0$.
We begin with the following standing assumption for the initial function $u_0$, which has been used in \cite{abedin2020inverse} to analyze the convergence of the AKI method.

\begin{assumption}
\label{ass_blanket1}
$\Omega$ is a bounded convex open domain in $\mathbb{R}^d$, 
$u_0 \in \mathcal{C}(\overline{\Omega})$ be a convex function on $\Omega$ satisfying $u_0 \le 0$ on $\partial \Omega$,
with $R(u_0)<\infty$ and $\det D^2 u_0 \ge c_0$ in $\Omega$ for some constant $c_0>0$.
\end{assumption}

Denote $M := \sum_{k=0}^\infty \xi_k < \infty$. 
The following result provides a uniform lower bound for the sequence $\{u_k\}_{k\ge0}$.
\begin{proposition}
\label{inftyinf} 
Under Assumption \ref{ass_blanket1}, suppose $\{u_k\}_{k\ge0}$ to be the sequence generated by Algorithm \ref{alg:inexact_AK}. 
Then there exists a constant $c_1>0$ such that, for any $k \ge 0$,
\begin{align}
    \|u_{k+1}\|_{L^\infty(\Omega)} \ge c_1 \lambda_{\rm MA}^{-1/d}. 
\end{align}

\end{proposition}
\begin{proof}
 By Lemma \ref{solvability}, for all $k\ge0$, $u_{k+1}\in \mathcal{C}(\overline{\Omega})$  is a convex function satisfying  $u_{k+1}=0$ on $\partial\Omega$. Assumption~\ref{ass_blanket1} implies that the initial function $u_0$ satisfies $\det D^2 u_0 \ge c_0$ in $\Omega$ for some constant $c_0>0$. 
According to Lemma \ref{existence and uniqueness}, there exists a Monge-Amp\'ere eigenfunction $\omega \in \mathcal{K} \cap \mathcal C^\infty(\Omega)$ that solves the MAE problem \eqref{eq_MAeig} and satisfies $\|\omega\|_{L^\infty(\Omega)}^d = c_0 \lambda_{\rm MA}^{-1}$. 
Then for any Borel set $E \subset \Omega$, this yields
$$
    \int_{E}\det D^2\omega\,\mathrm{d}\bm{x}  = \lambda_{\rm MA} \int_E |\omega|^d \,\mathrm{d}\bm{x} 
    \le \lambda_{\rm MA}c_0 \lambda_{\rm MA}^{-1} \mathcal{L}^d(E)
    \le \int_{E}\det D^2 u_0\,\mathrm{d}\bm{x}.
$$
In addition, since $\omega \ge u_0$ on $\partial\Omega$, the comparison principle (Lemma \ref{Comparison Principle}) implies that $\omega \ge u_0$ in $\Omega$.
In the following, by induction, we prove that for all $k \ge 0$,  
\begin{equation}
   u_{k+1}\leq  \prod_{i=0}^k (1 - \xi_i)^{1/d} \omega  \quad \text{in } \Omega.
    \label{eq:3.20}
\end{equation}
In fact, when $k = 0$, recalling the perturbed subproblem \eqref{app sub} and the definition of $\lambda_{\rm MA}$ in \eqref{inf}, 
we can obtain for every Borel set $E \subset \Omega$,
$$
\begin{array}{ll}
\displaystyle
\int_{E}\det D^2u_1\,\mathrm{d}\bm{x} = R(u_0) \int_E |u_0|^d \, \mathrm{d}\bm{x}  + \int_E\varepsilon_0 \, \mathrm{d}\bm{x} 
\\
\displaystyle
\ge (1 - \xi_0) R(u_0) \int_E |u_0|^d \, \mathrm{d}\bm{x} 
\ge (1 - \xi_0)\lambda_{\rm MA} \int_E |\omega|^d \, \mathrm{d}\bm{x} 
= \int_E \det D^2 \left((1 - \xi_0) ^{{1}/{d}}\omega \right)\, \mathrm{d}\bm{x}.  
\end{array}
$$
Moreover, since $(1 - \xi_0) ^{{1}/{d}}\omega = u_1 = 0$ on $\partial\Omega$,  
we have $(1 - \xi_0) ^{{1}/{d}} \omega \ge u_1$ in $\Omega$ by Lemma \ref{Comparison Principle}, which immediately leads to \eqref{eq:3.20} for $k = 0$.

Assume that \eqref{eq:3.20} holds for all $k$ with $0\le k\le n-1$ where $n\ge 2$ is a certain integer.  
Then according to the perturbed subproblem \eqref{app sub} and the definition of $\lambda_{\rm MA}$ in \eqref{inf}, for every Borel set $E \subset \Omega$,
$$
\begin{array}{ll}
\displaystyle
\int_{E}\det D^2u_{n+1}\,\mathrm{d}\bm{x} 
= R(u_n) \int_{E} |u_n|^d \, \mathrm{d}\bm{x} 
  + \int_{E} \varepsilon_n \, \mathrm{d}\bm{x} 
  \\
\displaystyle
\ge (1 - \xi_n) R(u_n) \int_{E} |u_n|^d \, \mathrm{d}\bm{x} 
\ge 
(1 - \xi_n) \lambda_{\rm MA} \int_{E} |u_n|^d \, \mathrm{d}\bm{x} 
\\
\displaystyle
\ge  
\lambda_{\rm MA} \int_{E}\Big|\prod\limits_{i=0}^n (1 - \xi_i)^{{1}/{d}} \omega\Big|^d \, \mathrm{d}\bm{x}
=  
\int_{E} \det D^2\Big(  \prod\limits_{i=0}^n (1 - \xi_i) ^{{1}/{d}} \omega \Big)\, \mathrm{d}\bm{x} .
\end{array}
$$
Since $\left(\prod_{i=0}^n (1 - \xi_i) \right) ^{{1}/{d}}\omega = u_{n+1} = 0$ on $\partial\Omega$, Lemma \ref{Comparison Principle} yields
that \eqref{eq:3.20} also holds for $k=n$. 
Thus inequality \eqref{eq:3.20} holds for all $k \ge 0$, and the induction is complete.

Finally, we show that the infinite product $ P:=\prod_{k=0}^\infty (1-\xi_k) $ converges to a positive constant. Let $S:= \sum_{k = 0}^\infty - \ln (1-\xi_k)$. Since $\{\xi_{k}\}_{k \geq 0}$ is a  summable sequence of nonnegative real numbers in $[0,1)$, it follows that $\lim_{k\rightarrow\infty} \xi_k = 0$. 
Thus there exists an integer $K$ such that for all $k\geq K$, $\xi_k\leq 1/2$. Note that for $x\in[0,1/2]$, the following inequality holds
$$
-\ln(1-x)=\sum\limits_{m=1}^\infty\frac{x^m}{m}\le\sum\limits_{m=1}^\infty x^m=\frac{x}{1-x}\le 2x.
$$
Consequently, the series $S$ is bounded by
$$
S
= \sum\limits_{k=0}^{K-1}-\ln(1-\xi_k) + \sum_{k=K}^\infty-\ln(1-\xi_k)
\le \sum\limits_{k=0}^{K-1}-\ln(1-\xi_k) + 2\sum\limits_{k=K}^{\infty}\xi_k< \infty.
$$
Thus the infinite product $P = e^{-S}$ is a positive constant, and for any $k\geq 0$ one has $\prod_{i=0}^k(1-\xi_i)\geq P >0$.
One may take
$c_1:=(c_0P)^{1/d} > 0,$
and obtain the  bound
$$
c_1\leq \Big(c_0\prod_{i=0}^k (1-\xi_i)\Big)^{1/d} \quad\forall\,k\ge0.
$$
Then for any $k \ge 0$, since both $w$ and $u_{k+1}$ are convex functions vanishing on $\partial\Omega$, it follows from \eqref{eq:3.20} that
\[
\|u_{k+1}\|_{L^\infty(\Omega)} 
\ge \Big( c_0\prod_{i=0}^k (1 - \xi_i) \Big)^{1/d}  \lambda_{\rm MA}^{-1/d}\ge c_1\lambda_{\rm MA}^{-1/d}>0.
\]
This completes the proof. 
\end{proof}

Then the following result shows the monotonicity of the sequence $\{u_k\}_{k\ge0}$. The proof is given in Appendix \ref{appendixineq}.

\begin{proposition}
\label{inequality1}
Under the conditions of Proposition \ref{inftyinf},
one has  
\begin{equation}
\label{inequalitye1}
R(u_{k+1})\|u_{k+1}\|^d_{L^{d+1}(\Omega)} \le R(u_{k})\|u_{k}\|^d_{L^{d+1}(\Omega)} +\left(\mathcal{L}^d(\Omega)\right)^{\frac{d}{d+1}}\xi_k
\quad \forall\; k\geq 0.
\end{equation}
\end{proposition}

In addition, the regularity of the sequence $\{u_k\}_{k\ge0}$ is established in the following proposition.  
Its proof, similar to that of \cite[Proposition 3.1]{le2020convergence}, is given in Appendix \ref{appendixsmoothness}.

\begin{proposition}
\label{eventual smoothness1} 
 Under the conditions of Proposition \ref{inftyinf}, then one has $u_{k}\in  \mathcal{C}^{0, {1}/{d}}(\overline{\Omega}) $  for all $k\geq 1$, and 
the H\"older $\mathcal{C}^{0,{1}/{d}}$-norm of $\{u_{k}\}_{k\ge 1}$ is uniformly bounded by a constant $C\equiv C(d, \Omega, u_0,M)>0$. 
Furthermore, $u_{k}$ is strictly convex in $\Omega$ and $u_{k}\in \mathcal{C}^{2(k-1), {1}/{d}}(\Omega) $ for all $k\geq 2$. 
\end{proposition}

Moreover, the next result provides an upper bound for the sequence $\{R(u_k)\}_{k\ge1}$.
\begin{proposition}
\label{inftysup} 
Under the conditions of Proposition \ref{inftyinf}, 
there exists a positive constant $C'\equiv C'(d,\Omega,$ $u_0, M,\lambda_{\rm MA})$ 
such that $R(u_k) \le C'(d, \Omega,u_0, M,\lambda_{\rm MA})$ for all $k \ge 1$.
\end{proposition}
\begin{proof}
Proposition \ref{inequality1} implies that for all $k\ge1$
\begin{equation}
\begin{array}{ll}
\label{4.1}
R(u_k)
&\leq\dfrac{R(u_{k-1})\|u_{k-1}\|^d_{L^{d+1}(\Omega)}+\left(\mathcal{L}^d(\Omega)\right)^{\frac{d}{d+1}}\xi_{k-1}}{\|u_k\|^d_{L^{d+1}(\Omega)}} 
\leq\dfrac{R(u_{0})\|u_{0}\|^d_{L^{d+1}(\Omega)}+\left(\mathcal{L}^d(\Omega)\right)^{\frac{d}{d+1}}{\sum\limits_{i=0}^{k-1}  \xi_i}}{\|u_k\|^d_{L^{d+1}(\Omega)}} .
\end{array}
\end{equation}
According to Proposition \ref{eventual smoothness1},  
$u_k \in \mathcal{C}^{0,{1}/{d}}(\overline{\Omega})$ for all $k \ge 1$. 
Moreover, since $u_k$ vanishes on the boundary $\partial \Omega$, 
Lemma \ref{Comparability} implies
\[
\|u_k\|^d_{L^{d+1}(\Omega)}=\Big(\int_\Omega |u_k|^{d+1}\Big)^\frac{d}{d+1}\ge(d+1)^{-d}\left(\mathcal{L}^d(\Omega)\right)^\frac{d}{d+1}\|u_k\|^d_{L^{\infty}(\Omega)}.
\]  
Combining the above inequality with \eqref{4.1}, Proposition \ref{inftyinf} implies
\begin{equation*}
R(u_k)
\leq \dfrac{R(u_{0})\|u_{0}\|^d_{L^{d+1}(\Omega)}+\left(\mathcal{L}^d(\Omega)\right)^{\frac{d}{d+1}}M}{c_1^d \lambda_{\rm MA}^{-1}(d+1)^{-d}\left(\mathcal{L}^d(\Omega)\right)^\frac{d}{d+1}}=: C'(d, \Omega,u_0, M,\lambda_{\rm MA})>0.
\end{equation*}

This completes the proof. 
\end{proof}

Based on the above propositions, we are ready to establish the following main result regarding the convergence of Algorithm \ref{alg:inexact_AK} for the general case. 
\begin{theorem}
Under the conditions of Proposition \ref{inftyinf}, it holds that
\begin{equation*}
\lim \limits_{k\to \infty} R(u_k)= \lambda_{\rm MA}, 
\end{equation*}
with $ \lambda_{\rm MA}$ being defined by \eqref{inf}.
Moreover, the sequence $\{u_k\}_{k\ge0}$  converges uniformly on $\overline{\Omega}$ to a non-zero Monge-Amp{\`e}re  eigenfunction $u_{\infty}$.
\end{theorem}
\begin{proof}

Proposition \ref{eventual smoothness1} establishes that the sequence of functions $\{u_k\}_{k\geq 0}$ is uniformly bounded and equicontinuous. Hence, the Arzel\`a-Ascoli theorem guarantees the existence of a uniformly convergent subsequence. Without loss of generality, $\{u_{k(j)}\}_{j\ge1}$ is taken to be a subsequence that converges uniformly to a convex function $u_\infty \in \mathcal{C}(\overline{\Omega})$ satisfying $u_\infty = 0$ on $\partial \Omega$. 
Meanwhile, the shifted sequence $\{u_{k(j)+1}\}_{j\ge1}$ is itself a subsequence of the pre-compact sequence $\{u_k\}_{k\geq 0}$. 
Therefore, by passing to a further subsequence if necessary, we can ensure that $\{u_{k(j)+1}\}_{j\ge1}$ also converges uniformly to a convex function $\omega_\infty \in \mathcal{C}(\overline{\Omega})$ satisfying $\omega_\infty = 0$ on $\partial \Omega$. 
According to Proposition \ref{inftyinf}, $u_\infty \not\equiv 0$ and $\omega_\infty \not\equiv 0$ are obtained. 
Hence, $u_\infty \in \mathcal{K} $ and $\omega_\infty \in \mathcal{K}$.

Next, we show the convergence of the Rayleigh quotients associated with the subsequences 
$\{u_{k(j)}\}_{j\ge1}$ and $\{u_{k(j)+1}\}_{j\ge1}$.
By Proposition~\ref{eventual smoothness1}, there exists a constant $C_1(d,\Omega,u_0,M)$ such that $\sup_{\Omega}|u_k|\le C_1$ for all $k\ge1$. 
Proposition~\ref{inftysup} further gives $R(u_k)\le C'(d,\Omega,u_0,M,\lambda_{\rm MA})$ for all $k\ge1$. 
Together with the uniform bound $\|\varepsilon_k\|_{L^{\infty}(\Omega)}\le\xi_k\le \sum_{k=0}^\infty \xi_k=M < \infty$, we obtain 
\[
\det D^2u_{k+1}=R(u_k)|u_k|^d+\varepsilon_k \le C'C_1^d+M=:b.
\]
Hence, $\mathcal{M}u_{k+1}\le b\,\mathcal{L}^d$ for all $k\ge1$, where $b>0$ is independent of $k$. 
Therefore, from Lemma \ref{Continuity property}, 
combined with the definition of $R(u)$ in \eqref{rayli}, one has
$$
\begin{cases}
\displaystyle
\lim_{j \to \infty} R(u_{k(j)}) \, \|u_{k(j)}\|_{L^{d+1}(\Omega)}^{d+1}  = R(u_\infty) \, \|u_\infty\|_{L^{d+1}(\Omega)}^{d+1},
\\[1mm]
\displaystyle
\lim_{j \to \infty} R(u_{k(j)+1}) \, \|u_{k(j)+1}\|_{L^{d+1}(\Omega)}^{d+1}  = R(\omega_\infty) \, \|\omega_\infty\|_{L^{d+1}(\Omega)}^{d+1}.
\end{cases}
$$
Since $u_\infty \not\equiv 0$ and $\omega_\infty \not\equiv 0$, it follows that 
\[ \lim_{j \to \infty} \|u_{k(j)}\|_{L^{d+1}(\Omega)}^{d+1} = \|u_\infty\|_{L^{d+1}(\Omega)}^{d+1} \neq 0\quad \mbox{and}\quad 
\lim_{j \to \infty} \|u_{k(j)+1}\|_{L^{d+1}(\Omega)}^{d+1} = \|\omega_\infty\|_{L^{d+1}(\Omega)}^{d+1} \neq 0. \]
Thus 
\begin{equation} \label{twoR}
\lim_{j \to \infty} R(u_{k(j)}) = R(u_\infty) \quad \mbox{and}\quad  \lim_{j \to \infty} R(u_{k(j)+1}) = R(\omega_\infty).
\end{equation}
In addition, since $\{\xi_k\}_{k\geq 0}$ is a summable sequence of nonnegative real numbers, we have 
\begin{equation} \label{thm3.6}
\lim_{k\to\infty} \| \varepsilon_k \|_{L^{\infty}(\Omega)} \le \lim_{k\to\infty} \xi_k =0.
\end{equation}
Then, taking the limit with $j\to \infty$ in 
\[ \det D^2u_{k(j)+1}=R(u_{k(j)})|u_{k(j)}|^d+\varepsilon_{k(j)}, \]
combining with \eqref{twoR}, \eqref{thm3.6} and the weak convergence of the Monge-Amp{\`e}re measure (c.f. \cite[Corollary 2.12]{figalli2017monge}), we obtain
\begin{equation} \label{equ}
\det D^2\omega_{\infty} =R(u_\infty)|u_{\infty}|^d
\end{equation}
in the sense of Aleksandrov.
Furthermore, from Proposition \ref{inequality1}, one has
\begin{equation}
\label{thm3.7}
\begin{array}{ll}
R(u_{k(j+1)})\|u_{k(j+1)}\|_{L^{d+1}(\Omega)}^d &\le R(u_{k(j)+1})\|u_{k(j)+1}\|_{L^{d+1}(\Omega)}^d + \left(\mathcal{L}^d(\Omega)\right)^{\frac{d}{d+1}}\sum\limits_{i=k(j)+1}^{k(j+1)-1}\xi_i
\\[1mm]
&\le R(u_{k(j)})\|u_{k(j)}\|_{L^{d+1}(\Omega)}^d + \left(\mathcal{L}^d(\Omega)\right)^{\frac{d}{d+1}}\sum\limits_{i=k(j)}^{k(j+1)-1}\xi_i.
\end{array}
\end{equation}
Since $\{\xi_k\}_{k\geq 0}$ is summable,  
$ \sum_{i=k(j)+1}^{k(j+1)-1}\xi_i \to 0 \quad \text{and} \quad 
\sum_{i=k(j)}^{k(j+1)-1}\xi_i \to 0 \quad \text{as } j \to \infty$.
Therefore, letting $j \to \infty$ in \eqref{thm3.7}, one can get
\begin{equation} \label{thm3.8}
R(\omega_\infty)\|\omega_\infty\|_{L^{d+1}(\Omega)}^d = R(u_\infty)\|u_\infty\|_{L^{d+1}(\Omega)}^d.
\end{equation}
Thus, multiplying both sides of \eqref{equ} by $|\omega_\infty|$ and integrating over $\Omega$, the H\"older inequality implies that
\begin{equation*}
\begin{split}
R(\omega_\infty)\|\omega_\infty\|_{L^{d+1}(\Omega)}^{d+1} 
&= \int_{\Omega} |\omega_{\infty}| \det D^2 \omega_{\infty}\, \mathrm{d}\bm{x}  
= R(u_\infty)\int_\Omega |u_\infty|^d|\omega_\infty|\, \mathrm{d}\bm{x} 
\\
&\le R(u_\infty)\|u_\infty\|_{L^{d+1}(\Omega)}^d \|\omega_\infty\|_{L^{d+1}(\Omega)} 
= R(\omega_\infty)\|\omega_\infty\|_{L^{d+1}(\Omega)}^{d+1}.
\end{split}
\end{equation*}
Consequently, the above inequality holds as an equality, implying that the two functions $\omega_\infty$ and $u_\infty$ are proportional, 
i.e., there exists a constant $c>0$ such that $\omega_\infty = c\,u_\infty$. 
Substituting this relation into the definition \eqref{rayli} of $R(\cdot)$ yields $R(\omega_\infty) = R(u_\infty)$. 
By combining this identity with \eqref{thm3.8}, the constant is determined as $c=1$.
Therefore, $u_\infty = \omega_\infty$. 
Immediately, from \eqref{equ} one can see that
\[ 
\det D^2 u_\infty = R(u_\infty)\, |u_\infty|^d. 
\]
According to Lemma \ref{existence and uniqueness}, the function $u_\infty$ is identified as a nontrivial Monge-Amp{\`e}re eigenfunction, and $R(u_\infty)=\lambda_{\rm MA}$ represents the corresponding MAE.

Finally, we show that the entire sequence $\{u_k\}$ converges uniformly to the same eigenfunction $u_\infty$. 
Let $\{u_{k_1(j)}\}_{j\ge1}$ and $\{u_{k_2(j)}\}_{j\ge1}$ be two uniformly convergent subsequences of $\{u_k\}_{k\geq 0}$ with limits $u_{1,\infty}$ and $u_{2,\infty}$, respectively. 
Then, as shown above, both limits are nontrivial Monge-Amp{\`e}re eigenfunctions, and therefore
$$
R(u_{1,\infty}) = R(u_{2,\infty}) = \lambda_{\rm MA}. 
$$
Two new subsequences $\{u_{l_1(j)}\}_{j\geq 2}\subset\{u_{k_1(j)}\}_{j\geq 1}$ and $\{u_{l_2(j)}\}_{j\geq 1}\subset\{u_{k_2(j)}\}_{j\geq 1}$ are constructed by defining 
$$
\begin{cases}
l_1(j) := \min\{\,k_1(m) \mid k_1(m) > l_2(j-1)\,\}& j \ge 2,\\ 
l_2(j) := \min\{\,k_2(m) \mid k_2(m) > l_1(j)\,\}& j \ge 1.
\end{cases}
$$
Since $\{u_{l_1(j)}\}_{j\ge1}$ and $\{u_{l_2(j)}\}_{j\ge1}$ are subsequences of the convergent sequences $\{u_{k_1(j)}\}_{j\ge1}$ and $\{u_{k_2(j)}\}_{j\ge1}$, they also converge to $u_{1,\infty}$ and $u_{2,\infty}$, respectively. 
In addition, their indices satisfy $l_1(j) < l_2(j) < l_1(j+1)$ for all $j \ge 1$.
Therefore, Proposition \ref{inequality1} implies
$$ R(u_{l_2(j)})\|u_{l_2(j)}\|_{L^{d+1}(\Omega)}^d \le R(u_{l_1(j)})\|u_{l_1(j)}\|_{L^{d+1}(\Omega)}^d + \left(\mathcal{L}^d(\Omega)\right)^{\frac{d}{d+1}}\sum_{i=l_1(j)}^{l_2(j)-1}\xi_i, 
$$
and 
$$ R(u_{l_1(j+1)})\|u_{l_1(j+1)}\|_{L^{d+1}(\Omega)}^d \le R(u_{l_2(j)})\|u_{l_2(j)}\|_{L^{d+1}(\Omega)}^d + \left(\mathcal{L}^d(\Omega)\right)^{\frac{d}{d+1}}\sum_{i=l_2(j)}^{l_1(j+1)-1}\xi_i. 
$$
Since $\{\xi_k\}_{k\geq 0}$ is summable, taking the limit along with $j \to \infty$, it follows that 
$$ R(u_{1,\infty})\|u_{1,\infty}\|_{L^{d+1}(\Omega)}^d = R(u_{2,\infty})\|u_{2,\infty}\|_{L^{d+1}(\Omega)}^d.
$$
Note that $R(u_{1,\infty}) = R(u_{2,\infty})$, then we have
\begin{equation} \label{thm3.10}
\|u_{1,\infty}\|_{L^{d+1}(\Omega)} = \|u_{2,\infty}\|_{L^{d+1}(\Omega)}.
\end{equation}
Since both $u_{1,\infty}$ and $u_{2,\infty}$ are eigenfunctions, Lemma \ref{existence and uniqueness} implies that they must be multiples of each other, and \eqref{thm3.10} shows that $u_{1,\infty} = u_{2,\infty}$. 
Due to the fact that the subsequences $\{u_{k_1(j)}\}_{j\ge1}$ and $\{u_{k_2(j)}\}_{j\ge1}$ of $\{u_k\}_{k\ge1}$ are arbitrarily chosen, the entire sequence $\{u_k\}_{k\ge0}$ converges uniformly on $\overline{\Omega}$ to the same nontrivial Monge-Amp{\`e}re eigenfunction $u_\infty$. 
The proof is complete. 
\end{proof}

\subsection{Convergence analysis for nonnegative error functions}
\label{analysis_special}
In this subsection, we turn to the convergence analysis of Algorithm \ref{alg:inexact_AK} in the case where the perturbation functions 
$\{\varepsilon_k(\bm x)\}_{k\ge0}$ are known to be nonnegative, under a weaker condition than Assumption \ref{ass_blanket1}. 
The structure of the proof mainly follows that of \cite{le2020convergence}. 
The main difference is that \cite{le2020convergence} used a positive constant function as the error term, while we use nonnegative error functions. 
The following assumption comes from \cite[Theorem 1.4]{le2020convergence}.

\begin{assumption}
\label{ass_blanket2}
$\Omega$ is a bounded convex open domain in $\mathbb{R}^d$, 
$u_0 \in \mathcal{C}(\Omega)$ is a non-zero convex function on $\Omega$ 
with $0 < R(u_0) < \infty$.
\end{assumption}

Due to the non-negativity of the perturbation functions 
$\{\varepsilon_k(\bm x)\}_{k\ge0}$, the corresponding requirements in Algorithm~\ref{alg:inexact_AK} turn to the single condition
\[
\|\varepsilon_k\|_{L^\infty(\Omega)} \le \xi_k, \quad 
\]
and the sequence $\{\xi_k\}_{k\ge0}$ is merely required to be summable and nonnegative. 
We  present several properties for the sequence $\{u_k\}_{k\ge0}$ generated by Algorithm \ref{alg:inexact_AK} with $\varepsilon_k(\bm x)\ge0$ for all $k\ge 0$. Firstly, the following result demonstrates the monotonicity property of the sequence $\{u_k\}_{k\ge0}$ generated by Algorithm \ref{alg:inexact_AK} with $\varepsilon_k(\bm x)\ge0$ for all $k\ge 0$.
The proof is highly similar to that of Proposition \ref{inequality1}, so we omitted it here. 

\begin{proposition}
\label{inequality2}
Under Assumption \ref{ass_blanket2}, suppose that the sequence $\{u_k\}_{k\ge0}$ is generated by Algorithm \ref{alg:inexact_AK} with $\varepsilon_k(\bm x)\ge0$ for all $k\ge 0$,
one has  
\begin{equation}
\label{inequalitye2}
R(u_{k+1})\|u_{k+1}\|^d_{L^{d+1}(\Omega)} \le R(u_{k})\|u_{k}\|^d_{L^{d+1}(\Omega)} +\left(\mathcal{L}^d(\Omega)\right)^{\frac{d}{d+1}}\xi_k
\quad \forall\; k\geq 0.
\end{equation}
\end{proposition}
 
In addition, the following proposition establishes the regularity of the sequence $\{u_k\}_{k\ge0}$ generated by Algorithm \ref{alg:inexact_AK} with $\varepsilon_k(\bm x)\ge0$ for all $k\ge 0$.
We also omit its proof due to its similarity to Proposition \ref{eventual smoothness1}. 

\begin{proposition}
\label{eventual smoothness2} 
Under the conditions of Proposition \ref{inequality2}, then one has $u_{k}\in  \mathcal{C}^{0, {1}/{d}}(\overline{\Omega}) $  for all $k\geq 1$, and 
the H\"older $\mathcal{C}^{0,{1}/{d}}$-norm of $\{u_{k}\}_{k\ge 1}$ is uniformly bounded by a constant $C''(d, \Omega, u_0,M')>0$, where $M' := \sum_{k=0}^\infty \xi_k < \infty$. 
Furthermore, $u_{k}$ is strictly convex in $\Omega$ and $u_{k}\in \mathcal{C}^{2(k-1), {1}/{d}}(\Omega) $ for all $k\geq 2$. 
\end{proposition}

Now, we are ready to establish the following result regarding the convergence of Algorithm \ref{alg:inexact_AK} with $\varepsilon_k(\bm x)\ge0$ for all $k\ge 0$. 

\begin{theorem}
Under the conditions of Proposition \ref{inequality2}, it holds
\begin{equation}
\label{Ru_con}
\lim \limits_{k\to \infty} R(u_k)= \lambda_{\rm MA}, 
\end{equation}
with $ \lambda_{\rm MA}$ being defined by \eqref{inf}.
Moreover, the sequence $\{u_k\}_{k\ge0}$ converges uniformly on $\overline{\Omega}$ to a non-zero Monge-Amp{\`e}re  eigenfunction $u_{\infty}$. 
\end{theorem}

\begin{proof}
We divide the proof into three steps.

\paragraph{Step 1: The whole sequence $\{R(u_{k})\}_{k\geq 0}$ converges to $\lambda_{\rm MA}$.}~

\smallskip 
Under Assumption~\ref{ass_blanket2}, we have $0<R(u_0)<\infty$.
According to Proposition~\ref{eventual smoothness2}, $\sup_{\Omega}|u_{k}|\le C''(d, \Omega, u_0,M')$ for all $k\ge1$. 
Moreover, recalling the perturbed subproblem \eqref{app sub}, if \(R(u_k)\) is finite for some \(k \ge 0\), then the Monge-Amp{\`e}re measure of $\mathcal{M}u_{k+1}$ is finite. Consequently, both $\int_{\Omega} |u_{k+1}|\det D^2 u_{k+1}\,\mathrm{d}\bm{x}$ and 
$\int_{\Omega} |u_{k+1}|^{d+1}\,\mathrm{d}\bm{x}$ are finite. 
By induction, we conclude that 
$R(u_{k})$ is finite for all $k\geq0$.
Furthermore, from Lemma~\ref{existence and uniqueness}, there exists a non-zero eigenfunction 
$w \in \mathcal{C}^{0,\beta}(\overline{\Omega})\cap \mathcal{C}^{\infty}({\Omega})$ for some $\beta \in (0,1)$, 
and thus there exists a constant 
$M_w := \sup_{\Omega} |w| < \infty$. 
Consequently,
\[
\int_{\Omega} (\det D^2u_{k+1})^{1/d} |w|^{\,d-1}\,\mathrm{d}\bm{x}
\le \mathcal{L}^d(\Omega)\, (R(u_k)(C'')^d+\xi_k)^{1/d} M_w^{\,d-1} < \infty,
\qquad \forall\; k \ge 0.
\]
By Proposition \ref{eventual smoothness2}, $u_{k+1}\in \mathcal{C}^{2k, {1}/{d}}(\Omega)$  for $k \ge 3$.
Therefore, the reverse Aleksandrov estimate 
\cite[Proposition 2.5]{le2020convergence} 
can be applied to the perturbed subproblem \eqref{app sub} yields for $k\geq 3$, 
\begin{equation}
\label{monotonicity formula}
\begin{array}{ll}
&\int_\Omega  |u_{k+1}| |w|^d\ \mathrm{d} \bm{x}
\\[1mm]
&\geq {\lambda_{\rm MA}^{-1/d}} \int_\Omega (\det D^2 u_{k+1})^{1/d} |w|^d\ \mathrm{d} \bm{x}
={\lambda_{\rm MA}^{-1/d}} \int_\Omega (R(u_k)|u_k|^d+ \varepsilon_k)^{1/d} |w|^d\ \mathrm{d} \bm{x}
\\[2mm]
&\geq \dfrac{(R(u_k))^{1/d}}{\lambda_{\rm MA}^{1/d}} \int_\Omega |u_k| |w|^d\ \mathrm{d} \bm{x} 
= \int_{\Omega} |u_k| |w|^d\ \mathrm{d} \bm{x} + \dfrac{(R(u_k))^{1/d}-\lambda_{\rm MA}^{1/d}}{\lambda_{\rm MA}^{1/d}} \int_{\Omega} |u_k| |w|^d\ \mathrm{d} \bm{x},   
\end{array}
\end{equation}
where the last inequality holds because $\varepsilon_k(\bm x)\geq0$ for all $k\ge 0$.
Recalling the definition of $\lambda_{\rm MA}$ in \eqref{inf}, the inequality \eqref{monotonicity formula} shows that the sequence $\{\int_\Omega |u_k||w|^d\ \mathrm{d} \bm{x}\}_{k\geq 3}$ is non-decreasing for $k\geq 3$. 
Recall that Assumption \ref{ass_blanket2} ensures $u_k \not\equiv 0$ for all $k \ge 0$.
Combining this with \eqref{monotonicity formula} and the fact that $\omega \not\equiv 0$, we obtain the uniform lower bound 
\begin{equation}
\label{Linf} 
\|u_k\|_{L^\infty(\Omega)} \ge \frac{\int_\Omega |u_k||w|^d\ \mathrm{d} \bm{x}}{\int_\Omega |w|^d\ \mathrm{d} \bm{x}} 
\ge \frac{\int_\Omega |u_3||w|^d\ \mathrm{d} \bm{x}}{\int_\Omega |w|^d\ \mathrm{d} \bm{x}}>0\quad \forall\; k\ge3.
\end{equation}
Furthermore, Proposition \ref{inequality2} ensures that $\|u_k\|_{L^{d+1}(\Omega)}$ is bounded, and thus $\{\int_\Omega |u_k||w|^d\ \mathrm{d} \bm{x}\}_{k\geq 3}$ is bounded above. Hence, $\{\int_\Omega |u_k||w|^d\ \mathrm{d} \bm{x}\}_{k\geq 3}$ converges to a finite limit $L$.
Inequality \eqref{monotonicity formula} yields that $L\geq\int_\Omega |u_3||w|^d\ \mathrm{d} \bm{x} >0$, and for $k\geq 3$, 
\begin{equation}
\label{R}
0\le (R(u_k))^{1/d}-\lambda_{\rm MA}^{1/d} \leq \lambda_{\rm MA}^{1/d}
\frac{\int_{\Omega}(|u_{k+1}|- |u_{k}|) |w|^d\ \mathrm{d} \bm{x}} {\int_{\Omega} |u_k| |w|^d\ \mathrm{d} \bm{x}}
\leq\lambda_{\rm MA}^{1/d}
\frac{\int_{\Omega}(|u_{k+1}|- |u_{k}|) |w|^d\ \mathrm{d} \bm{x}} {\int_{\Omega} |u_3| |w|^d\ \mathrm{d} \bm{x}}.
\end{equation}
Consequently, taking the limit with $k\to\infty$ in \eqref{R} leads to \eqref{Ru_con}.

\paragraph{Step 2: A subsequence $\{u_{k(j)}\}_{j\geq 1} $ and the corresponding shifted subsequence $\{u_{k(j)+1}\}_{j\geq 1} $ converge to the same nontrivial eigenfunction $u_{\infty}$.
}~

\smallskip 
By Proposition \ref{eventual smoothness2}, the sequence $\{u_k\}_{k\geq 0}$ is uniformly bounded and equicontinuous. 
The Arzel{\`a}– Ascoli Theorem and \eqref{Linf} imply that there exists a subsequence $\{u_{k(j)}\}_{j\geq 1}$ satisfying $u_{k(j)}\to u_\infty$ as $j\to\infty$ uniformly on $\overline{\Omega}$ with $u_\infty\not\equiv0$ being convex, and $u_\infty=0$ on $\partial\Omega$. 
Similarly, we have $u_{k(j)+1}\to\omega_\infty\not\equiv0$.  
Since $\{\xi_k\}_{k\geq 0}$ is a summable sequence of nonnegative real numbers, we have
\[
\lim_{k \to \infty} \|\varepsilon_k\|_{L^{\infty}(\Omega)} \le 
\lim_{k \to \infty} \xi_k = 0,
\]
which implies that $\varepsilon_k \to 0$.
Therefore, taking the limit as $j\to\infty$ in
\[
\det D^2 u_{k(j)+1} = R(u_{k(j)})|u_{k(j)}|^d+\varepsilon_{k(j)},
\]
and using  the weak convergence of the Monge–Ampère measure (c.f. \cite[Corollary 2.12]{figalli2017monge}), combined with $R(u_{k(j)})\to\lambda_{\rm MA}$, we obtain
\begin{align}\label{equ-1}
\det D^2\omega_\infty = \lambda_{\rm MA}|u_\infty|^d
\end{align}
in the Aleksandrov sense. 
By the monotonicity inequality \eqref{inequalitye2} in Proposition \ref{inequality2}, we have
$$
\begin{array}{ll}
R(u_{k(j+1)})\|u_{k(j+1)}\|^d_{L^{d+1}(\Omega)}
&\le R(u_{k(j)+1})\|u_{k(j)+1}\|^d_{L^{d+1}(\Omega)}  + \left(\mathcal{L}^d(\Omega)\right)^{\frac{d}{d+1}}\sum\limits_{i=k(j)+1}^{k(j+1)-1}\xi_i
\\[1mm]
&\le R(u_{k(j)})\|u_{k(j)}\|^d_{L^{d+1}(\Omega)}+ \left(\mathcal{L}^d(\Omega)\right)^{\frac{d}{d+1}}\sum\limits_{i=k(j)}^{k(j+1)-1}\xi_i.
\end{array}
$$
Then, by $\sum_{i=k(j)}^{k(j+1)-1}\xi_i\to 0$ as  $j \to \infty$, it concludes that
\begin{equation*}
 \|\omega_{\infty}\|_{L^{d+1}(\Omega)}=\|u_\infty\|_{L^{d+1}(\Omega)}.
\end{equation*}
In addition, multiplying \eqref{equ-1} by $|\omega^{\infty}|$ and then integrating over $\Omega$ result in
\begin{align}
R(\omega_{\infty})\|\omega_{\infty}\|^{d+1}_{L^{d+1}(\Omega)} 
&
=\int_{\Omega}
|\omega_{\infty}| \det D^2 \omega_{\infty}\ \mathrm{d} \bm{x} 
=\lambda_{\rm MA} \int_{\Omega}|u_{\infty}|^d|\omega_{\infty}|  \ \mathrm{d} \bm{x} \notag\\
&\le  \lambda_{\rm MA} \|u_{\infty}\|^{d}_{L^{d+1}(\Omega)}\|\omega_{\infty}\|_{L^{d+1}(\Omega)}=\lambda_{\rm MA}\|\omega_{\infty}\|^{d+1}_{L^{d+1}(\Omega)}\notag.
\end{align}
Since $R(w_{\infty})\geq \lambda_{\rm MA}$,  we have $R(w_{\infty})= \lambda_{\rm MA}$, and the inequality above holds as an equality. 
This gives $u_{\infty}= c w_{\infty}$ for some constant $c>0$. 
Thus from \eqref{equ-1}, we have $\det D^2 w_{\infty} = c^d\lambda_{\rm MA}|w_\infty|^d$. 
According to Lemma \ref{existence and uniqueness}, $c=1$ and $w_\infty= u_{\infty}$ is a Monge-Amp\`ere eigenfunction of $\Omega$.

\paragraph{Step 3: Convergence of the full sequence.}  
~

\smallskip 
Taking the limit as $k\to \infty$ in \eqref{monotonicity formula}, we have
$$\int_\Omega |u_{\infty}| |w|^d\ \mathrm{d} \bm{x} =\lim_{k\rightarrow \infty} \int_\Omega |u_k||w|^d\ \mathrm{d} \bm{x}=L.$$
With this property and the uniqueness up to positive multiplicative constants of the Monge-Amp\`ere eigenfunctions of $\Omega$, we conclude that the limit $u_{\infty}$ does not depend on the subsequence $\{u_{k(j)}\}_{j\geq 1}$. This shows that the whole sequence $\{u_k\}_{k\geq 0}$ converges to a non-zero Monge-Amp\`ere eigenfunction $u_{\infty}$,
which completes the proof. 
\end{proof}

\section{Fixed-point approaches for solving subproblems}
\label{sec:fixed}
To efficiently solve the subproblem \eqref{app sub} arising in Algorithm \ref{alg:inexact_AK}, we adopt the fixed-point iteration strategy proposed in \cite{benamou2010two,froese2011convergent}, and restrict to the convex bounded open domain $\Omega$ in $\mathbb{R}^2$ and $\mathbb{R}^3$. 
We also prove the convergence of this fixed-point method under the $\mathcal{C}^{2,\alpha}$ boundary condition in this section. 
Consider a more general form of the MA equation with a zero Dirichlet boundary condition
\begin{equation}
\label{3.6}
u \mbox{ is convex\  and }\ 
\begin{cases}
\det D^2u =f &\mbox{ in } \Omega, 
\\
u=0  & \mbox{ on }  \partial \Omega, 
\end{cases}
\end{equation}
where $f$ is a given function with $f>0$ in $\Omega$ $\subset \mathbb{R}^d$.
The subproblem \eqref{app sub} is a special case of \eqref{3.6}. 
Let $\lambda_i[D^2u]$ be the $i$-th eigenvalue of the Hessian $D^2 u$. 
Note that the Laplacian can be expressed in terms of the eigenvalues of the Hessian, i.e., 
$$
\Delta u = \sum_{i=1}^d \lambda_i[D^2u].
$$ 
Raising this identity to the $d$-th power and expanding gives the sum of all products of $d$ eigenvalues. 
In particular, introducing
$$
(\Delta u)^d = d! \prod_{i=1}^d \lambda_i[D^2u] + P_d(\lambda_1[D^2u], \ldots, \lambda_d[D^2u]),
$$
where $P_d(\lambda_1, \ldots, \lambda_d)$ denotes the remaining $d$-homogeneous polynomial terms. 
Combining it with the MA equation \eqref{3.6} and noticing $\prod_{i=1}^d \lambda_i[D^2u]=\det D^2u=f$, one has the following Poisson equation (\cite[Section 4.2]{froese2011convergent})
$$
\Delta u = (d! f + P_d(\lambda_1[D^2u], \ldots, \lambda_d[D^2u]))^{1/d}.
$$
Define the operator 
\begin{equation*}
    {Q_d}(u): =\Delta^{-1}(d!f+P_d(\lambda_1[D^2u],\dots,\lambda_d[D^2u]))^{1/d},
\end{equation*}
where $\Delta^{-1}$ denotes the Poisson solution operator with zero Dirichlet boundary condition. 
Note that a solution $u$ to the MA equation \eqref{3.6} is a fixed point of the operator ${Q_d}$, i.e., $u={Q_d}(u)$.
This formulation naturally leads to a fixed-point iteration scheme 
\[u_{k+1}={Q_d}(u_k).\]
The inexact AKI method with a fixed-point iteration (inexact AKI-FP) to solve the MAE problem \eqref{eq_MAeig} is given in Algorithm \ref{alg:fix}.

\begin{algorithm}
 \caption{A fixed-point-based inexact AKI method for MAE problems}
 \label{alg:fix} 
 \KwIn{An initial non-zero convex function $u_0 $ following Assumption \ref{ass_blanket1}, a summable sequence $\{\xi_{k}\}_{k \geq 0}$ of nonnegative real numbers satisfying $\xi_{k} \in [0,1)$ for all $k\ge0$.} 
 \KwOut{A sequence of real numbers $\{R(u_k)\}_{k\geq 0}$ and a sequence of functions $\{u_k\}_{k\geq 0}$.}
 \For{$k = 0, 1, \dots$}{
     calculate $R(u_k)$ according to \eqref{rayli};\\
   set $\bar{u}_0 = u_k$;\\
    \For{$n = 0, 1, \dots$}{
         compute the solution $\bar{u}_{n+1}$ of the following Poisson equation
        \begin{equation*}
        \begin{cases}
            \Delta \bar{u}_{n+1} = \Big( P_d(\lambda_1[D^2\bar{u}_{n}],\ldots,\lambda_d[D^2\bar{u}_{n}])  + d!\Big( R(u_k) |u_k|^d  \Big)\Big)^{1/d} & \text{in } \Omega, \\
            \bar{u}_{n+1} = 0 & \text{on } \partial \Omega,
        \end{cases}
        \end{equation*}
 if
$\begin{cases}
 \|\det D^2\bar{u}_{n+1}- R(u_k) |u_k|^d \|_{L^{\infty}(\Omega)} \le \xi_k,\\
\det D^2\bar{u}_{n+1}- R(u_k) |u_k|^d  \geq  - \xi_k R(u_k)|u_k|^d,
\end{cases}$
\textbf{then} break;\\
    }
     set $u_{k+1} := \bar{u}_{n+1}$.
 }
\end{algorithm}

Under Assumption \ref{ass_blanket1} and convergence conditions of the fixed-point method for the subproblem \eqref{app sub}, Algorithm \ref{alg:fix} produces a sequence $\{R(u_k)\}_{k\geq 0}$ that converges to the Monge-Amp\`ere eigenvalue $\lambda_{\rm MA}$ of $\Omega$. Moreover, the  sequence $\{u_k\}_{k\ge0}$ generated by Algorithm \ref{alg:fix} converges uniformly on $\overline{\Omega}$ to a non-zero Monge-Amp{\`e}re  eigenfunction $u_{\infty}$ of $\Omega$. In what follows, proofs of the convergence for the subproblems are provided for 2D and 3D cases, respectively.

\subsection{A fixed-point approach for \texorpdfstring{$\Omega\subset\mathbb{R}^2$}{Omega subset R2}}
In the 2D case with $\bm x=(x,y)^T\in \Omega$, it holds that
$$
\prod_{i=1}^2 \lambda_i[D^2u]=\det D^2u = u_{xx}u_{yy} - u_{xy}^2,\quad
(\Delta u)^2 = u_{xx}^2 + u_{yy}^2 + 2u_{xx}u_{yy}.
$$
Expanding $(\Delta u)^2$ and isolating the determinant term yield
\begin{equation*}
P_2(\lambda_1[D^2u],\lambda_2[D^2u]) = (\Delta u)^2 - 2!\det D^2u = u_{xx}^2 + u_{yy}^2 + 2u_{xy}^2.   
\end{equation*}
Since $u$ is a convex solution and $\Delta u \geq 0$, we take the positive square root and define an operator ${Q_2}:=Q_2(u)$ as
\begin{equation}
\label{eq:Q2}
    {Q_2}(u): =\Delta^{-1}(P_2(\lambda_1[D^2u],\lambda_2[D^2u])+2f)^{1/2}.
\end{equation}
As observed in \cite[Lemma 2.1]{benamou2010two}, a solution $u$ to the MA equation \eqref{3.6} is a fixed point of the operator ${Q_2}$, i.e., $u={Q_2}(u)$.
This formulation naturally leads to a fixed-point iteration scheme 
\[u_{k+1}={Q_2}(u_k).\]
More specifically, for obtained $u_k$, $u_{k+1}$ is obtained by solving the linear Poisson equation:
\begin{equation*}
\begin{cases}
\Delta u_{k+1} =\Bigl({P_2(\lambda_1[D^2u_k],\lambda_2[D^2u_k])} +2f \Bigl)^{1/2}
&  \mbox{ in } \Omega,\\
  u_{k+1} =0 & \mbox{ on } \partial \Omega.
\end{cases}
\end{equation*}
This iterative method transforms the task of solving the nonlinear MA equation into a sequence of solutions for the above linear Poisson equation, which is often computationally advantageous. 
The following theorem establishes the convergence of this fixed-point iteration under specific conditions.
\begin{theorem}
Let $\Omega\subset\mathbb{R}^2$ be a bounded open domain with $\mathcal{C}^{2,\alpha}$ boundary for some $\alpha\in(0,1)$, $f \in  \mathcal{C}^{0,\alpha}(\overline\Omega)$ satisfying $f \ge \tau >0$ on $\overline\Omega$, and $ \| f\|_{\mathcal{C}^{0,\alpha}(\overline\Omega)}\le M_f$. 
Then there exists a closed ball  
$$B_\rho := \{ u\in \mathcal{C}^{2,\alpha}(\overline\Omega) \mid  \|u\|_{\mathcal{C}^{2,\alpha}(\overline\Omega)} \leq \rho \}
\quad
\mbox{with}\quad \rho> 0,
$$ 
such that for any initial function $ u_0 \in B_\rho $, the sequence $\{u_k\}_{k\geq 0}$ generated by the fixed-point iteration $ u_{k
+1} = {Q_2}(u_{k})$ converges in the $\mathcal{C}^{2,\alpha}$-norm to the unique solution $u$ of the MA equation \eqref{3.6} within the ball $B_{\rho}$.
\end{theorem}
\begin{proof}
Based on the Banach fixed-point theorem  \cite[Theorem 5.7]{brezis2011functional}, the proof is divided into the following three steps. 
The first step is to analyze the contraction property of the operator ${Q_2}$.
Then the second step is to analyze the mapping property of the operator ${Q_2}$.
Finally, the third step is to verify that the sequence $\{u_{k}\}_{k\ge0}$ converges to the solution of the MA equation \eqref{3.6}.

\paragraph{Step 1: Contraction property of the operator ${Q_2}$.}~
 
\smallskip
Choose the parameter $\rho>0$ such that
\begin{equation}   
\label{delta}
\begin{array}{ll}
\delta := \max\left\{C\Big(\dfrac{5\rho}{\sqrt{2\tau}}+\dfrac{3M_f}{\rho\sqrt{2\tau}}\Big), C\Big(\dfrac{9\rho}{\sqrt{2\tau}}+\dfrac{4\rho^3+4\rho M_f}{(2\tau)^{{3}/{2}}}\Big)\right\} <1.
\end{array}
\end{equation}
To simplify the notation, for each $u\in \mathcal{C}^{2,\alpha}(\overline\Omega)$, we denote 
$$g[u] :=\sqrt{u_{xx}^2+ u_{yy}^2+2u_{xy}^2 +2f}.$$ 
Let $u,v \in B_\rho$, $\omega(t):=v+t(u-v)$ and $h:=u-v$, the difference $g[u]- g[v]$ can be analyzed by the fact that  
\begin{align*}
g[u]-g[v]=\int_0^1  \frac{\mathrm{d}}{\mathrm{d}t}g[\omega(t)] \ \mathrm{d} t=\int_0^1 \left\{\frac{\omega_{xx}(t)}{g[\omega(t)]}h_{xx}+\frac{\omega_{yy}(t)}{g[\omega(t)]}h_{yy}+2\frac{\omega_{xy}(t)}{g[\omega(t)]}h_{xy}\right\} \mathrm{d} t.
\end{align*}
Recall that the $\mathcal{C}^{0,\alpha}$-norm of the difference $g[u] - g[v]$ is given by
\begin{equation}
\label{eq: c0a-norm}
\left\|  g[u]-g[v]\right\|_{\mathcal{C}^{0,\alpha}(\overline\Omega)}=\left\|  g[u]-g[v]\right\|_{\mathcal{C}^{0}(\overline\Omega)}+\left|  g[u]-g[v]\right|_{0,\alpha;\overline\Omega}.
\end{equation}
On the one hand, for the first term on the right-hand side of \eqref{eq: c0a-norm}, it holds 
$$
\begin{array}{ll}
&\left\|  g[u]-g[v]\right\|_{\mathcal{C}^{0}(\overline\Omega)}
\\[3mm]
&\qquad=\left\|\displaystyle\int_0^1  
\left\{ \dfrac{\omega_{xx}(t)}{g[\omega(t)]}h_{xx}+\dfrac{\omega_{yy}(t)}{g[\omega(t)]}h_{yy}+2\dfrac{\omega_{xy}(t)}{g[\omega(t)]}h_{xy}\right\} 
\mathrm{d} t\right\|_{\mathcal{C}^{0}(\overline\Omega)}
\\[4mm]
&\qquad\le \sup\limits_{t \in [0,1]} 
\left\| \dfrac{\omega_{xx}(t)}{g[\omega(t)]}h_{xx}+\dfrac{\omega_{yy}(t)}{g[\omega(t)]}h_{yy}+2\dfrac{\omega_{xy}(t)}{g[\omega(t)]}h_{xy}
\right\|_{\mathcal{C}^{0}(\overline\Omega)}
\\[3mm]
&\qquad\le \sup\limits_{t \in [0,1]} \left(\sup\limits_{\bm{x}\in \overline\Omega}\left|\dfrac{\omega_{xx}(t)}{g[\omega(t)]}(\bm{x})h_{xx}(\bm{x})\right|+\sup\limits_{\bm{x}\in \overline\Omega}\left|\dfrac{\omega_{yy}(t)}{g[\omega(t)]}(\bm{x})h_{yy}(\bm{x})\right|+2\sup\limits_{\bm{x}\in \overline\Omega}\left|\dfrac{\omega_{xy}(t)}{g[\omega(t)]}(\bm{x})h_{xy}(\bm{x})\right|\right)
\\[3mm]
&\qquad\stackrel{(\clubsuit)}{\le} \dfrac{\rho}{\sqrt{2\tau}}\left(\sup\limits_{\bm{x}\in \overline\Omega}|h_{xx}(\bm{x})|+\sup\limits_{\bm{x}\in \overline\Omega}|h_{yy}(\bm{x})|+2\sup\limits_{\bm{x}\in \overline\Omega}|h_{xy}(\bm{x})|\right)
\\[3mm]
&\qquad\le \dfrac{\rho}{\sqrt{2\tau}}\|u-v\|_{\mathcal{C}^{2,\alpha}(\overline\Omega)},
\end{array}
$$
where the inequality $(\clubsuit)$ is obtained from the fact that $u,v\in B_{\rho}$, and $g[u]$ is bounded from below by $\sqrt{2f} \geq \sqrt{2\tau} > 0$.

On the other hand, for the  $\alpha$-H\"older semi-norm term on the right-hand side of \eqref{eq: c0a-norm}, one has
$$
\begin{array}{lll}
\left|g[u]-g[v]\right|_{0,\alpha;\overline\Omega}
\le\sup\limits_{t \in [0,1]}\left|\dfrac{\omega_{xx}(t)}{g[\omega(t)]}h_{xx}+\dfrac{\omega_{yy}(t)}{g[\omega(t)]}h_{yy}+2\dfrac{\omega_{xy}(t)}{g[\omega(t)]}h_{xy}\right|_{0,\alpha;\overline\Omega} 
\le I_1+I_2+I_3, 
\end{array}
$$
where 
$$
I_1:=\sup\limits_{t \in [0,1]} \left|\dfrac{\omega_{xx}(t)}{g[\omega(t)]}h_{xx}\right|_{0,\alpha;\overline\Omega},
\quad
I_2:=\sup\limits_{t \in [0,1]}\left|\dfrac{\omega_{yy}(t)}{g[\omega(t)]}h_{yy}\right|_{0,\alpha;\overline\Omega},
\quad
I_3:=
2\sup\limits_{t \in [0,1]}\left|\dfrac{\omega_{xy}(t)}{g[\omega(t)]}h_{xy}\right|_{0,\alpha;\overline\Omega}.
$$
From direct computations, one has
$$
\begin{array}{lll}
I_1&
=\sup\limits_{t \in [0,1]}\sup\limits_{\substack{\bm{x}, \bm{y} \in \overline\Omega \\ \bm{x} \neq \bm{y}}}\dfrac{\left|\dfrac{\omega_{xx}(t)}{g[\omega(t)]}(\bm{x})h_{xx}(\bm{x})-\dfrac{\omega_{xx}(t)}{g[\omega(t)]}(\bm{y})h_{xx}(\bm{y})\right|}{\|\bm{x}-\bm{y}\|^\alpha}\\
& \le \sup\limits_{t \in [0,1]}\sup\limits_{\substack{\bm{x}, \bm{y} \in \overline\Omega \\ \bm{x} \neq \bm{y}}}\dfrac{\left|\dfrac{\omega_{xx}(t)}{g[\omega(t)]}(\bm{x})(h_{xx}(\bm{x})-h_{xx}(\bm{y}))\right|+\left|\left(\dfrac{\omega_{xx}(t)}{g[\omega(t)]}(\bm{x})-\dfrac{\omega_{xx}(t)}{g[\omega(t)]}(\bm{y})\right)h_{xx}(\bm{y})\right|}{\|\bm{x}-\bm{y}\|^\alpha}\\[1em]
& \le \sup\limits_{t \in [0,1]}\sup\limits_{\substack{\bm{x}, \bm{y} \in \overline\Omega \\ \bm{x} \neq \bm{y}}}\dfrac{\Big|\dfrac{\omega_{xx}(t)}{g[\omega(t)]}(\bm{x})\Big|\cdot\Big|  h_{xx}(\bm{x})- h_{xx}(\bm{y})\Big|+\Big|\dfrac{\omega_{xx}(t)}{g[\omega(t)]}(\bm{x})-\dfrac{\omega_{xx}(t)}{g[\omega(t)]}(\bm{y})\Big|\cdot\left|h_{xx}(\bm{y})\right|}{\|\bm{x}-\bm{y}\|^\alpha}\\
&\le\|h \|_{\mathcal{C}^{2,\alpha}(\overline\Omega)} \left(\dfrac{\rho}{\sqrt{2\tau}}+\sup\limits_{t \in [0,1]} \sup\limits_{\substack{\bm{x}, \bm{y} \in \overline\Omega \\ \bm{x} \neq \bm{y}}}\dfrac{\left|\dfrac{\omega_{xx}(t)}{g[\omega(t)]}(\bm{x})-\dfrac{\omega_{xx}(t)}{g[\omega(t)]}(\bm{y})\right|}{\| \bm{x}- \bm{y}\|^\alpha}\right)\\
&\le\|h \|_{\mathcal{C}^{2,\alpha}(\overline\Omega)} \left(\dfrac{\rho}{\sqrt{2\tau}}+\sup\limits_{t \in [0,1]} \sup\limits_{\substack{\bm{x}, \bm{y} \in \overline\Omega \\ \bm{x} \neq \bm{y}}}\dfrac{\left|\dfrac{\omega_{xx}(t)}{g[\omega(t)]}(\bm{x})-\dfrac{\omega_{xx}(t)(\bm{y})}{g[\omega(t)](\bm{x})}\right|+\left|\dfrac{\omega_{xx}(t)(\bm{y})}{g[\omega(t)](\bm{x})}-\dfrac{\omega_{xx}(t)}{g[\omega(t)]}(\bm{y})\right|}{\| \bm{x}- \bm{y}\|^\alpha}\right)\\
&\stackrel{(\spadesuit)}{\le}\|h \|_{\mathcal{C}^{2,\alpha}(\overline\Omega)} \Big(\dfrac{\rho}{\sqrt{2\tau}}+\dfrac{\rho}{\sqrt{2\tau}}+\rho\sup\limits_{t \in [0,1]}\Big|\dfrac{1}{g[\omega]}\Big|_{0,\alpha;\overline{\Omega}}\Big),
\end{array}
$$
where the inequality $(\spadesuit)$ is obtained from the fact that 
$$
\begin{array}{lll}
\sup\limits_{t \in [0,1]} \sup\limits_{\substack{\bm{x}, \bm{y} \in \overline\Omega \\ \bm{x} \neq \bm{y}}}\dfrac{\left|\dfrac{\omega_{xx}(t)}{g[\omega(t)]}(\bm{x})-\dfrac{\omega_{xx}(t)(\bm{y})}{g[\omega(t)](\bm{x})}\right|}{\| \bm{x}- \bm{y}\|^\alpha}
\leq\sup\limits_{t \in [0,1]} \sup\limits_{\substack{\bm{x}, \bm{y} \in \overline\Omega \\ \bm{x} \neq \bm{y}}}\left|\dfrac{1}{g[\omega(t)]}(\bm{x})\right|\cdot\dfrac{|\omega_{xx}(t)(\bm{x})-\omega_{xx}(t)(\bm{y})|}{\| \bm{x}- \bm{y}\|^\alpha}
\\
\leq \dfrac{1}{\sqrt{2\tau}}\sup\limits_{t \in [0,1]}|\omega_{xx}(t)|_{0,\alpha;\overline\Omega}
\leq \dfrac{1}{\sqrt{2\tau}}\sup\limits_{t \in [0,1]}\|\omega(t)\|_{\mathcal{C}^{2,\alpha}(\overline\Omega)}\leq \dfrac{\rho}{\sqrt{2\tau}}.
\end{array}
$$
 Moreover, note that
$$
\begin{array}{ll}
\left|\dfrac{1}{g[\omega(t)]}(\bm{x})-\dfrac{1}{g[\omega(t)]}(\bm{y})\right|=\left|\dfrac{g[\omega(t)](\bm{y})-g[\omega(t)](\bm{x})}{g[\omega(t)](\bm{x})g[\omega(t)](\bm{y})}\right|\le \dfrac{1}{2\tau}|g[\omega(t)](\bm{y})-g[\omega(t)](\bm{x})|.
\end{array}
$$ 
Then, recalling the definition of the H{\"o}lder semi-norm, we obtain
\begin{equation}
\label{fixed.3}
\begin{array}{ll}
 &\sup\limits_{t \in [0,1]}\left|\dfrac{1}{g[\omega(t)]}\right|_{0,\alpha;\overline\Omega}\le \dfrac{1}{2\tau}\sup\limits_{t \in [0,1]} \Big|g[\omega(t)]\Big|_{0,\alpha;\overline\Omega}
 \\[3mm]
 &\qquad\le\dfrac{1}{2\tau} \sup\limits_{t \in [0,1]} \sup\limits_{\substack{\bm{x}, \bm{y} \in \overline\Omega \\ \bm{x} \neq \bm{y}}}\Big(\dfrac{|\omega _{xx}^2(t)(\bm{x})-\omega^2_{xx}(t)(\bm{y})|}{(g[\omega(t)](\bm{x})+g[\omega(t)](\bm{y}))\| \bm{x}- \bm{y}\|^\alpha}+\dfrac{|\omega^2 _{yy}(t)(\bm{x})-\omega^2_{yy}(t)(\bm{y})|}{(g[\omega(t)](\bm{x})+g[\omega(t)](\bm{y}))\| \bm{x}- \bm{y}\|^\alpha}
 \\[3mm]
 &\hspace{3cm}+\dfrac{2|\omega _{xy}^2(t)(\bm{x})-\omega^2_{xy}(t)(\bm{y})|}{(g[\omega(t)](\bm{x})+g[\omega(t)](\bm{y}))\| \bm{x}- \bm{y}\|^\alpha}+\dfrac{2|f(\bm{x})-f(\bm{y})|}{(g[\omega(t)](\bm{x})+g[\omega(t)](\bm{y}))\| \bm{x}- \bm{y}\|^\alpha} \Big) \\[4mm] 
&\qquad\le \dfrac{1}{(2\tau)^{{3}/{2}}}  \sup\limits_{t \in [0,1]} \sup\limits_{\substack{\bm{x}, \bm{y} \in \overline\Omega\\ \bm{x} \neq \bm{y}}}
\Big(
\dfrac{\rho|\omega _{xx}(t)(\bm{x})-\omega_{xx}(t)(\bm{y})|}{\| \bm{x}- \bm{y}\|^\alpha} 
+\dfrac{\rho|\omega _{yy}(t)(\bm{x})-\omega_{yy}(t)(\bm{y})|}{\| \bm{x}- \bm{y}\|^\alpha} \\
&\hspace{4cm} +\dfrac{2\rho|\omega _{xy}(t)(\bm{x})-\omega_{xy}(t)(\bm{y})|}{\| \bm{x}- \bm{y}\|^\alpha}
+\dfrac{|f(\bm{x})-f(\bm{y})|}{\| \bm{x}- \bm{y}\|^\alpha}
\Big) \\ 
 &\qquad\le\dfrac{\rho}{(2\tau)^{{3}/{2}}}\sup\limits_{t \in [0,1]} \|\omega(t)\|_{\mathcal{C}^{2,\alpha}(\overline\Omega)}+\dfrac{1}{(2\tau)^{{3}/{2}}}\sup\limits_{t \in [0,1]} \|f\|_{\mathcal{C}^{0,\alpha}(\overline\Omega)} \le
 \dfrac{\rho^2+M_f}{(2\tau)^{{3}/{2}}}.
\end{array}
\end{equation}
Therefore, the following estimate holds immediately
\begin{equation*}
\begin{array}{ll}
I_1=\sup\limits_{t \in [0,1]} \Big|\dfrac{\omega_{xx}(t)}{g[\omega(t)]}h_{xx}\Big|_{0,\alpha;\overline\Omega}\le\Big(\dfrac{2\rho}{\sqrt{2\tau}}+\dfrac{\rho^3+\rho M_f}{(2\tau)^{{3}/{2}}}\Big)\|u-v\|_{\mathcal{C}^{2,\alpha}(\overline\Omega)}.
\end{array}
\end{equation*}
Similarly, for $I_2$ and $I_3$, one has
\begin{equation*}
\begin{cases}
I_2=\sup\limits_{t \in [0,1]} \Big|\dfrac{\omega_{yy}(t)}{g[\omega(t)]}h_{yy}\Big|_{0,\alpha;\overline\Omega}\le\Big(\dfrac{2\rho}{\sqrt{2\tau}}+\dfrac{\rho^3+\rho M_f}{(2\tau)^{{3}/{2}}}\Big)\|u-v\|_{\mathcal{C}^{2,\alpha}(\overline\Omega)},
\\
I_3=2\sup\limits_{t \in [0,1]} \Big|\dfrac{\omega_{xy}(t)}{g[\omega(t)]}h_{xy}\Big|_{0,\alpha;\overline\Omega} \le2\Big(\dfrac{2\rho}{\sqrt{2\tau}}+\dfrac{\rho^3+\rho M_f}{(2\tau)^{{3}/{2}}}\Big)\|u-v\|_{\mathcal{C}^{2,\alpha}(\overline\Omega)}.
\end{cases}
\end{equation*}
Accordingly, we obtain 
\begin{equation}
\label{fixed.4}
\|  g[u]-g[v]\|_{\mathcal{C}^{0,\alpha}(\overline\Omega)}\le
\Big(\frac{9\rho}{\sqrt{2\tau}}+\frac{4\rho^3+4\rho M_f}{(2\tau)^{{3}/{2}}}\Big)\|u-v\|_{\mathcal{C}^{2,\alpha}(\overline\Omega)}.
\end{equation}

Finally, for $u,v\in \mathcal{C}^{2,\alpha}(\overline\Omega)$, based on the definition of the operator $Q_2$ in \eqref{eq:Q2}, by Schauder's estimates \cite[Theorem 6.6]{gilbarg1977elliptic} and the maximum principle \cite[Theorem 3.1]{gilbarg1977elliptic}, combined with \eqref{delta} and \eqref{fixed.4}, one has the contraction estimate for the operator ${Q_2}$ mapping on the complete metric space $B_{\rho}$ in the sense of $\mathcal{C}^{2,\alpha}$-norm that
\begin{equation}
\label{contraction}
\| {Q_2}(u) - {Q_2}(v) \|_{\mathcal{C}^{2,\alpha}(\overline\Omega)} \leq C\| g[u]- g[v]\|_{\mathcal{C}^{0,\alpha}(\overline\Omega)}\le \delta \|u-v\|_{\mathcal{C}^{2,\alpha}(\overline\Omega)},
\end{equation}
where $\delta <1$ and the constant $C$ depends only on $\alpha$ and $\Omega$. 

\paragraph{Step 2: Verifying that the operator ${Q_2}$ maps $B_{\rho}$ into itself.}~
 
\smallskip
Based on the definition of the operator ${Q_2}$, by applying Schauder's estimates 
 \cite[Theorem 6.6]{gilbarg1977elliptic} and the maximum principle \cite[Theorem 3.1]{gilbarg1977elliptic}, for any $u \in B_\rho$, we have
\begin{equation}
\label{fixed.5}
\| {Q_2}(u)  \|_{\mathcal{C}^{2,\alpha}(\overline\Omega)} \leq C\| g[u]\|_{\mathcal{C}^{0,\alpha}(\overline\Omega)},
\end{equation}
where the $\mathcal{C}^{0,\alpha}$-norm of $g[u]$ is given by 
\begin{equation}
\label{eq: c0agu}
\|g[u]\|_{\mathcal{C}^{0,\alpha}(\overline\Omega)}=\|  g[u]\|_{\mathcal{C}^{0}(\overline\Omega)}+|  g[u]|_{0,\alpha;\overline\Omega}.
\end{equation}
For the  $\mathcal{C}^0$-norm term on the right-hand side of \eqref{eq: c0agu}, one has 
\begin{equation}
\label{fixed.6}
\begin{array}{ll}
\|  g[u]\|_{\mathcal{C}^{0}(\overline\Omega)}
= \sup\limits_{\bm{x}\in \overline\Omega} \dfrac{ |g[u](\bm{x})|^2}{ |g[u](\bm{x})|} 
\le \sup\limits_{\bm{x}\in \overline\Omega}
\dfrac{| u_{xx}^2(\bm{x})+ u_{yy}^2(\bm{x})+2u_{xy}^2(\bm{x}) +2f(\bm{x})|}{\sqrt{u_{xx}^2(\bm{x})+ u_{yy}^2(\bm{x})+2u_{xy}^2(\bm{x}) +2f(\bm{x})}}  \le \dfrac{4\rho^2+2M_f}{\sqrt{2\tau}}.
\end{array}
\end{equation}
Based on the derivation in \eqref{fixed.3}, the estimate for $|  g[u]|_{0,\alpha;\overline\Omega}$ is given by
\begin{equation}
\label{fixed.7}
|  g[u]|_{0,\alpha;\overline\Omega}
\le \frac{\rho^2+M_f}{\sqrt{2\tau}}.
\end{equation}
Combining \eqref{delta}, \eqref{fixed.5}, \eqref{fixed.6} and \eqref{fixed.7} together, we obtain
\begin{equation}
\label{mapping}
\| {Q_2}(u)\|_{\mathcal{C}^{2,\alpha}(\overline\Omega)} \leq C\Big(\frac{5\rho^2+3M_f}{\sqrt{2\tau}}\Big)\le\delta \rho <\rho,
\end{equation}
which shows that the operator ${Q_2}$ maps the closed ball $B_{\rho}$ into itself, i.e., ${Q_2}(B_{\rho})\subset B_{\rho}$.

\paragraph{Step 3: The sequence $\{u_{k}\}_{k\ge0}$ converges to the solution to the MA equation \eqref{3.6}.}~
 
\smallskip
According to \eqref{contraction} and \eqref{mapping}, we can apply the Banach fixed-point theorem  \cite[Theorem 5.7]{brezis2011functional}, which implies that the operator ${Q_2}$ has a unique fixed point $u\in B_{\rho}$. 
Furthermore, for any initial value $u_0\in B_{\rho}$, the iterative sequence defined by $u_{k+1} = {Q_2}(u_{k})$ converges to this unique fixed point $u$ in the $\mathcal{C}^{2,\alpha}$-norm, i.e., $\lim_{k \to \infty} u_k =u$.
Moreover, a solution to the MA equation \eqref{3.6} is the fixed point of the operator ${Q_2}$. 
Since $u$ is the unique fixed point of ${Q_2}$ within $B_{\rho}$, it follows that $u$ is the solution to the MA equation \eqref{3.6}.
This completes the proof. 
\end{proof}

\subsection{A fixed-point approach for \texorpdfstring{$\Omega\subset\mathbb{R}^3$}{Omega subset R3}}
In the 3D case with $\bm x=(x,y,z)^T=(x_1,x_2,x_3)^T\in \Omega$, the determinant of the Hessian is given by
$$
\prod_{i=1}^3 \lambda_i[D^2u]=\det D^2u = u_{xx} (u_{yy} u_{zz} - u_{yz}^2) 
- u_{xy} (u_{xy} u_{zz} - u_{xz} u_{yz}) 
+ u_{xz} (u_{xy} u_{yz} - u_{xz} u_{yy}),
$$
and the cube of the Laplacian expands as
$$
(\Delta u)^3 = u_{xx}^3 + u_{yy}^3 + u_{zz}^3 
+ 3 \big(u_{xx}^2 u_{yy} + u_{xx}^2 u_{zz} + u_{yy}^2 u_{xx} + u_{yy}^2 u_{zz} + u_{zz}^2 u_{xx} + u_{zz}^2 u_{yy}\big) 
+ 6 u_{xx} u_{yy} u_{zz}.
$$
Expanding $(\Delta u)^3$ and isolating the determinant term yield
$$
\begin{aligned}
P_3\bigl(\lambda_1[D^2u], \lambda_2[D^2u], \lambda_3[D^2u]\bigr) 
&= (\Delta u)^3 - 6\,\det(D^2 u) \\
&= u_{xx}^3 + u_{yy}^3 + u_{zz}^3 \\
&\quad + 3 ( u_{xx}^2 u_{yy} + u_{xx}^2 u_{zz} 
                 + u_{yy}^2 u_{xx} + u_{yy}^2 u_{zz} 
                 + u_{zz}^2 u_{xx} + u_{zz}^2 u_{yy} ) \\
&\quad + 6 ( u_{xx} u_{yz}^2 + u_{yy} u_{xz}^2 + u_{zz} u_{xy}^2 ) 
     - 12 \, u_{xy} u_{yz} u_{xz}.
\end{aligned}
$$
Define the operator ${Q_3}:=Q_3(u)$ by
\begin{equation}\label{eq:q3}
    {Q_3}(u) := \Delta^{-1}  \big(P_3(\lambda_1[D^2u],\lambda_2[D^2u],\lambda_3[D^2u]) + 6 f \big)^{1/3},
\end{equation}
where $\Delta^{-1}$ denotes the solution operator for the Poisson equation with zero Dirichlet boundary condition. A solution $u$ to the 3D MA equation is a fixed point of ${Q_3}$, i.e., $u = {Q_3}(u)$, which motivates the iteration scheme
$$
u_{k+1} = {Q_3}(u_k),
$$
where $u_{k+1}$ is obtained by solving the linear Poisson equation
\begin{equation*}
\begin{cases}
\Delta u_{k+1} = \big( {P_3(\lambda_1[D^2u_k],\lambda_2[D^2u_k],\lambda_3[D^2u_k])} + 6 f \big)^{
1/3}, & \text{in } \Omega,\\[1ex]
u_{k+1} = 0, & \text{on } \partial \Omega.
\end{cases}
\end{equation*}

The following theorem establishes the local convergence of this fixed-point iteration in the 3D case under appropriate conditions. 
In contrast to the 2D case, Theorem \ref{convergence3d} requires the initial point $u_0$ to lie in a small neighborhood of the solution $u^*$. 
Nevertheless, this limitation is not a severe issue, since in practice we can employ a warm-start strategy to obtain suitable initial points, as introduced in the next section.

\begin{theorem}
\label{convergence3d}
Let $\Omega\subset \mathbb{R}^3$ be a bounded open domain with $\mathcal{C}^{2,\alpha}$ boundary for some $\alpha\in(0,1)$, and $f\in \mathcal{C}^{0,\alpha}(\overline{\Omega})$. 
Assume that there exists a solution $u^*$ to the MA equation \eqref{3.6}.
Then, there exists a constant $\tau\equiv\tau(\Omega,\alpha,u^*)>0$, such that for any $f$ with $f> \tau$ on $\overline\Omega$, 
one has a neighborhood 
$$
B(u^*;\varepsilon):= \{u\in \mathcal{C}^{2,\alpha}(\overline\Omega)\mid \|u-u^*\|_{\mathcal{C}^{2,\alpha}(\overline\Omega)}\leq\varepsilon,\;u\big|_{\partial\Omega} = 0\}
\quad\mbox{with}\quad \varepsilon>0,
$$
ensuring that, for every initial function $u_0\in B(u^*;\varepsilon)$, 
the sequence $\{u_k\}_{k\geq 0}$ generated by the fixed-point iteration $u_{k+1} = Q_3(u_k)$ converges to $u^*$ in the $\mathcal{C}^{2,\alpha}$-norm. 
\end{theorem}

\begin{proof}
Based on the Banach fixed-point theorem  \cite[Theorem 5.7]{brezis2011functional}, the proof is divided into the following three steps. 
The first step is to analyze the contraction property of the operator $Q_3$.
Then the second step is to prove that $Q_3$ is a self-mapping on $B(u^*;\varepsilon)$.
The final step is to verify that the iterative sequence $\{u_{k}\}_{k\ge0}$ converges to the solution to the MA equation \eqref{3.6}.

\paragraph{Step 1: Contraction property of the operator $Q_3$.}~
 
\smallskip
Let $u,v\in\mathcal{C}^{2,\alpha}(\overline{\Omega})\cap B(u^*,\varepsilon)$, given $0<\varepsilon < 1$, for any $u\in B(u^*;\varepsilon)$, its $\mathcal{C}^{2,\alpha}$-norm is uniformly bounded by $K(u^*):=\|u^*\|_{\mathcal{C}^{2,\alpha}}+1$. 
This implies a uniform bound $K(u^*)$ on their second derivatives, and
 $$|\lambda_l[D^2u]|\leq \|D^2u\|_F = \sqrt{\sum_{i,j=1}^3|u_{x_ix_j}|^2}\leq \sqrt{9 (K(u^*))^2} =  3K(u^*),\quad l=1,2,3.$$
 For simplicity of notation, denote
$P_3(\lambda[D^2u]) := P_3(\lambda_1[D^2u], \lambda_2[D^2u], \lambda_3[D^2u])$, and define
$$
\begin{aligned}
g[u]:=& P_3(\lambda[D^2u])
=  
u_{xx}^3 + u_{yy}^3 + u_{zz}^3 
+ \big(3u_{xx}u_{yy}^2 + 3u_{yy}u_{xx}^2\big) 
+ \big(3u_{xx}u_{zz}^2 + 3u_{zz}u_{xx}^2\big) 
\\[1mm]
&+ \big(3u_{yy}u_{zz}^2 + 3u_{zz}u_{yy}^2\big)+ \big(6u_{xx}u_{yz}^2 + 6u_{yy}u_{xz}^2 + 6u_{zz}u_{xy}^2\big) 
- 12u_{xz}u_{xy}u_{yz},
\end{aligned}
$$
with
 $$
 \begin{array}{ll}
|P_3(\lambda[D^2u])| &= |(\lambda_1+\lambda_2+\lambda_3)^3 - 6\lambda_1\lambda_2\lambda_3|
\leq |(\lambda_1+\lambda_2+\lambda_3)^3|+|6\lambda_1\lambda_2\lambda_3|
\\[1mm]
&\leq (|\lambda_1|+|\lambda_2|+|\lambda_3|)^3 + |6\lambda_1\lambda_2\lambda_3| 
\leq K_p:= 
(6\times3^3+9^3) (K(u^*))^3.
\end{array}
$$
Let $\omega(t):=v+t(u-v)$ and $h:=u-v$, the difference $g[u]- g[v]$ can be expanded as
\begin{equation*}
g[u]- g[v]=\int_0^1  \frac{\mathrm{d}}{\mathrm{d}t}g[\omega(t)] \ \mathrm{d} t = \int_0^1 I_1(t) + I_2(t) +I_3(t)+I_4(t)\ \mathrm{d} t,
\end{equation*}
where
\begin{equation*}
\begin{cases}
\begin{aligned}
 I_1(t) &:= 3\omega_{xx}^2(t) h_{xx}+3\omega_{yy}^2(t) h_{yy}+3\omega_{zz}^2(t) h_{zz}, \\[2mm] 
 I_2(t) &:= 3\omega_{yy}^2(t) h_{xx} + 6\omega_{xx}(t)\omega_{yy}(t) h_{yy}
      + 3\omega_{zz}^2(t) h_{xx} + 6\omega_{xx}(t)\omega_{zz}(t) h_{zz} \\
     &\quad+ 3\omega_{xx}^2(t) h_{yy} + 6\omega_{xx}(t)\omega_{yy}(t) h_{xx}
      + 3\omega_{zz}^2(t) h_{yy} + 6\omega_{yy}(t)\omega_{zz}(t) h_{zz} \\
       &\quad+ 3\omega_{xx}^2(t) h_{zz} + 6\omega_{xx}(t)\omega_{zz}(t) h_{xx}
        + 3\omega_{yy}^2(t) h_{zz} + 6\omega_{zz}(t)\omega_{yy}(t) h_{yy}, \\[2mm] 
I_3(t) &:= 6\omega_{yz}^2(t) h_{xx} + 12\omega_{xx}(t)\omega_{yz}(t) h_{yz} + 6\omega_{xy}^2(t) h_{zz} \\
     &\quad+ 12\omega_{zz}(t)\omega_{xy}(t) h_{xy} + 6\omega_{xz}^2(t) h_{yy} + 12\omega_{yy}(t)\omega_{xz}(t) h_{xz}, \\[2mm] 
I_4(t) &:= -12\omega_{xz}(t)\omega_{yz}(t) h_{xy}
    -12\omega_{xy}(t)\omega_{xz}(t) h_{yz}
                     -12\omega_{xy}(t)\omega_{yz}(t) h_{xz}.
\end{aligned}
\end{cases}
\end{equation*}
Recall that the $\mathcal{C}^{0,\alpha}$-norm of the difference $g[u] - g[v]$ is given by
\begin{equation}
\label{eq: c0a-norm2}
\|  g[u]-g[v]\|_{\mathcal{C}^{0,\alpha}(\overline\Omega)}=\|  g[u]-g[v]\|_{\mathcal{C}^{0}(\overline\Omega)}+|  g[u]-g[v]|_{0,\alpha;\overline\Omega}.
\end{equation}
For the first term on the right-hand side of \eqref{eq: c0a-norm2}, one has
\begin{equation}
\label{eq:first-norm}
\begin{array}{lll}
\|  g[u] - g[v] \|_{\mathcal{C}^{0}(\overline\Omega)} 
& 
\le \sup\limits_{t \in [0,1]} \big( \left\| \displaystyle I_1(t)\right\|_{\mathcal{C}^{0}(\overline\Omega)} +\left\| \displaystyle I_2(t)\right\|_{\mathcal{C}^{0}(\overline\Omega)}+\left\| \displaystyle I_3(t)\right\|_{\mathcal{C}^{0}(\overline\Omega)}+\left\| \displaystyle I_4(t)\right\|_{\mathcal{C}^{0}(\overline\Omega)}
\big)\\[1mm]
&\stackrel{(\diamondsuit)}{\le}  (K(u^*))^2\Big(
27\sup\limits_{\bm{x} \in \overline\Omega} |h_{xx}(\bm{x})|
+27 \sup\limits_{\bm{x} \in \overline\Omega} |h_{yy}(\bm{x})|
+ 27\sup\limits_{\bm{x} \in \overline\Omega} |h_{zz}(\bm{x})|  \\[1mm]
&\qquad 
+ 24 \sup\limits_{\bm{x} \in \overline\Omega} |h_{xy}(\bm{x})|
+ 24\sup\limits_{\bm{x} \in \overline\Omega} |h_{xz}(\bm{x})|
+ 24 \sup\limits_{\bm{x} \in \overline\Omega} |h_{yz}(\bm{x})|
\Big)
\\[2mm]
&\le 27 (K(u^*))^2 \| u - v \|_{\mathcal{C}^{2,\alpha}(\overline\Omega)},
\end{array}
\end{equation}
where the inequality $(\diamondsuit)$ is obtained from the fact that $u,v\in B(u^*;\varepsilon)$.
Moreover, note that
\begin{align*}
&\displaystyle\left|\int_0^1  \omega_{xx}(t)\omega_{yy}(t) h_{yy} \mathrm{d}t\right|_{0,\alpha;\overline\Omega} 
\\
&
\le \sup\limits_{t \in [0,1]} \sup\limits_{\substack{\bm{x}, \bm{y} \in \overline\Omega \\ \bm{x} \neq \bm{y}}} 
\dfrac{ \big| 
\omega_{xx}(t)(\bm{x}) \omega_{yy}(t)(\bm{x}) h_{yy}(\bm{x}) 
- \omega_{xx}(t)(\bm{x}) \omega_{yy}(t)(\bm{x}) h_{yy}(\bm{y}) 
\big| }{\|\bm{x}-\bm{y}\|^\alpha}  \\
&\quad +  \sup\limits_{t \in [0,1]} \sup\limits_{\substack{\bm{x}, \bm{y} \in \overline\Omega \\ \bm{x} \neq \bm{y}}} 
\dfrac{\big| 
\omega_{xx}(t)(\bm{x}) \omega_{yy}(t)(\bm{x}) h_{yy}(\bm{y})  
- \omega_{xx}(t)(\bm{y}) \omega_{yy}(t)(\bm{y}) h_{yy}(\bm{y}) 
\big| }{\|\bm{x}-\bm{y}\|^\alpha}
\\
&\le 
 (K(u^*))^2|h_{yy}|_{0,\alpha;\overline\Omega}
+ \sup\limits_{ \bm{y} \in \overline\Omega}|h_{yy}(\bm{y}) |  \sup\limits_{t \in [0,1]} \sup\limits_{\substack{\bm{x}, \bm{y} \in \overline\Omega \\ \bm{x} \neq \bm{y}}} 
\dfrac{ \big| 
\omega_{xx}(t)(\bm{x}) \omega_{yy}(t)(\bm{x}) 
- \omega_{xx}(t)(\bm{x}) \omega_{yy}(t)(\bm{y}) 
\big| }{\|\bm{x}-\bm{y}\|^\alpha} \\ 
&\quad 
+ \sup\limits_{ \bm{y} \in \overline\Omega}|h_{yy}(\bm{y}) |  \sup\limits_{t \in [0,1]} \sup\limits_{\substack{\bm{x}, \bm{y} \in \overline\Omega \\ \bm{x} \neq \bm{y}}} 
\dfrac{ \big| 
\omega_{xx}(t)(\bm{x}) \omega_{yy}(t)(\bm{y}) 
- \omega_{xx}(t)(\bm{y}) \omega_{yy}(t)(\bm{y}) 
\big| }{\|\bm{x}-\bm{y}\|^\alpha}\\
&\le (K(u^*))^2\left(|h_{yy}|_{0,\alpha;\overline\Omega}+\sup\limits_{ \bm{y} \in \overline\Omega}|h_{yy}(\bm{y})|\right),
\end{align*}
and
\begin{align*}
  &  \displaystyle\left|\int_0^1 \omega_{xx}^2(t)\ h_{xx} \mathrm{d}t\right|_{0,\alpha;\overline\Omega} 
\le \sup\limits_{t \in [0,1]} \sup\limits_{\substack{\bm{x}, \bm{y} \in \overline\Omega \\ \bm{x} \neq \bm{y}}} 
\dfrac{ \big| 
\omega_{xx}^2(t)(\bm{x}) h_{xx}(\bm{x})-\omega_{xx}^2(t)(\bm{y}) h_{xx}(\bm{y}) 
\big| }{\|\bm{x}-\bm{y}\|^\alpha}  \\
&\le  \sup\limits_{t \in [0,1]} \sup\limits_{\substack{\bm{x}, \bm{y} \in \overline\Omega \\ \bm{x} \neq \bm{y}}} 
\left(\dfrac{ \big| 
\omega_{xx}^2(t)(\bm{x})  h_{xx}(\bm{x}) 
- \omega_{xx}^2(t)(\bm{x})h_{xx}(\bm{y}) 
\big| }{\|\bm{x}-\bm{y}\|^\alpha}  
+  
\dfrac{ \big| 
\omega_{xx}^2(t)(\bm{x}) h_{xx}(\bm{y})  
- \omega_{xx}^2(t)(\bm{y})h_{xx}(\bm{y}) 
\big| }{\|\bm{x}-\bm{y}\|^\alpha}\right)\\
&\le  (K(u^*))^2|h_{xx}|_{0,\alpha;\overline\Omega}+ \sup\limits_{ \bm{y} \in \overline\Omega}|h_{xx}(\bm{y}) |  \sup\limits_{t \in [0,1]} \sup\limits_{\substack{\bm{x}, \bm{y} \in \overline\Omega \\ \bm{x} \neq \bm{y}}} 
\dfrac{ \big| 
\omega_{xx}^2(t)(\bm{x}) 
- \omega_{xx}^2(t)(\bm{y}) 
\big| }{\|\bm{x}-\bm{y}\|^\alpha}\\
&\le (K(u^*))^2|h_{xx}|_{0,\alpha;\overline\Omega}+ \sup\limits_{ \bm{y} \in \overline\Omega}|h_{xx}(\bm{y}) |  \sup\limits_{t \in [0,1]} \sup\limits_{\substack{\bm{x}, \bm{y} \in \overline\Omega \\ \bm{x} \neq \bm{y}}} 
\dfrac{ \big| 
\omega_{xx}(t)(\bm{x}) \omega_{xx}(t)(\bm{x}) 
- \omega_{xx}(t)(\bm{x}) \omega_{xx}(t)(\bm{y}) 
\big| }{\|\bm{x}-\bm{y}\|^\alpha} \\
&\qquad+  \sup\limits_{ \bm{y} \in \overline\Omega}|h_{xx}(\bm{y}) | \sup\limits_{t \in [0,1]} \sup\limits_{\substack{\bm{x}, \bm{y} \in \overline\Omega \\ \bm{x} \neq \bm{y}}} 
\dfrac{ \big| \omega_{xx}(t)(\bm{x}) \omega_{xx}(t)(\bm{y}) 
- \omega_{xx}(t)(\bm{y}) \omega_{xx}(t)(\bm{y}) 
\big| }{\|\bm{x}-\bm{y}\|^\alpha} \\
&\le  (K(u^*))^2\left(|h_{xx}|_{0,\alpha;\overline\Omega}+2 \sup\limits_{ \bm{y} \in \overline\Omega}|h_{xx}(\bm{y})|\right).
\end{align*}
Therefore, for the $\alpha$-H\"older semi-norm term on the right-hand side of \eqref{eq: c0a-norm2}, one has  
\begin{equation}\label{eq:second-norm}
\begin{array}{lll}
|g[u] - g[v]|_{0,\alpha;\overline\Omega}
&\le 27(K(u^*))^2(|h_{xx}|_{0,\alpha;\overline\Omega}+|h_{yy}|_{0,\alpha;\overline\Omega}+|h_{zz}|_{0,\alpha;\overline\Omega})\\[2mm]
&\quad+24(K(u^*))^2(|h_{xy}|_{0,\alpha;\overline\Omega}+|h_{xz}|_{0,\alpha;\overline\Omega}+|h_{yz}|_{0,\alpha;\overline\Omega}) \\[1mm]
&\quad+42(K(u^*))^2(\sup\limits_{ \bm{y} \in \overline\Omega}|h_{xx}(\bm{y}) |+\sup\limits_{ \bm{y} \in \overline\Omega}|h_{yy}(\bm{y}) |+\sup\limits_{ \bm{y} \in \overline\Omega}|h_{zz}(\bm{y}) |) \\[1mm]
&\quad+24(K(u^*))^2(\sup\limits_{ \bm{y} \in \overline\Omega}|h_{xy}(\bm{y}) |+\sup\limits_{ \bm{y} \in \overline\Omega}|h_{xz}(\bm{y}) |+\sup\limits_{ \bm{y} \in \overline\Omega}|h_{yz}(\bm{y}) |) \\[2mm]
&\le27(K(u^*))^2|h|_{2,\alpha;\overline\Omega}+42(K(u^*))^2\|h\|_{\mathcal{C}^{2}(\overline\Omega)}\\[2mm]
&\le42(K(u^*))^2\|u-v\|_{\mathcal{C}^{2,\alpha}(\overline\Omega)}.
\end{array}
\end{equation}
Combining \eqref{eq:first-norm} and \eqref{eq:second-norm}, we can obtain that 
\begin{equation*}
\|  g[u]-g[v]\|_{\mathcal{C}^{0,\alpha}(\overline\Omega)}\le
69(K(u^*))^2\|u-v\|_{\mathcal{C}^{2,\alpha}(\overline\Omega)}.
\end{equation*}
In addition, let $\hat{u}:= Q_3(u)-Q_3(v)$, one has 
$
\Delta \hat{u} = (6f + P_3(\lambda[D^2u]))^{1/3} - (6f + P_3(\lambda[D^2v]))^{1/3},
$
and $\hat{u}$ vanishes at the boundary. Define $\phi(f) := f^{1/3}$, then $\phi^\prime(f) = \frac{1}{3}f^{-2/3}$.
For each $\bm x\in \overline\Omega$, by the mean value theorem, there exists $\xi(\bm x)$ such that
\begin{equation}\label{eq:laplacew}
\Delta \hat{u}(\bm x) = \phi^\prime(\xi(\bm x))(P_3(\lambda[D^2u])(\bm x)- P_3(\lambda[D^2v])(\bm x))=\phi^\prime(\xi(\bm x))(g[u](\bm x)-g[v](\bm x)).
\end{equation}
Assume  $\tau> \frac{1}{6}K_p$, since $f>\tau$ on $\overline\Omega$, it follows that 
$$
\xi(\bm x) > \inf_{\bm x\in\overline\Omega}6 f(\bm x) - K_p > 6\tau - K_p  >0.
$$
Note that $\phi^{\prime}$ is decreasing for $z>0$, we have
\begin{equation}
\label{phiprime}
|\phi^{\prime}(\xi(\bm x))|< \Gamma_p :=\frac{1}{3(6\tau- K_p)^{2/3}}.
\end{equation}
Thus the $\mathcal{C}^{0,\alpha}$-norm of $\Delta \hat{u}$ is bounded by
$$
\|\Delta \hat{u}\|_{\mathcal{C}^{0,\alpha}(\overline\Omega)} \leq \sup\limits_{\bm x \in \overline\Omega}|\phi^{\prime}(\xi(\bm x))|\|g[u]-g[v]\|_{ \mathcal{C}^{0,\alpha}(\overline\Omega)}\leq \Gamma_p\cdot 69(K(u^*))^2 \|u-v\|_{\mathcal{C}^{2,\alpha}(\overline\Omega)}.
$$
Based on the form of the operator ${Q_3}$ \eqref{eq:q3}, by applying Schauder's estimates \cite[Theorem 6.6]{gilbarg1977elliptic} and the maximum principle \cite[Theorem 3.1]{gilbarg1977elliptic}, when $\tau>\frac{1}{6}((\frac{69C(K(u^*))^2}{3})^{3/2} + K_p)$, we have 
\begin{equation}  
\label{3-fix1}
\|Q_3(u)-Q_3(v)\|_{\mathcal{C}^{2,\alpha}(\overline\Omega)}
< \frac{C\cdot  69(K(u^*))^2}{3((\frac{69C(K(u^*))^2}{3})^{3/2} + K_p-K_p)^{2/3}}\|u-v\|_{\mathcal{C}^{2,\alpha}(\overline\Omega)} = \|u-v\|_{\mathcal{C}^{2,\alpha}(\overline\Omega)},
\end{equation}
where $C$ depends only on $\alpha$ and $\Omega$. 
It shows that $Q_3$ is a contraction on $B(u^*;\varepsilon)$.

\paragraph{Step 2: The operator $Q_3$ is self-mapping on $B(u^*;\varepsilon)$. }~
 
\smallskip
For any given $u\in B(u^*;\varepsilon)$, we want to demonstrate that $\|Q_3(u)-u^*\|_{\mathcal{C}^{2,\alpha}(\overline\Omega)} = \|Q_3(u)-Q_3(u^*)\|_{\mathcal{C}^{2,\alpha}(\overline\Omega)}\leq \varepsilon$. From the above discussion in {\it Step 1}, by replacing $v$ with $u^*$, when $\tau>\frac{1}{6}((\frac{69C(K(u^*))^2}{3})^{3/2} + K_p)$, one has 
\begin{equation}
\label{3-fix2}
\|Q_3(u)-Q_3(u^*)\|_{\mathcal{C}^{2,\alpha}(\overline\Omega)} < 
\|u-u^*\|_{\mathcal{C}^{2,\alpha}(\overline\Omega)} \leq \varepsilon.
\end{equation}
So the operator $Q_3$ is self-mapping on $B(u^*;\varepsilon)$.

\paragraph{ Step 3: The sequence $\{u_{k}\}_{k\ge0}$ converges to the solution of the MA equation \eqref{3.6}.}~
 
\smallskip
According to \eqref{3-fix1} and \eqref{3-fix2}, we can apply the Banach fixed-point theorem  \cite[Theorem 5.7]{brezis2011functional}, which implies that the operator ${Q_3}$ has a unique fixed point $u\in B(u^*;\varepsilon)$. 
Since $u^*$ solves \eqref{3.6}, we have $Q_3(u^*) = u^*$. 
Thus $u^*$ is a fixed point of $Q_3$ in $B(u^*;\varepsilon)$.
Consequently, for any initial value $u_0\in B(u^*;\varepsilon)$, the iterative sequence defined by $u_{k+1} ={Q_3}(u_{k})$ converges to this unique fixed point $u^{*}$ in the $\mathcal{C}^{2,\alpha}$-norm.
This completes the proof. 
\end{proof}

\section{Numerical experiments}
\label{sec:experiments}
In this section, we show numerical experiments to evaluate the performance of our proposed inexact AKI-FP method (Algorithm \ref{alg:fix}) to solve the MAE problem \eqref{eq_MAeig} in 2D and 3D.
All numerical experiments were implemented in MATLAB (R2023b) on a desktop PC with one Intel Core i9-9900KF CPU (8 cores, 3.60 GHz) and 64 GB of memory.

Our method is implemented within a finite element framework. 
The computational domain $\Omega$ is discretized into a mesh $T^h$, which is composed of triangles in 2D and tetrahedra in 3D, with $h$ denoting the characteristic element length.
These meshes are generated with the DistMesh algorithm \cite{persson2004simple}.  
The resulting meshes are then assembled via the IFEM software package \cite{chen2009ifem}, which identifies the boundary and interior nodes required to set up the finite element systems.
Figure \ref{grid2d} and Figure \ref{grid3d} show the meshes for the domains of the tested problems.

\begin{figure}
\centering
\begin{subfigure}{.49\textwidth}
    \centering
    \includegraphics[width=.52\linewidth]{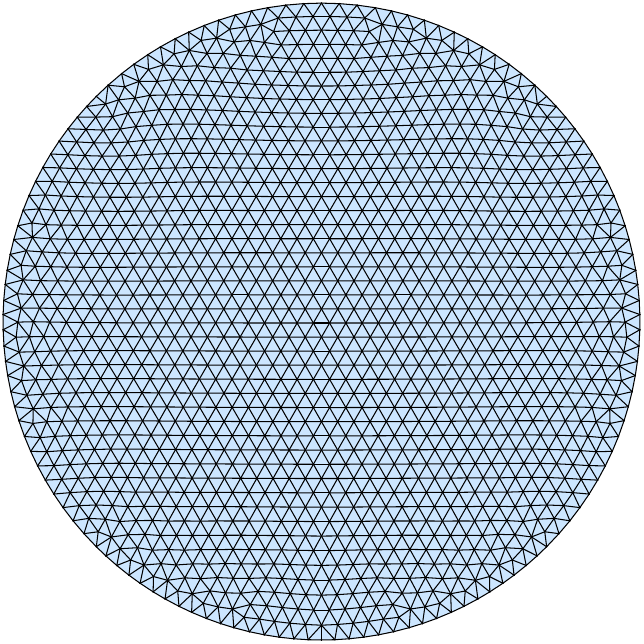}
    \caption{The unit disk domain (\ref{eq.disk})}
     \label{subfig:disk}
\end{subfigure}
\hfill
\begin{subfigure}{.49\textwidth}
    \centering
    \includegraphics[width=.66\linewidth]{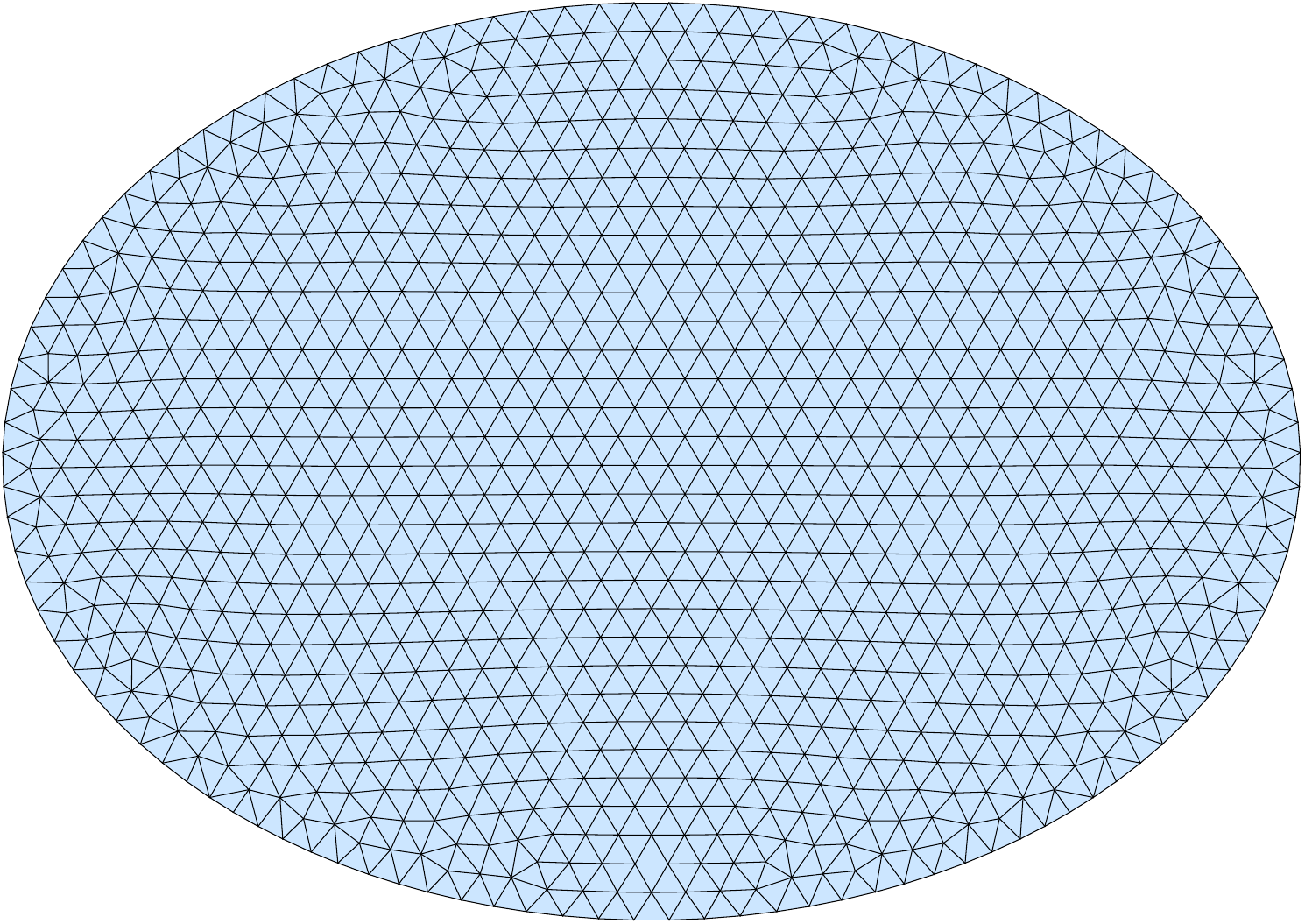}
    \caption{The ellipse domain (\ref{eq.ellipse})}
    \label{subfig:ellipse}
\end{subfigure}

\vspace{2mm}

\begin{subfigure}{.49\textwidth}
    \centering
    \includegraphics[width=.51\linewidth]{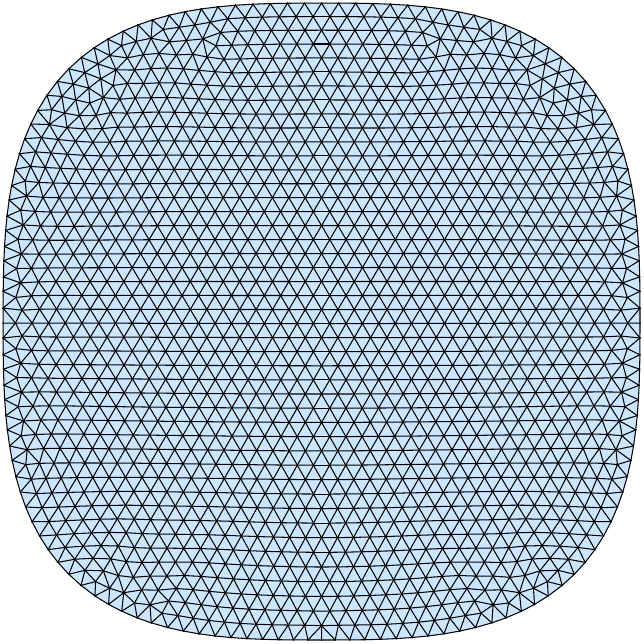}
    \caption{The smoothed square domain (\ref{eq.smoothsq})}
    \label{subfig:smoothsq}
\end{subfigure}
\hfill
\begin{subfigure}{.49\textwidth}
    \centering
    \includegraphics[width=.5\linewidth]{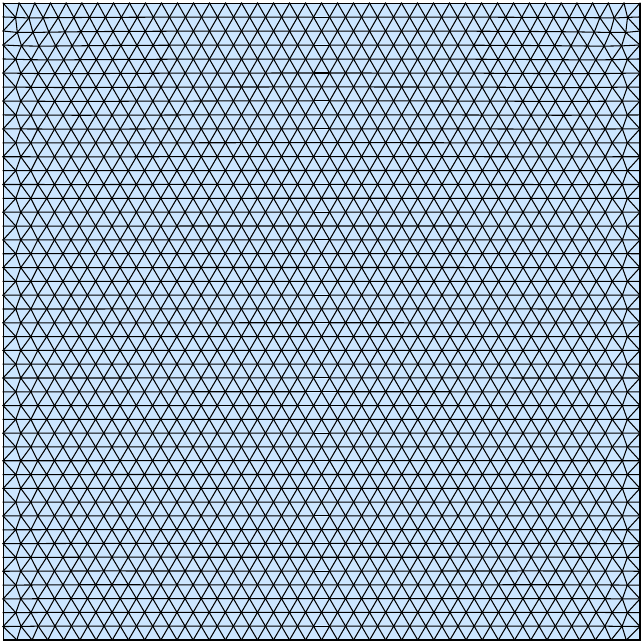}
    \caption{The unit square domain (\ref{eq.sq})}
    \label{subfig:square}
\end{subfigure}

\caption{The domains of the problems and the meshes in $\mathbb{R}^2$ (edge length $h=1/20$)}
\label{grid2d}
\end{figure}

\begin{figure}
\centering
\begin{subfigure}{0.31\textwidth}
    \centering
    \includegraphics[width=.66\linewidth]{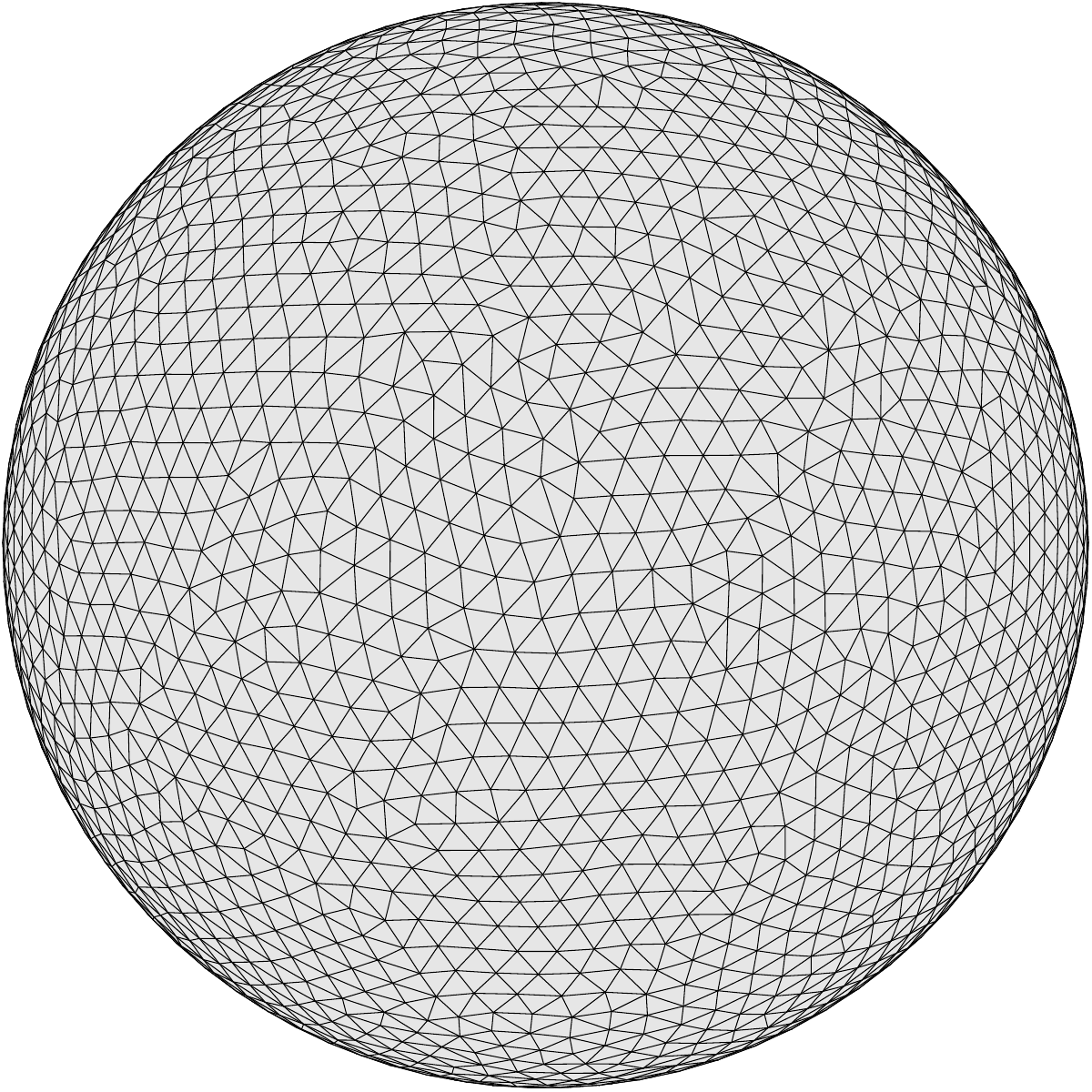}
    \caption{The unit ball domain (\ref{eq.ball})}
    \label{subfig:ball}
\end{subfigure}
\hfill
\begin{subfigure}{0.31\textwidth}
    \centering
    \includegraphics[width=.7\linewidth]{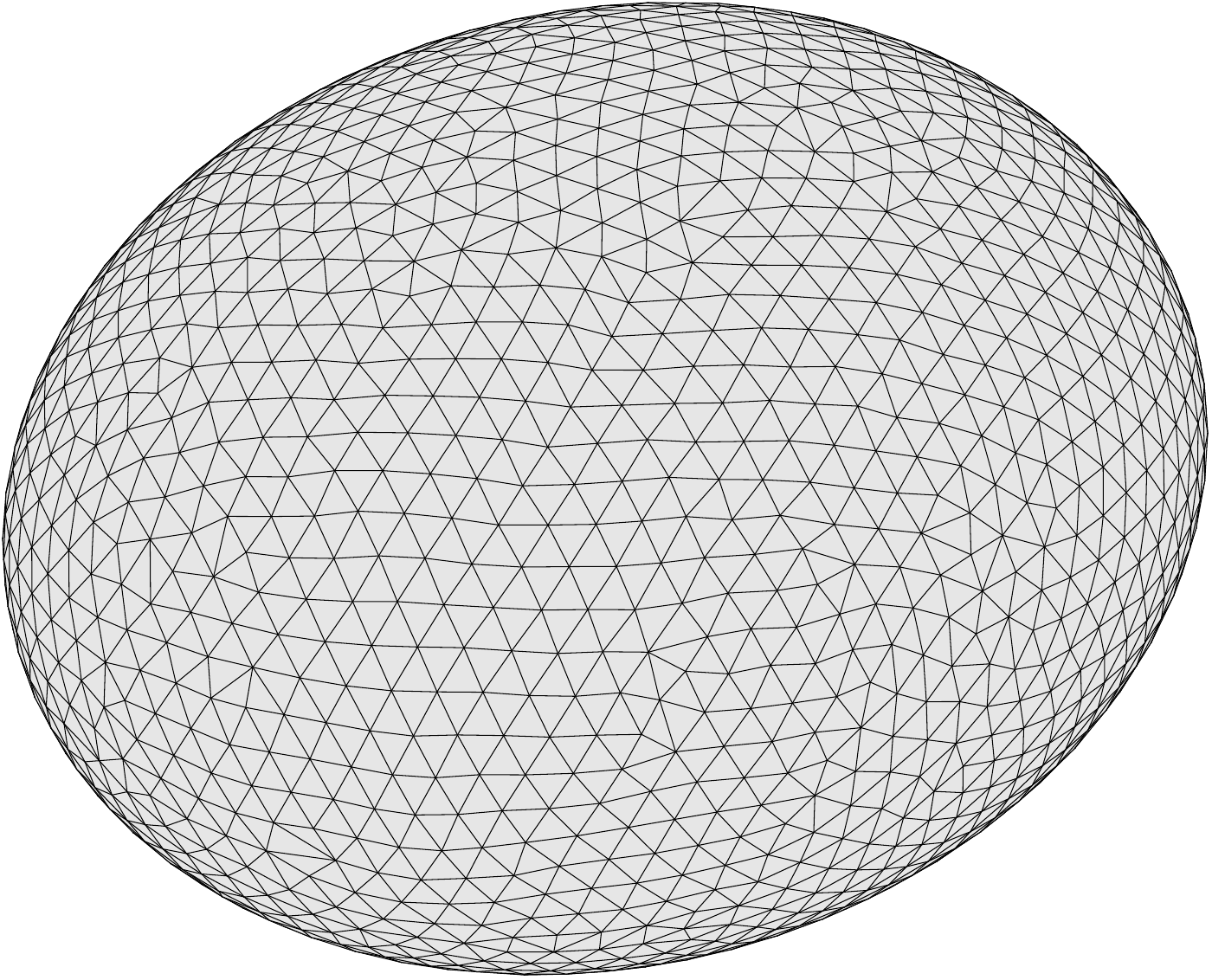}
    \caption{The ellipsoid domain (\ref{eq.ellipsoid})}
    \label{subfig:ellipsoid}
\end{subfigure}
\hfill
\begin{subfigure}{0.32\textwidth}
    \centering
    \includegraphics[width=.6\linewidth]{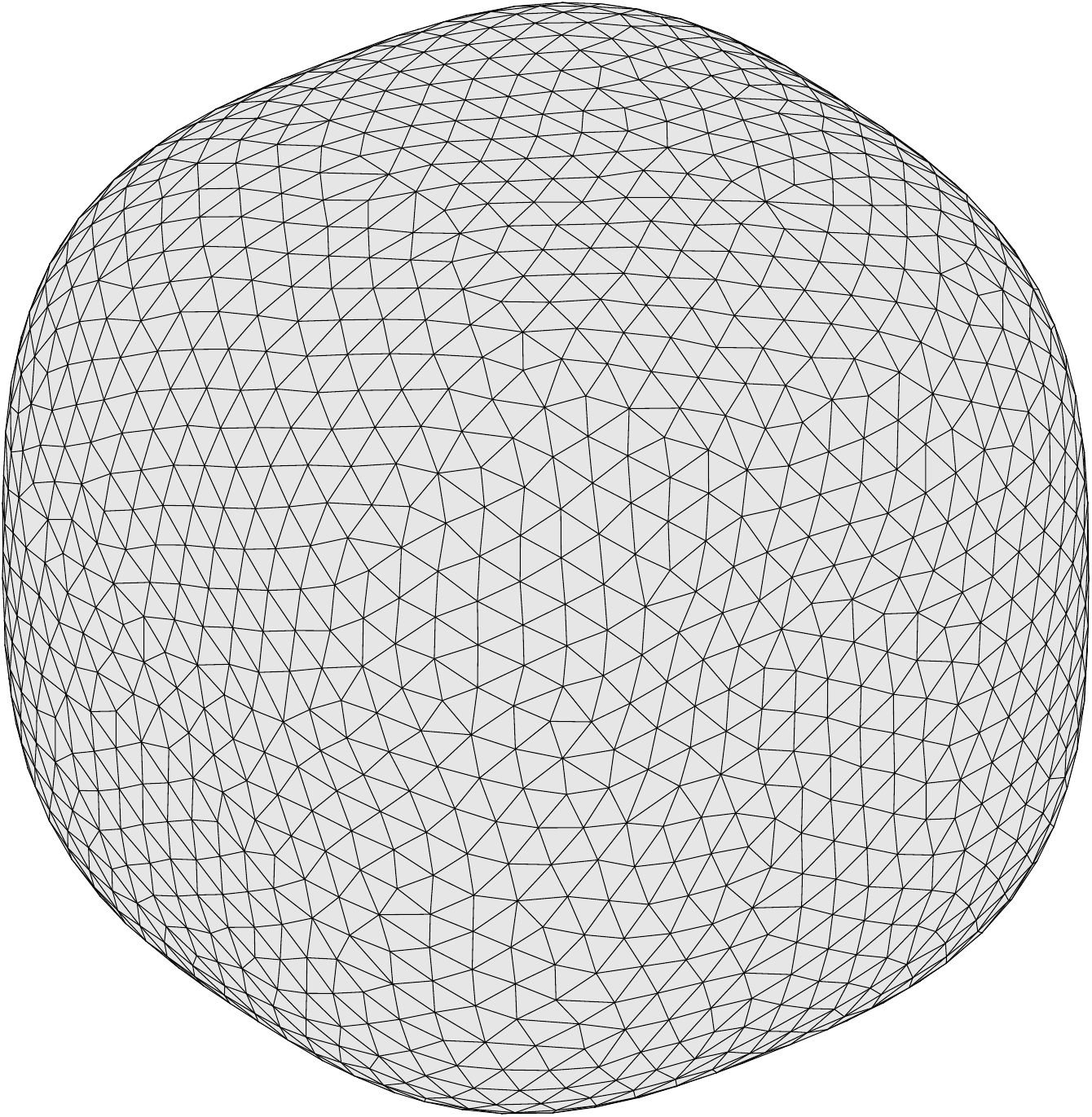}
    \caption{The smoothed cube domain (\ref{eq.smoothedcube})}
    \label{subfig:smoothedcube}
\end{subfigure}

\caption{The domains of the problems and the meshes in $\mathbb{R}^3$ (edge length $h=1/16$)}
\label{grid3d}
\end{figure}

Denote the set of vertices (nodes) of a mesh by 
$V^h=\{v_1,\dots,v_N\}$, where $N$ is the number of vertices. 
A discrete function associated with the mesh is represented by its vector of nodal values $\bm{f}^h\in\mathbb{R}^N$, where the 
component $\bm{f}^h_v$ corresponds to the value at vertex $v \in V^h$.
To measure discrete functions, we define the norm
$$
\|\bm{f}^h\|_h := \Big(\sum_{v\in V^h}|\omega_v|(\bm{f}_v^h)^2\Big)^{{1}/{2}},
$$
where $|\omega_v|$ denotes the volume of the union of elements that share $v$ as a common vertex. 
Similarly, the discrete  $L^2$-norm and $L^\infty$-norm are defined by
$$
\|\bm{f}^h\|_{L^2_h}: = \Big( \frac{1}{d+1}\sum_{v \in V_h} |\omega_v| \, (\bm{f}_v^h)^2 \Big)^{1/2} 
\quad \mbox{and}\quad
\|\bm{f}^h\|_{L^\infty_h}: = \max_{v\in V_h}|f^h_v|,
$$
where $d\in\{2,3\}$ is the spatial dimension. 
Let $\{\bm{u}^h_k\}_{k\ge0}$ be the sequence 
of discrete vectors, where each vector $\bm{u}^h_k \in \mathbb{R}^N$ consists of the nodal values representing the finite element approximation of the function $u_k$ obtained at the $k$-th outer iteration of Algorithm \ref{alg:fix}.
Then we use $D_h^2$ to denote the discrete analog of the Hessian operator $D^2$, which acts on the discrete vector $\bm{u}^h_k$.
Moreover, denote $R^h$ for the discrete analog of the operator $R(\cdot)$ in \eqref{rayli}, and $\{R^h(\bm{u}^h_k)\}_{k\ge0}$ for the corresponding sequence of scalar values generated by  Algorithm \ref{alg:fix}.

\paragraph{Relative errors for termination.}
We define the relative error $\eta_1$ by
$$
\eta_1 := \frac{\| \det D_h^2 \bm{u}^h_k - R^h(\bm{u}^h_k)|\bm{u}^h_k|^d \|_h}{1 + R^h(\bm{u}^h_k)\|(\bm{u}^h_k)^d\|_h},
$$
which measures the accuracy of the numerical solution in dimension $d\in\{2,3\}$.
In addition, to evaluate the sequence 
$\{\bm{u}^h_{k+1}\}_{k\ge0}$ obtained by solving the subproblem \eqref{eq:MAsub} by Cholesky decomposition, 
the following relative error is used:
$$
\eta_2 := \frac{\| \det D_h^2 \bm{u}^h_{k+1} - R^h(\bm{u}^h_k)|\bm{u}^h_k|^d \|_{L^{\infty}_h}}{1 + R^h(\bm{u}^h_k)\|(\bm{u}^h_k)^d\|_{L^{\infty}_h}}.
$$
In our numerical experiments, both the inexact AKI-FP method (Algorithm \ref{alg:fix}) and the original AKI method use the outer loop stopping criterion $\eta_1 < \texttt{tol-outer}$ with $\texttt{tol-outer} = 10^{-6}$.

\paragraph{Initial point and stopping criterion for subproblems.}
A practical approach to obtain the initial function $u_0$ in Algorithm \ref{alg:fix} is first to solve the following Poisson equation:  
\begin{equation}
\label{initial condition}
\begin{cases}
 \Delta u = g &  \mbox{ in } \Omega,\\ 
  u =0 & \mbox{ on } \partial \Omega,
\end{cases}
\end{equation} 
where $g$ is a given constant function.
The initial guess is then refined through several preliminary fixed-point iterations. 
This refinement step aims to drive the initial guess towards a function that numerically exhibits the desired structure in Assumption \ref{ass_blanket1}.

We also employ a warm start strategy to automatically generate an initial point on finer meshes. 
For instance, consider a 2D domain with different element lengths $h\in\{1/20,1/40,1/80,\\1/160\}$. 
On the coarse mesh $h=1/20$, we compute the solution $\bm{u}_*^{{1}/{20}}$ with a given stopping criterion $\eta_1 < 10^{-6}$. 
Then for each finer mesh $h\in\{1/40,1/80,1/160\}$, we first run the algorithm until the outer tolerance reaches $5\times10^{-4}$, and denote the resulting iterate by $\bm{u}_k^h$; we then restart the algorithm with the initial point 
$$
\tilde{\bm{u}}_0^h = \frac{1}{5} \|\bm{u}_*^{{1}/{20}}\|_{{1}/{20}}\cdot \frac{\bm{u}_k^h}{\|\bm{u}_k^h\|_h}.
$$

Algorithm \ref{alg:fix} provides a general framework for choosing the summable tolerance sequence $\{\xi_k\}_{k\geq0}$. 
In our implementation, we adopt the following dynamically adjusted scheme
\begin{equation}
\label{eq:dynamic-adj}
\xi_k :=
\displaystyle  \min\Big( 
\zeta_1 \cdot \xi_{k-1},\  \zeta_2 \cdot \frac{10 \cdot2^{d+l}}{\|(\bm u^{h}_{k-1})^d\|_{L^\infty_h}}\cdot  \eta_{1,k-1} \Big)
\quad 
\mbox{with}\quad 
1>\zeta_1,\zeta_2> 0.
\end{equation}
Here, $\eta_{1,k-1}$ denotes the relative error from the preceding $(k-1)$-th outer step. 
The integer $l\geq1$ denotes the mesh refinement level, corresponding to the sequence of grids in our experiments. 
It ranges from $l=1$ (coarsest) up to $l=4$ (finest) for the 2D cases and up to $l=3$ for the 3D cases.
Since $\zeta_1<1$, one has from \eqref{eq:dynamic-adj} that $0\le \xi_k\le \xi_{k-1}\zeta_1$. 
Therefore, $\{\xi_k\}_{k\geq 0}$ is a summable sequence of nonnegative real numbers. 

The motivation of the dynamic scheme \eqref{eq:dynamic-adj} is to relax the inner tolerance during the early outer iterations, thereby reducing unnecessary inner iterations and improving overall efficiency without compromising accuracy.

\subsection{Numerical experiments in \texorpdfstring{$\mathbb{R}^2$}{num R2}}

In the 2D case, we compare the inexact AKI-FP method (Algorithm~\ref{alg:fix}) with the original AKI method and the operator-splitting approach in \cite{glowinski2020numerical} and \cite{liu2022efficient}. 
For the inexact AKI-FP method, we set $g = 0.5$ in~\eqref{initial condition} and use $(\zeta_1, \zeta_2) = (0.99, 0.9)$ in~\eqref{eq:dynamic-adj} for all meshes.
The original AKI method is tested with two fixed inner tolerances, $\texttt{tol-inner} = 10^{-10}$ or $10^{-6}$, where the inner loop terminates once $\eta_2 \le \texttt{tol-inner}$.

\paragraph{Example 1.} We first consider the MAE problem \eqref{eq_MAeig} on the unit disk domain
\begin{equation}
   	\Omega=\{(x,y)\mid x^2+y^2 < 1\},
	\label{eq.disk}
\end{equation}
which was studied in \cite[Scetion 5.3]{glowinski2020numerical}  and \cite[Example 1]{liu2022efficient}.  
Figure \ref{subfig:disk} illustrates the mesh we used. 
This problem admits a radial solution that can be approximated with high accuracy, making it very suitable for testing the capability of numerical methods for the MAE problem. 
Specifically, let $r:=\sqrt{x^2+y^2}$. 
For a radial function $v(r)$, one has $\det D^2 v=\frac{v'v''}{r}$ to the MAE problem \eqref{eq_MAeig} on the unit disk. 
Substituting this structure into \eqref{eq_MAeig} reduces the MAE problem to the following one-dimensional boundary value problem for $v:=v(r)$:
$$
\begin{cases}
		v'\,v''=\lambda\, r\,v^2\quad \mbox{ in } (0,1),\quad v\leq 0, \ \lambda> 0,\\
		v'(0)=0,\ v(1)=0,\ 
        2\pi \displaystyle\int_0^1 |v|^2\,r\,dr=1.
\end{cases}
$$
By employing the shooting method, a high-precision numerical solution for the above system is obtained as 
$u(0,0)=v(0)\approx-1.0628$ and $\lambda\approx 7.490039$, with the $L^2$-error of $v'v'' - \lambda r v^2$ reaching $ 10^{-6}$.
The maximum number of outer iterations is set to $2000$, as the operator-splitting methods may not always achieve the desired accuracy.
The numerical results for a sequence of refined meshes ($h = 1/20,1/40,1/80,1/160$) are summarized in Figure \ref{fig.disk}, Table \ref{Tab.disk3}, and Table \ref{Tab.disk1}.

\begin{figure}[H]
\centering
\begin{subfigure}{0.49\textwidth}
    \centering
    \includegraphics[width=0.7\linewidth]{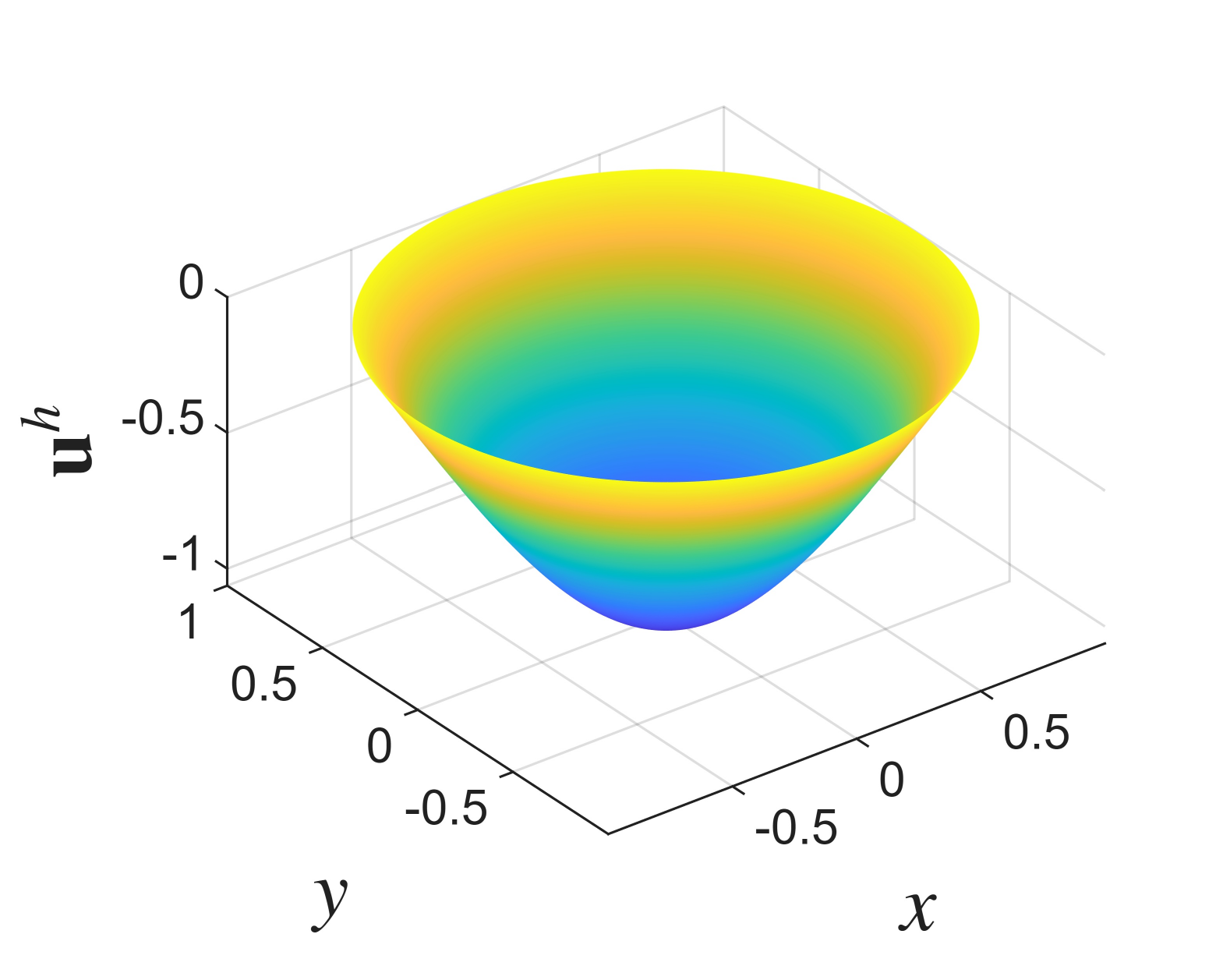}
    \caption{Solution profile with $h=1/80$}
    \label{subfig:disk-solution}
\end{subfigure}
\hfill
\begin{subfigure}{0.49\textwidth}
    \centering
    \includegraphics[width=0.7\linewidth]{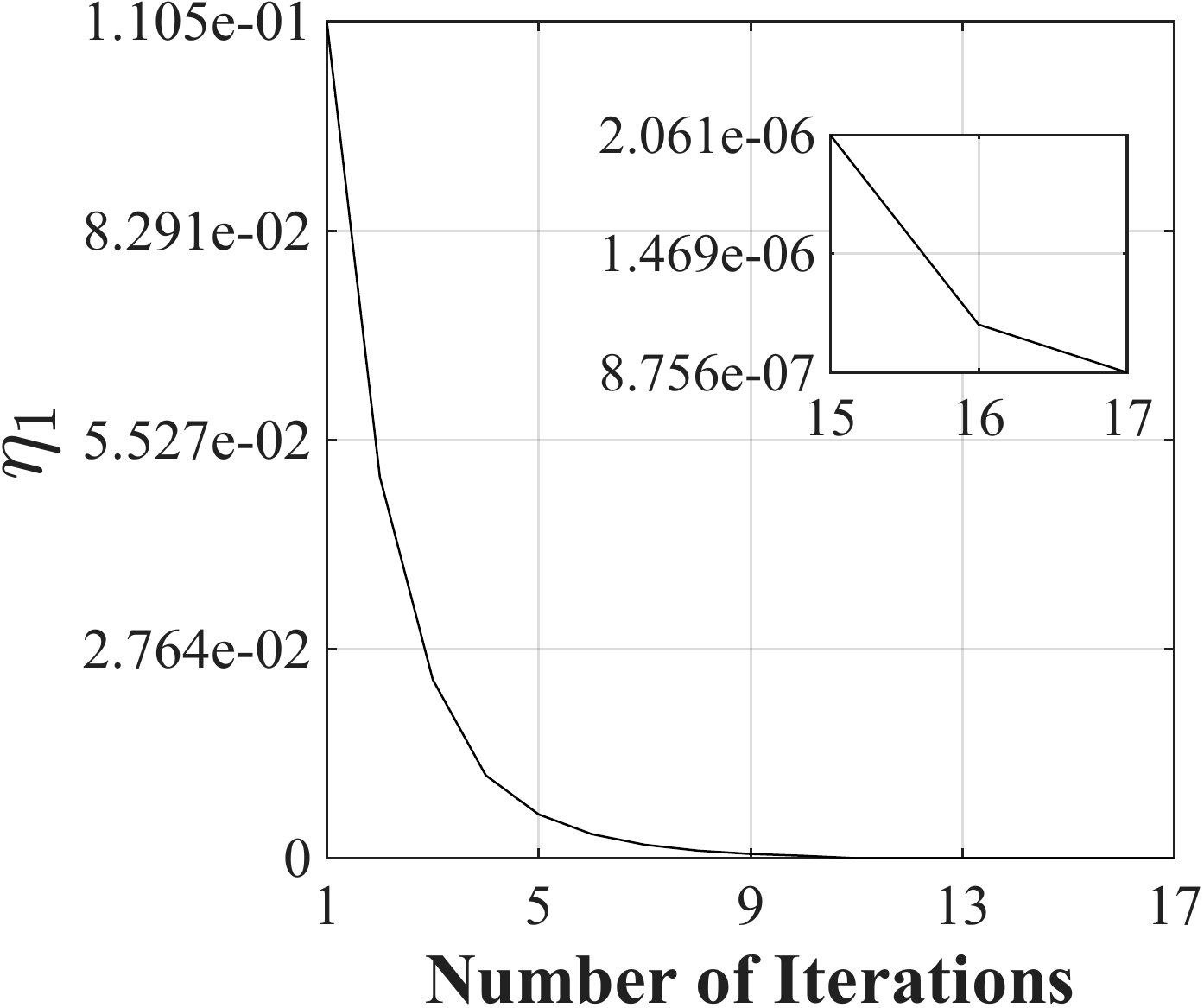}
    \caption{Convergence of $\eta_{1}$ with $h=1/80$}
    \label{subfig:disk-convergence}
\end{subfigure}

\vspace{2mm}

\begin{subfigure}{0.49\textwidth}
    \centering
    \includegraphics[width=0.7\linewidth]{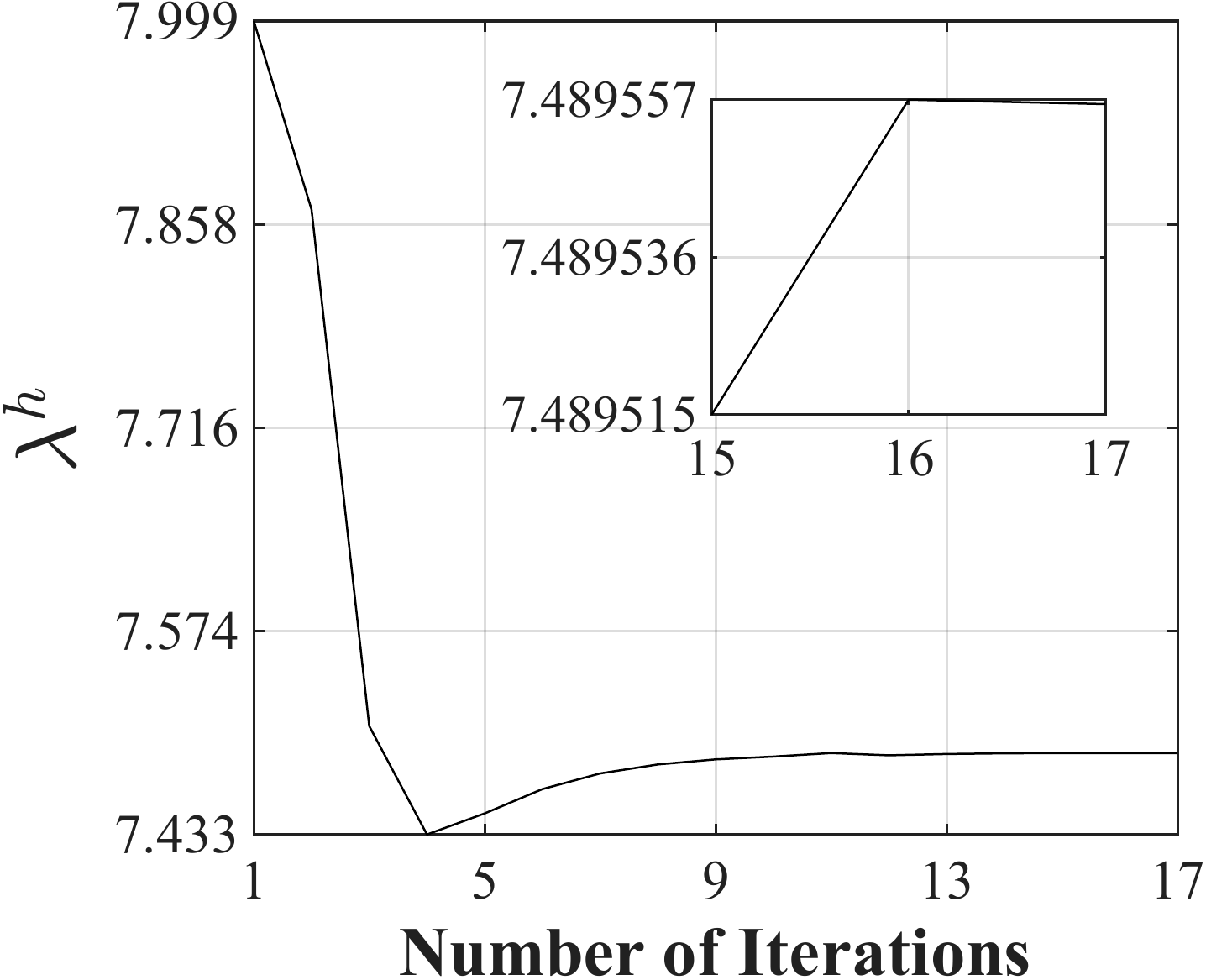}
    \caption{Value of $\lambda^h$ versus iteration $k$ with $h=1/80$}
    \label{subfig:disk-lambda}
\end{subfigure}
\hfill
\begin{subfigure}{0.49\textwidth}
    \centering
    \includegraphics[width=0.7\linewidth]{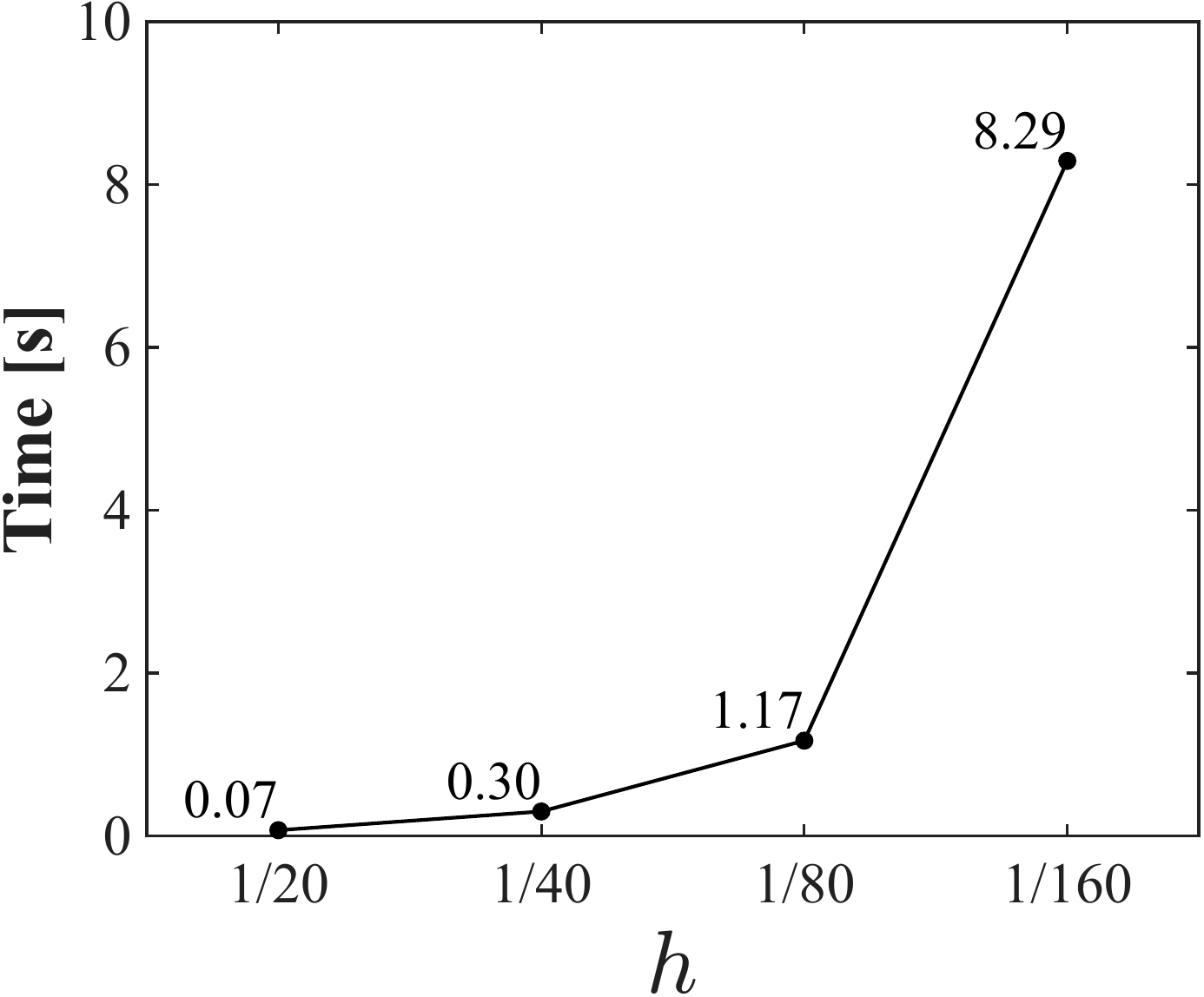}
    \caption{Computational time under $h$}
    \label{subfig:disk-time}
\end{subfigure}

\caption{Performance of the inexact AKI-FP method on the unit disk domain \eqref{eq.disk}}
\label{fig.disk}
\end{figure}

{
Figure \ref{fig.disk} illustrates the detailed results for the inexact AKI-FP method. 
One can observe that the computed solution profile for $h=1/80$, shown in Figure \ref{subfig:disk-solution}, is smooth as expected. 
Figure \ref{subfig:disk-convergence} and Figure \ref{subfig:disk-lambda} show the convergence of the relative error $\eta_1$ and the computed eigenvalue estimate $\lambda^h$ versus the iteration count $k$, respectively, obtained with the mesh parameter $h=1/80$. 
One can see that $\eta_1$ decreases to the tolerance of $10^{-6}$, while the sequence $\lambda^h$ also converges to the corresponding discrete eigenvalue $\lambda\approx 7.490039$. 
Figure \ref{subfig:disk-time} illustrates the relationship between the computation time of the inexact AKI-FP method and the mesh parameter $h$. 
With $h$ decreasing, the observed convergence rate for the $L^2$-error approaches $2$, while the rate for the $H^1$-error approaches $1$ (see Table \ref{Tab.disk3}).

\begin{table}
\centering
\caption{The convergence behavior under different norms of the inexact AKI-FP method on the unit disk domain \eqref{eq.disk}.}
\label{Tab.disk3}
\begin{tabular}{@{}cccccccc@{}}
\toprule
$h$   & \multicolumn{2}{c} {$L^2$-error}& rate &\multicolumn{2}{c}{$H^{1}$-error}&rate \\ 
\midrule
 1/20  & \multicolumn{2}{c}{$2.28\text{E-}{04}$} & -  & \multicolumn{2}{c}{$2.34\text{E-}{03}$}     &  - \\ 
1/40  & \multicolumn{2}{c}{$6.17\text{E-}{05}$} &  1.89 &\multicolumn{2}{c}{$1.16\text{E-}{03}$} &1.01 \\ 
 1/80   & \multicolumn{2}{c}{$1.21\text{E-}{05}$} &  2.35    & \multicolumn{2}{c}{$4.33\text{E-}{04}$} &1.42 \\ 
1/160  &  \multicolumn{2}{c}{$3.08\text{E-}{06}$}& 1.97 &\multicolumn{2}{c}{$1.80\text{E-}{04}$}      &1.27\\ 
\bottomrule
\end{tabular}
\end{table}

\begin{table}[ht]
\centering
\caption{Comparison on the numerical performance of the inexact AKI-FP method (``iAKI-FP''), the original AKI method (``AKI''), and the operator-splitting methods (``Split'') for solving the MAE problem \eqref{eq_MAeig} on the unit disk domain \eqref{eq.disk}. 
In the table, ``Iter'' denotes the number of outer iterations; 
``sub-Iter'' denotes the total number of subproblems solved during the algorithm's execution; 
``$\eta_1$'' denotes the outer-loop relative residual; 
``$\lambda^h$'' denotes the computed discrete eigenvalue; 
and
``$\min(\bm u^h)$'' denotes the minimum nodal value of the normalized computed solution with $\|\bm u^h\|_{L^2_h} = 1$.}
\label{Tab.disk1}
\begin{tabular}{@{}ccccccccc@{}}
\toprule
Algorithm & $h$ & Iter & sub-Iter & \multicolumn{2}{c}{$\eta_1$} & $\lambda^h$ & $\min(\bm{u}^h)$ & Time [s]\\
\midrule
iAKI-FP & 1/20 & 45 & 45 & \multicolumn{2}{c}{$9.85\text{E-}{07}$} & 7.481792881 & -1.061907028 & $\bm{0.07}$ \\
AKI-$10^{-6}$ & 1/20 & 9 & 92 & \multicolumn{2}{c}{$6.92\text{E-}{07}$} & 7.481792883 & -1.061906215 & 0.08 \\
AKI-$10^{-10}$ & 1/20 & 9 &417 & \multicolumn{2}{c}{$1.52\text{E-}{07}$} & 7.481792887 & -1.061906844 & 0.23 \\
Split & 1/20 & 2000 & 2000 & \multicolumn{2}{c}{$4.04\text{E-}{01}$} & 7.425159943 & -1.029582001 &73.50 \\
Split in \cite{liu2022efficient}  & 1/20 & 13 & 13 & \multicolumn{2}{c}{-} & 5.9716 & -1.0189 & - \\
Split in \cite{glowinski2020numerical}  & 1/20 & 312 & - & \multicolumn{2}{c}{-} & 6.13 & - & - \\
\midrule
iAKI-FP & 1/40 & 28 & 27 & \multicolumn{2}{c}{$9.49\text{E-}{07}$} & 7.488102312 & -1.062780396 & $\bm{0.30}$ \\
AKI-$10^{-6}$ & 1/40 & 9 & 117 & \multicolumn{2}{c}{$9.73\text{E-}{07}$} & 7.488103591 & -1.062755623 & 0.55 \\
AKI-$10^{-10}$ & 1/40 & 9 & 462 & \multicolumn{2}{c}{$6.53\text{E-}{07}$} & 7.488103584 & -1.062763767 & 1.40 \\
Split & 1/40 & 2000 & 2000 & \multicolumn{2}{c}{$6.50\text{E-}{01}$} & 7.603716691 & -1.042761461 & 315.30 \\
Split in \cite{liu2022efficient}  & 1/40 & 13 & 13 & \multicolumn{2}{c}{-} & 6.6656 & -1.0362 & - \\
Split in \cite{glowinski2020numerical} & 1/40 & 976 & - & \multicolumn{2}{c}{-} & 6.81 & - & - \\
\midrule
iAKI-FP & 1/80 & 17 & 31 & \multicolumn{2}{c}{$8.76\text{E-}{07}$} & 7.489556059  &  -1.062756964 & $\bm{1.17}$ \\
AKI-$10^{-6}$ & 1/80 & 9 & 168 & \multicolumn{2}{c}{$9.71\text{E-}{07}$} & 7.489560784 & -1.062735028 & 3.28 \\
AKI-$10^{-10}$ & 1/80 & 9 & 732 & \multicolumn{2}{c}{$6.52\text{E-}{07}$} & 7.489560774 &  -1.062743148 & 12.86 \\
Split & 1/80 & 2000 & 2000 & \multicolumn{2}{c}{$9.68\text{E-}{01}$} & 7.588529323 & -1.052027837 & 1094.55\\
Split in \cite{liu2022efficient}  & 1/80 & 13 & 13 & \multicolumn{2}{c}{-} & 7.0655 & -1.0484 & - \\
Split in \cite{glowinski2020numerical}  & 1/80 & 1017 & - & \multicolumn{2}{c}{-} & 7.12 & - & - \\
\midrule
iAKI-FP & 1/160 & 15 & 55 & \multicolumn{2}{c}{$3.82\text{E-}{07}$} & 7.489922294 & -1.062756210  & $\bm{8.29}$ \\
AKI-$10^{-6}$ & 1/160 & 9 & 260 &  \multicolumn{2}{c}{$9.71\text{E-}{07}$} & 7.489922511 & -1.062737279 &  29.30 \\
AKI-$10^{-10}$ & 1/160 & 9 & 1308 & \multicolumn{2}{c}{$6.51\text{E-}{07}$} & 7.489922500 & -1.062745394  & 138.82 \\
Split & 1/160 & 2000 & 2000 & \multicolumn{2}{c}{$1.35\text{E-}{00}$} & 7.546402832 & -1.057404781 & 4456.13 \\
Split in \cite{liu2022efficient}  & 1/160 & 13 & 13 & \multicolumn{2}{c}{-} & 7.2816 & -1.0556 & - \\
Split in \cite{glowinski2020numerical}  & 1/160 & 3315 & - & \multicolumn{2}{c}{-} & 7.32 & - & - \\
\bottomrule
\end{tabular}
\end{table}

Table \ref{Tab.disk1} compares the inexact AKI-FP method with the original AKI method and the operator-splitting methods. 
As highlighted in \cite{liu2022efficient}, its approach to implementing the operator-splitting method achieves faster convergence than the splitting method in \cite{glowinski2020numerical}, while providing solutions with a comparable accuracy.
Therefore, we only implement the operator-splitting method from \cite{liu2022efficient} for comparison.
In our implementation, we adopt the parameter settings from \cite{liu2022efficient} with one exception: we set $\tau = h^2$ rather than $h/2$.
This change is motivated by our observation that a large value of $\tau$ can cause the algorithm to terminate quickly under the stopping criterion used in \cite{liu2022efficient} (the difference between successive iterates $\|\bm{u}^h_{k+1}-\bm{u}^h_{k}\|_{L^2_h}$), leading to a less accurate eigenvalue.
To properly stop the operator-splitting methods, we set the maximum number of outer iterations as $2000$ for all the algorithms.
The ``Split'' row reports the numerical results of the operator-splitting method we implemented. 
The row labeled ``Split in \cite{liu2022efficient}'' and ``Split in \cite{glowinski2020numerical}'' comes directly from the published results in \cite[Example~1]{liu2022efficient} and \cite[Section 5.3]{glowinski2020numerical} for a comparison of accuracy.  

According to Table \ref{Tab.disk1}, the inexact AKI-FP method and the original AKI method produce nearly identical eigenvalue approximations ($\lambda^h$) and minimum values ($\min(\bm{u}^h)$) across all $h$. 
The computed eigenvalues are more accurate than those from the implemented operator-splitting method, and significantly better than the ones reported in ``Split in \cite{liu2022efficient}'' and ``Split in \cite{glowinski2020numerical}''. 
Notably, the inexact AKI-FP method achieves the given tolerance with significantly less computational cost, requiring approximately $1/17$ of the time needed by the original AKI method ($\texttt{tol-inner}\equiv10^{-10}$) on the finest mesh $h=1/160$. 
Moreover, the inexact AKI-FP method with the dynamic adjustment \eqref{eq:dynamic-adj} demonstrates superior numerical performance over that with a fixed tolerance of $10^{-6}$.
As shown in the ``sub-Iter'' column of Table \ref{Tab.disk1}, the improvement in computational efficiency of the inexact AKI  method is primarily attributed to the significantly smaller number of Poisson equations that need to be solved, compared to the original AKI method. 
Given the observed performance of the operator-splitting methods, we decided to compare only with the original AKI method for the subsequent problems.
}

\paragraph{Example 2.} Consider the MAE problem \eqref{eq_MAeig} on the ellipse domain
\begin{equation}
   		\Omega=\left\{(x,y)\mid x^2+2y^2<1\right\}.
	\label{eq.ellipse}
\end{equation}

Figure \ref{subfig:ellipse} illustrates the mesh we used.

Similar to Example 1, Figure \ref{fig.ell} presents the detailed results for the inexact AKI-FP method, with the solution displayed in Figure \ref{subfig:solution-profile}, where it is evident that the solution is smooth. 
For $h=1/80$, the convergence of the relative error $\eta_1$ and the computed eigenvalues $\lambda^h$ are shown in Figures \ref{subfig:convergence-eta} and \ref{subfig:lambda-iteration}, respectively. 
It is clear that the relative error $\eta_1$ decreases rapidly to below $10^{-6}$, and the eigenvalues $\lambda^h$ converge quickly. Moreover, Figure \ref{subfig:computation-time} illustrates the relationship between the computation time of the inexact AKI-FP method and the mesh parameter $h$.

Table \ref{Tab.ell} shows the computational results of the inexact AKI-FP method with a dynamically adjusted tolerance \eqref{eq:dynamic-adj} and the original AKI method on the ellipse domain \eqref{eq.ellipse}.  From Table \ref{Tab.ell}, it can be observed that for all tested mesh sizes $h$, both methods produce nearly identical converged eigenvalues $\lambda^h$ and a minimum value $\min(\bm{u}^h)$. 
It can be seen that the computational speed of the inexact AKI-FP method is more than $28$ times faster than the AKI method ($\texttt{tol-inner}\equiv10^{-10}$) when $h=1/160$.  As shown in the ``sub-Iter'' column of Table \ref{Tab.ell}, the number of Poisson equations solved by both methods reveals why the inexact AKI-FP method is faster.

\paragraph{Example 3. } Consider the MAE problem \eqref{eq_MAeig} on the smoothed square domain
\begin{equation}
   		\Omega=\left\{(x,y)\mid |x|^{3}+|y|^{3}< 1\right\}.
	\label{eq.smoothsq}
\end{equation}
Figure \ref{subfig:smoothsq} illustrates the mesh we used. 
The numerical results shown in Figure \ref{fig.smoothsq} are similar to those of the previous two examples. 
Specifically, Figure \ref{subfig:smoothsq-solution} represents the profile of the numerical solution.  
Figure \ref{subfig:smoothsq-convergence} illustrates the downward trajectory of the relative error $\eta_1$, while Figure \ref{subfig:smoothsq-lambda} presents the converging path of the computed eigenvalues $\lambda^h$. Figure \ref{subfig:smoothsq-time} shows how the computation time of the inexact AKI-FP method changes as the mesh parameter $h$ varies.

\begin{figure}
\centering
\begin{subfigure}{0.49\textwidth}
    \centering
    \includegraphics[width=.7\linewidth]{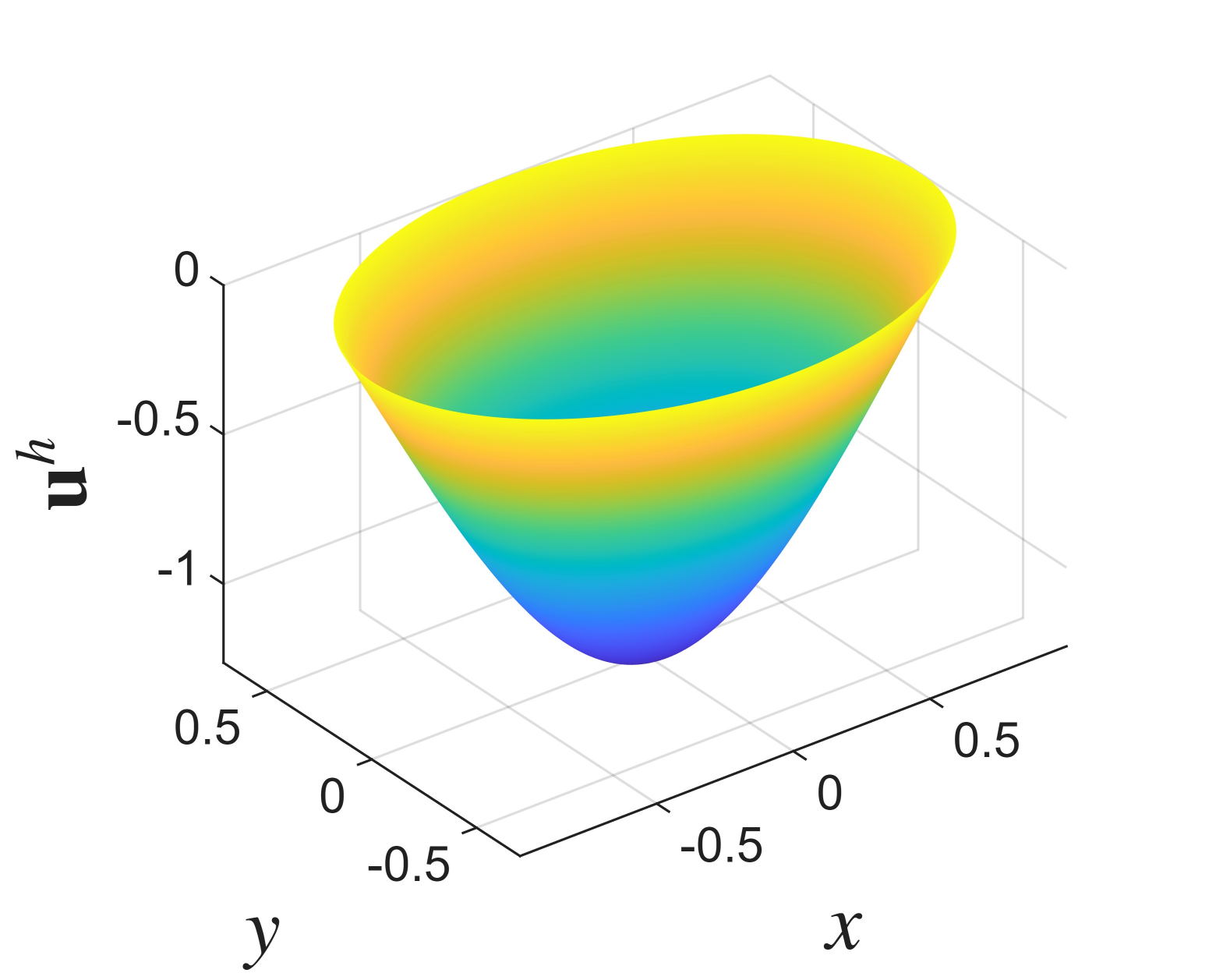}
    \caption{Solution profile with $h=1/80$}
    \label{subfig:solution-profile}
\end{subfigure}
\hfill
\begin{subfigure}{0.49\textwidth}
    \centering
    \includegraphics[width=.7\linewidth]{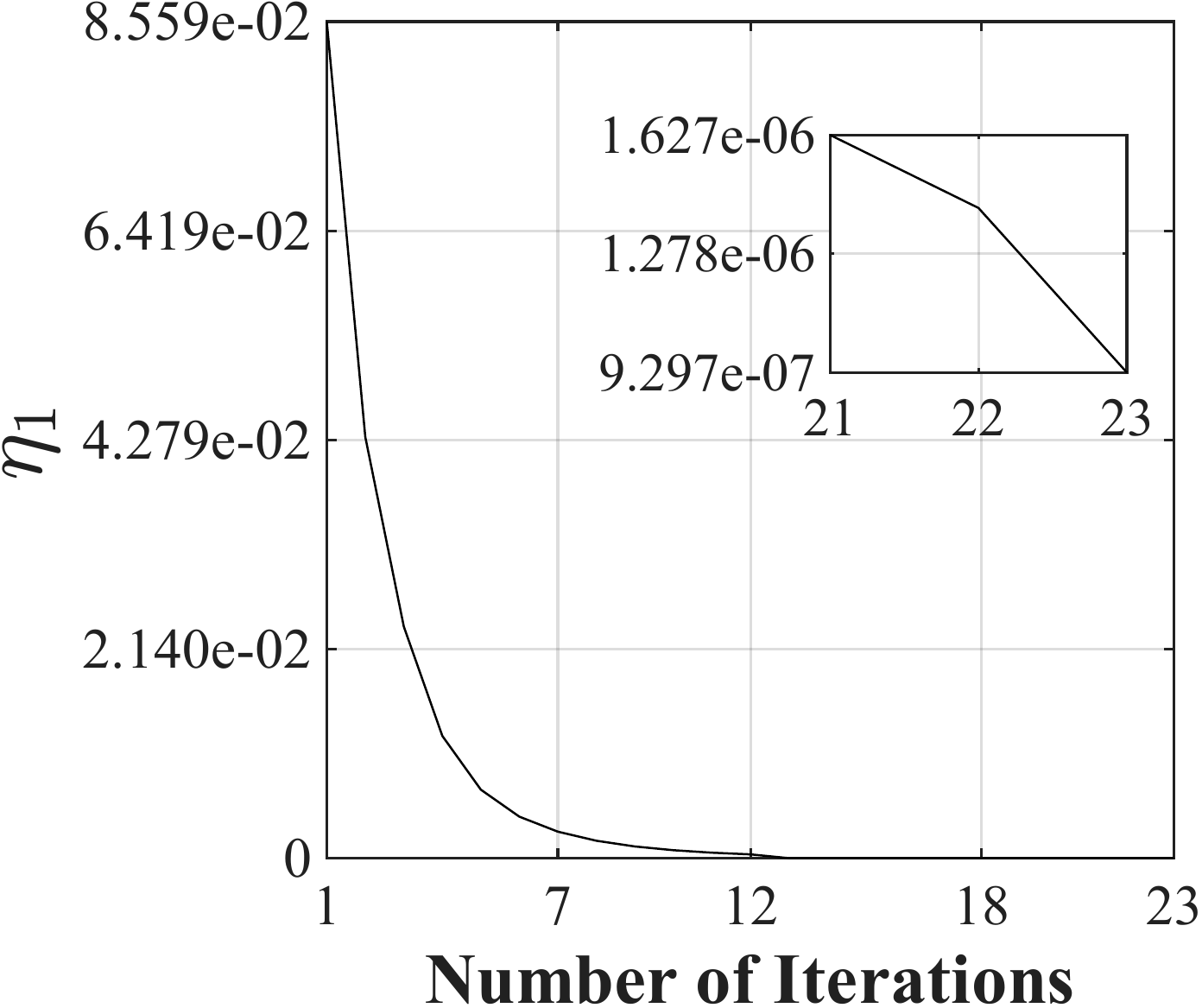}
    \caption{Convergence of $\eta_{1}$ with $h=1/80$}
    \label{subfig:convergence-eta}
\end{subfigure}

\vspace{2mm}

\begin{subfigure}{0.49\textwidth}
    \centering
    \includegraphics[width=.7\linewidth]{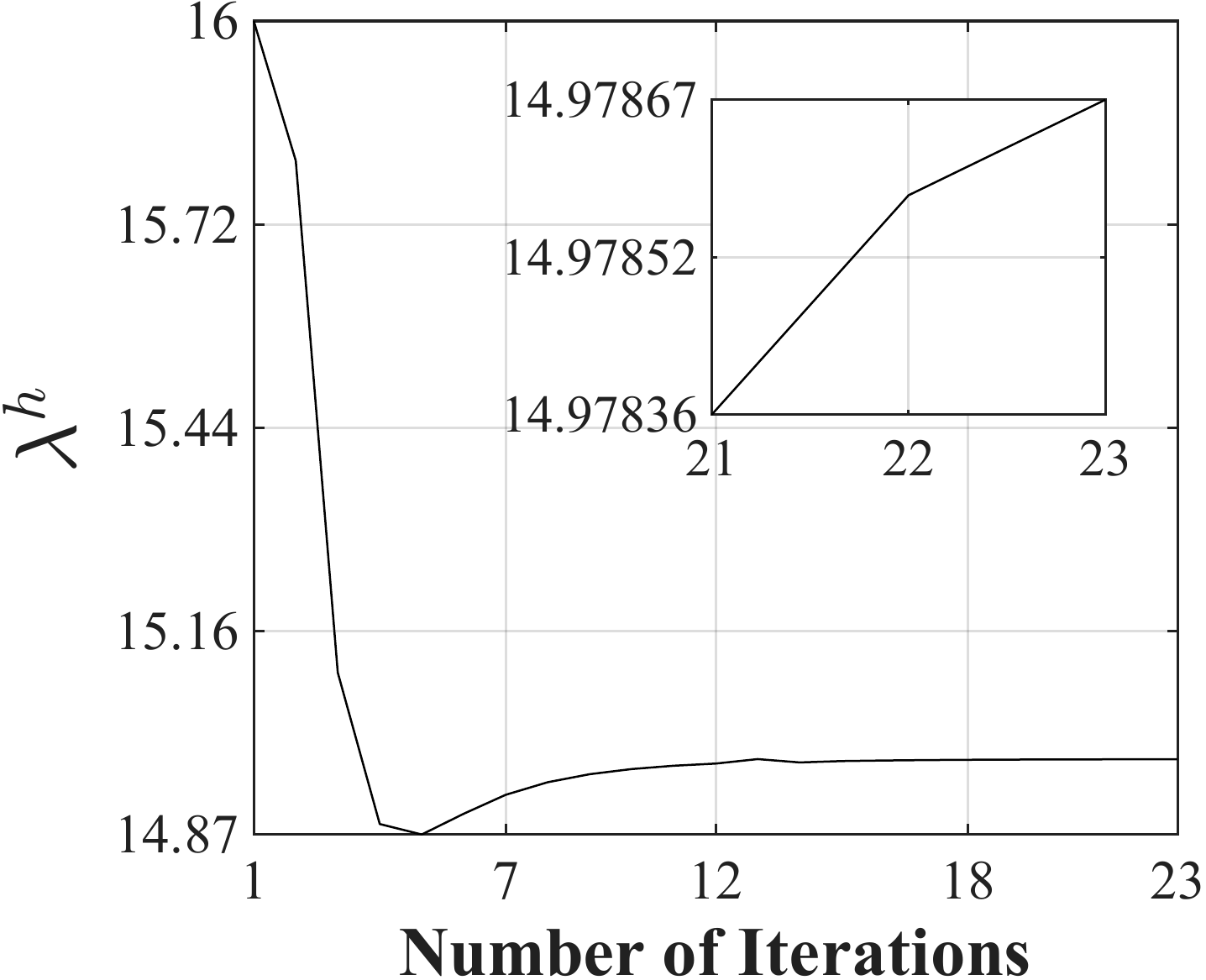}
    \caption{The value of $\lambda^h$ versus iteration $k$ with $h=1/80$}
    \label{subfig:lambda-iteration}
\end{subfigure}
\hfill
\begin{subfigure}{0.49\textwidth}
    \centering
    \includegraphics[width=.7\linewidth]{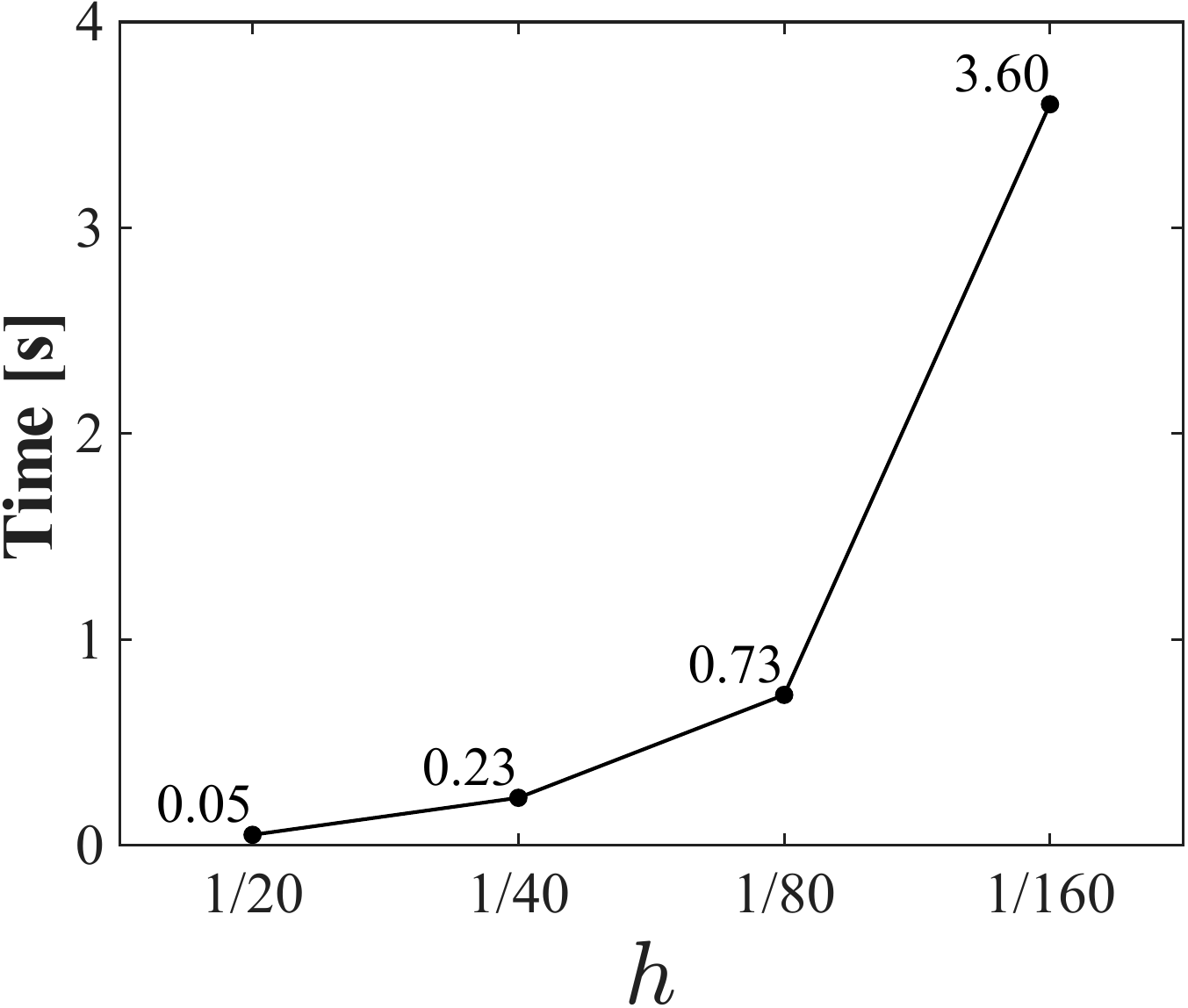}
    \caption{Computational time for different values of $h$}
    \label{subfig:computation-time}
\end{subfigure}

\caption{Performance of the inexact AKI-FP method on the ellipse domain \eqref{eq.ellipse}}
\label{fig.ell}
\end{figure}

\begin{table}
\centering
\caption{Numerical performances of the inexact AKI-FP method and the original AKI method for solving the MAE problem  \eqref{eq_MAeig} on the ellipse domain  \eqref{eq.ellipse}.  The notation is identical to Table~\ref{Tab.disk1}.}
\label{Tab.ell}
\begin{tabular}{@{}ccccccccc@{}}
\toprule
Algorithm & $h$ & Iter & sub-Iter & \multicolumn{2}{c}{$\eta_1$} & $\lambda^h$ & $\min(\bm{u}^h)$ & Time [s]\\
\midrule
iAKI-FP & 1/20 & 49 & 49 & \multicolumn{2}{c}{$9.17\text{E-}{07}$} & 14.956160794 & -1.262390481 & $\bm{0.05}$ \\
AKI-$10^{-6}$ & 1/20 & 9 & 119 & \multicolumn{2}{c}{$8.04\text{E-}{07}$} & 14.956160700 & -1.262389676 & 0.06 \\
AKI-$10^{-10}$ & 1/20 & 9 & 507 &  \multicolumn{2}{c}{$2.85\text{E-}{07}$} & 14.956160738 & -1.262390220 &  0.19 \\
\midrule
iAKI-FP & 1/40 & 27 & 26 &  \multicolumn{2}{c}{$9.41\text{E-}{07}$}& 14.974024191 & -1.263622632 & $\bm{0.23}$ \\
AKI-$10^{-6}$ & 1/40 & 10 & 158 & \multicolumn{2}{c}{$7.02\text{E-}{07}$} & 14.974065618 & -1.263608170 & 0.41 \\
AKI-$10^{-10}$ & 1/40 & 10 & 647 & \multicolumn{2}{c}{$1.88\text{E-}{07}$} & 14.974065735 & -1.263621594 & 1.19 \\
\midrule
iAKI-FP & 1/80 & 23 & 31 & \multicolumn{2}{c}{$9.30\text{E-}{07}$} & 14.978673288 & -1.263867192 & $\bm{0.73}$ \\
AKI-$10^{-6}$ & 1/80 & 10 & 230 & \multicolumn{2}{c}{$7.00\text{E-}{07}$} & 14.978703084 & -1.263851579 & 2.49 \\
AKI-$10^{-10}$ & 1/80 & 10 & 1013 & \multicolumn{2}{c}{$1.87\text{E-}{07}$} &  14.978703070 & -1.263864930 & 9.69 \\
\midrule
iAKI-FP & 1/160 & 18& 33 &  \multicolumn{2}{c}{$7.95\text{E-}{07}$} &  14.979722702 & -1.263848412 & $\bm{3.60}$ \\
AKI-$10^{-6}$ & 1/160 & 10 & 368 & \multicolumn{2}{c}{$7.00\text{E-}{07}$} & 14.979734343 & -1.263842836 & 22.06 \\
AKI-$10^{-10}$ & 1/160 & 10 & 1767 &  \multicolumn{2}{c}{$1.86\text{E-}{07}$} & 14.979734304 & -1.263856170 & 100.38 \\
\bottomrule
\end{tabular}
\end{table}

\begin{figure}
\centering
\begin{subfigure}{0.49\textwidth}
    \centering
    \includegraphics[width=0.7\linewidth]{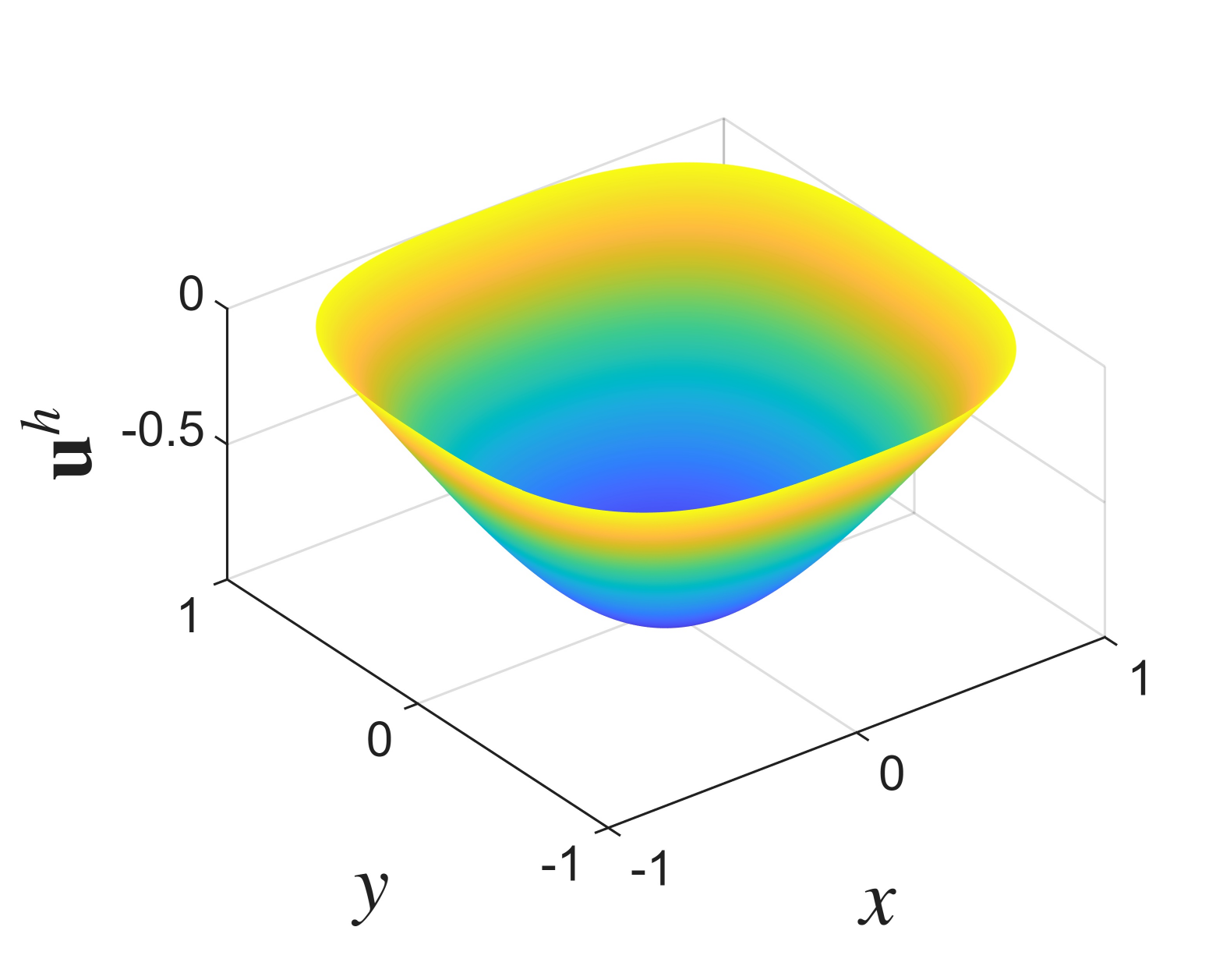}
    \caption{Solution profile with $h=1/80$}
    \label{subfig:smoothsq-solution}
\end{subfigure}
\hfill
\begin{subfigure}{0.49\textwidth}
    \centering
    \includegraphics[width=0.7\linewidth]{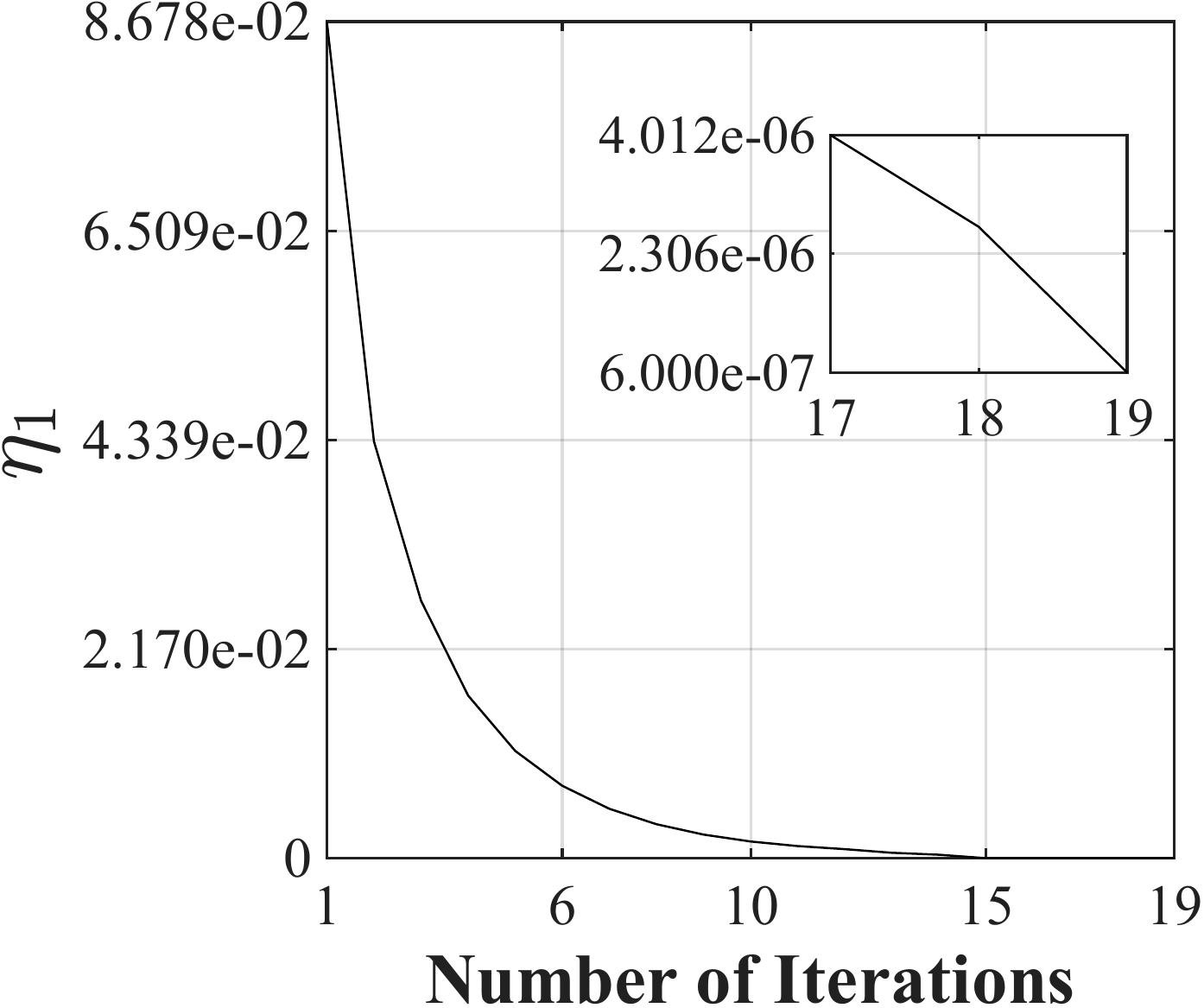}
    \caption{Convergence of $\eta_{1}$ with $h=1/80$}
    \label{subfig:smoothsq-convergence}
\end{subfigure}

\vspace{2mm}

\begin{subfigure}{0.49\textwidth}
    \centering
    \includegraphics[width=0.7\linewidth]{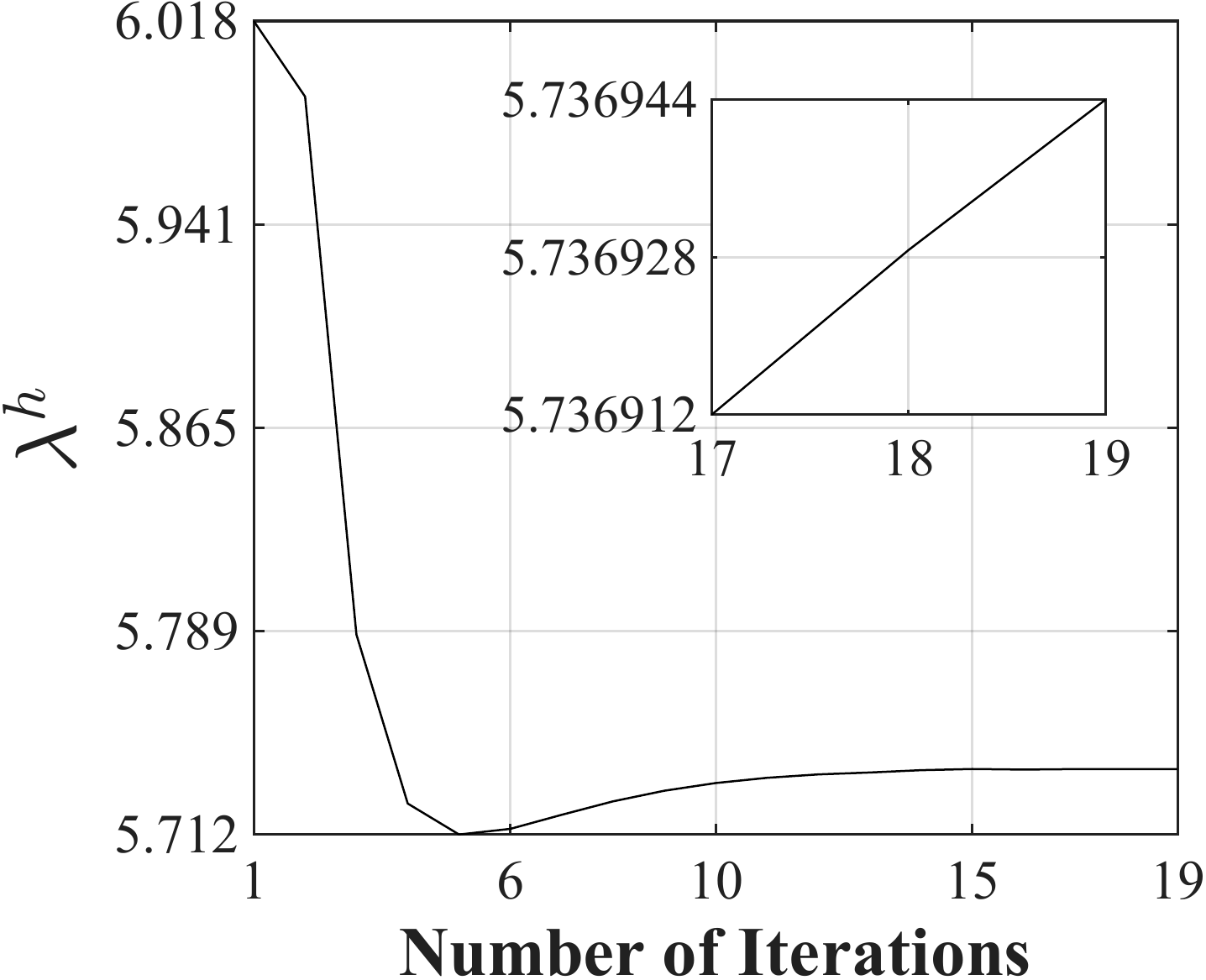}
    \caption{The value of $\lambda^h$ versus iteration $k$ with $h=1/80$}
    \label{subfig:smoothsq-lambda}
\end{subfigure}
\hfill
\begin{subfigure}{0.49\textwidth}
    \centering
    \includegraphics[width=0.7\linewidth]{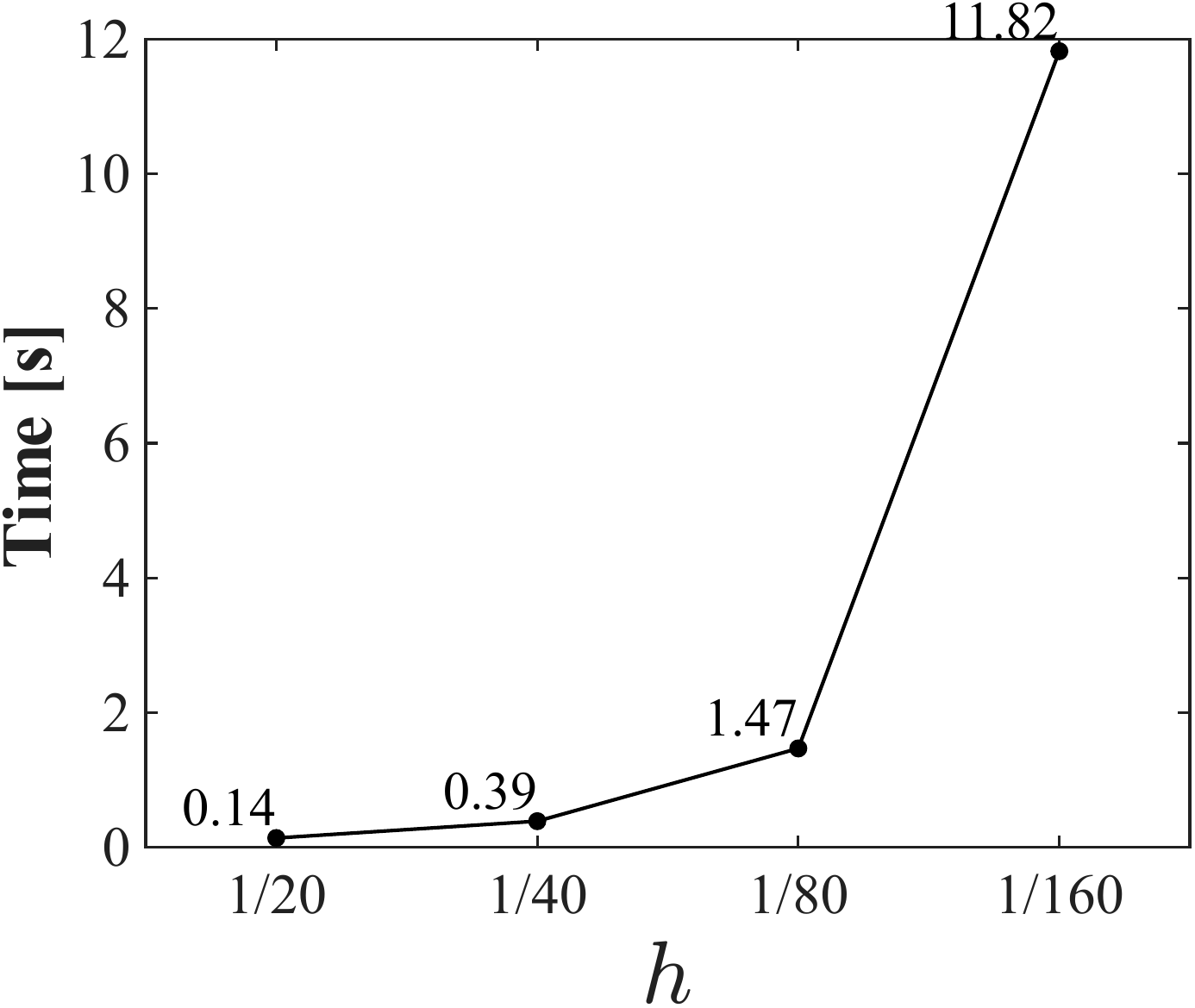}
    \caption{Computational time for different values of $h$}
    \label{subfig:smoothsq-time}
\end{subfigure}

\caption{Performance of the inexact AKI-FP method on the smoothed square \eqref{eq.smoothsq}}
\label{fig.smoothsq}
\end{figure}

\begin{table}
\centering
\caption{Numerical performances  of the inexact AKI-FP method and the original AKI method for solving the MAE problem  \eqref{eq_MAeig} on the smoothed square domain 
 \eqref{eq.smoothsq}. The notation is identical to Table \ref{Tab.disk1}.}
\label{Tab.smooth}
\begin{tabular}{@{}ccccccccc@{}}
\toprule
Algorithm & $h$ & Iter & sub-Iter & \multicolumn{2}{c}{$\eta_1$} & $\lambda^h$ & $\min(\bm{u}^h)$ & Time [s]\\
\midrule
iAKI-FP & 1/20 & 53 & 53 & \multicolumn{2}{c}{$9.52\text{E-}{07}$} & 5.733308886 & -0.996821459 & $\bm{0.14}$ \\
AKI-$10^{-6}$ & 1/20 & 10 & 119 & \multicolumn{2}{c}{$7.14\text{E-}{07}$} & 5.733308844 & -0.996820961 & 0.18\\
AKI-$10^{-10}$ & 1/20 & 9 & 507 & \multicolumn{2}{c}{$1.90\text{E-}{07}$} & 5.733308856 & -0.996821244 & 0.43 \\
\midrule
iAKI-FP & 1/40 & 22 &  41 &  \multicolumn{2}{c}{$5.77\text{E-}{07}$} &  5.736095201 &  -0.996759486 & $\bm{0.39}$ \\
AKI-$10^{-6}$ & 1/40 & 11 & 143 & \multicolumn{2}{c}{$7.60\text{E-}{07}$} & 5.736096366 & -0.996751821 &  0.68 \\
AKI-$10^{-10}$ & 1/40 & 9 & 489 & \multicolumn{2}{c}{$9.09\text{E-}{07}$} & 5.736096364 & -0.996739596 & 1.69 \\
\midrule
iAKI-FP & 1/80 & 19 & 40 & \multicolumn{2}{c}{$6.00\text{E-}{07}$} &  5.736943829 &  -0.996769700 & $\bm{1.47}$ \\
AKI-$10^{-6}$ & 1/80 & 11 & 184 &  \multicolumn{2}{c}{$7.59\text{E-}{07}$} & 5.736944292 & -0.996762487 & 4.15 \\
AKI-$10^{-10}$ & 1/80 & 9 & 702 & \multicolumn{2}{c}{$9.08\text{E-}{07}$} & 5.736944292 & -0.996750288 & 14.16 \\
\midrule
iAKI-FP & 1/160 & 17 & 73 &  \multicolumn{2}{c}{$1.27\text{E-}{07}$} &  5.737144091 &  -0.996812285  & $\bm{11.82}$ \\
AKI-$10^{-6}$ & 1/160 & 11 & 244 & \multicolumn{2}{c}{$7.59\text{E-}{07}$} & 5.737144110 & -0.996805393 & 32.43\\
AKI-$10^{-10}$ & 1/160 & 9 & 1016 & \multicolumn{2}{c}{$9.08\text{E-}{07}$} & 5.737144114 & -0.996793198 &  125.28  \\
\bottomrule
\end{tabular}
\end{table}

 Table \ref{Tab.smooth} shows the computational results of the inexact AKI-FP method with a dynamically adjusted tolerance \eqref{eq:dynamic-adj}, as well as the original AKI method for solving the MAE problem \eqref{eq_MAeig} on the smoothed square domain \eqref{eq.smoothsq}. 
 One can observe from the table that all methods yield identical eigenvalues $\lambda^h$  to six decimal places when $ h = 1/160$. 
 In terms of computation time, the inexact AKI-FP method is more than $11$ times faster than the AKI method ($\texttt{tol-inner}\equiv10^{-10}$)  when $h=1/160$.

\paragraph{Example 4. } Consider the MAE problem \eqref{eq_MAeig} on the unit  square domain 
\begin{equation}
   	\Omega=\{(x,y)\mid 0<x<1, 0<y< 1\}.
	\label{eq.sq}
\end{equation}
Figure \ref{subfig:square} illustrates the mesh we used. 
Figure \ref{fig.sq} presents the numerical results obtained using the inexact AKI-FP method with dynamically adjusted tolerance \eqref{eq:dynamic-adj} for this domain.  

The computed solution profile for $h=1/80$ is shown in Figure \ref{subfig:sq-solution}. Figures \ref{subfig:sq-convergence} and \ref{subfig:sq-lambda} depict the convergence history of the relative error $\eta_1$ and the computed eigenvalue estimate $\lambda^h$, respectively, for the case $h=1/80$. 
 We should mention that the total number of iterations is $243$ and we display only the first $30$ steps, which are sufficient to illustrate the steady decrease in the error and the rapid convergence of the eigenvalue estimate. 
Figure \ref{subfig:sq-time} illustrates the computation time for the inexact AKI-FP method with different mesh parameters $h$.

Table \ref{Tab.sq} compares the detailed performance of the inexact AKI-FP method with the original AKI method for the MAE problem \eqref{eq_MAeig} on the unit square domain \eqref{eq.sq}. 
In contrast to the smooth boundary cases, while the number of outer iterations remains largely unchanged, the count of Poisson equation solves required within the subproblems has increased. 
Nevertheless, the proposed inexact AKI-FP method is still capable of efficiently solving this non-smooth problem, and the computed eigenvalues $\lambda^h$ and minimum value $\min(\bm{u}^h)$ are nearly identical to those obtained using the original AKI method across all tested mesh sizes $h$. 
It can be seen that the inexact AKI-FP method takes approximately $38$ times less time than the AKI method $\texttt{tol-inner}\equiv10^{-10}$ when $h = 1/160$.

\begin{figure}
\centering
\begin{subfigure}{0.49\textwidth}
    \centering
    \includegraphics[width=0.7\linewidth]{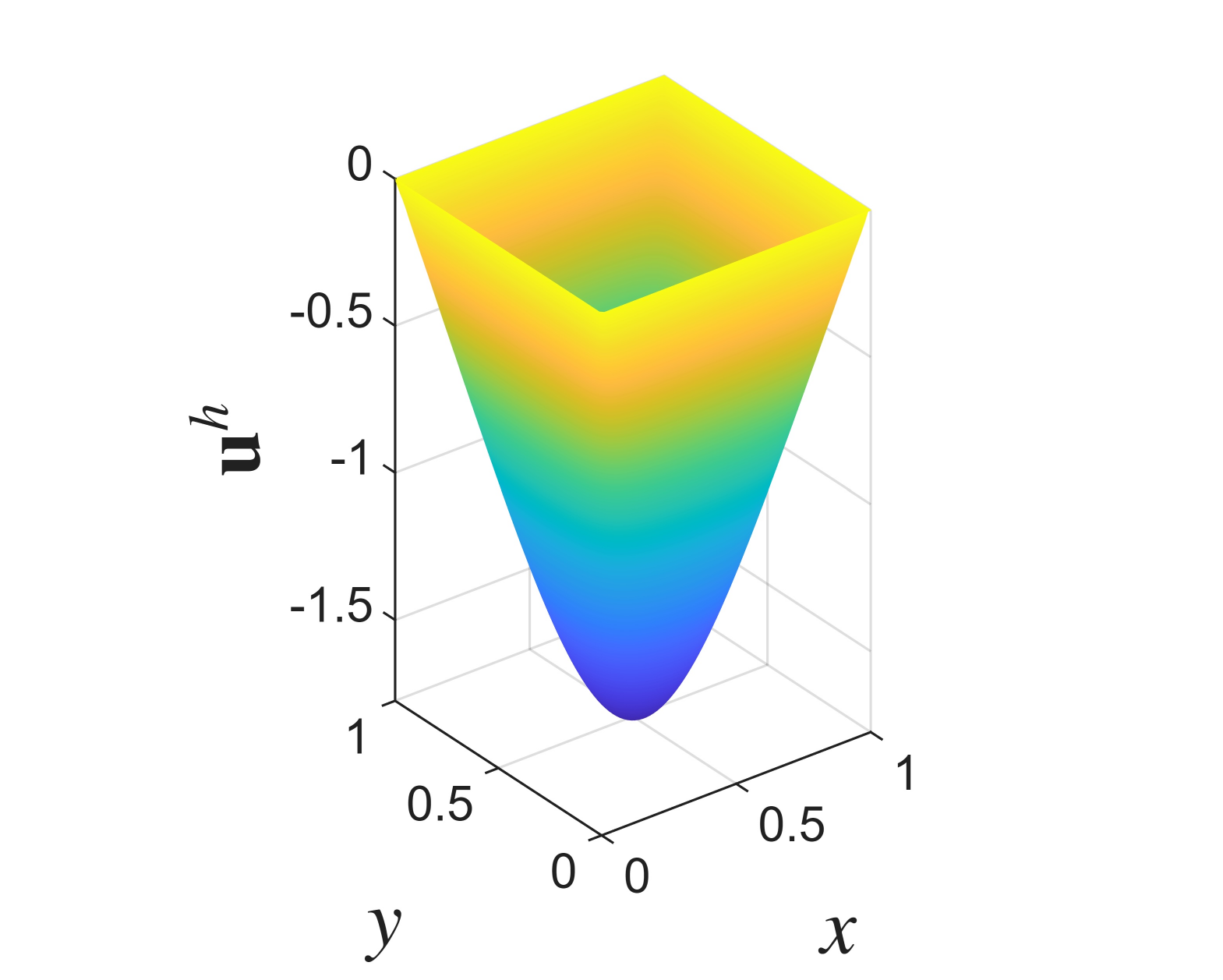}
    \caption{The computed solution profile with $h=1/80$}
    \label{subfig:sq-solution}
\end{subfigure}
\hfill
\begin{subfigure}{0.49\textwidth}
    \centering
    \includegraphics[width=0.7\linewidth]{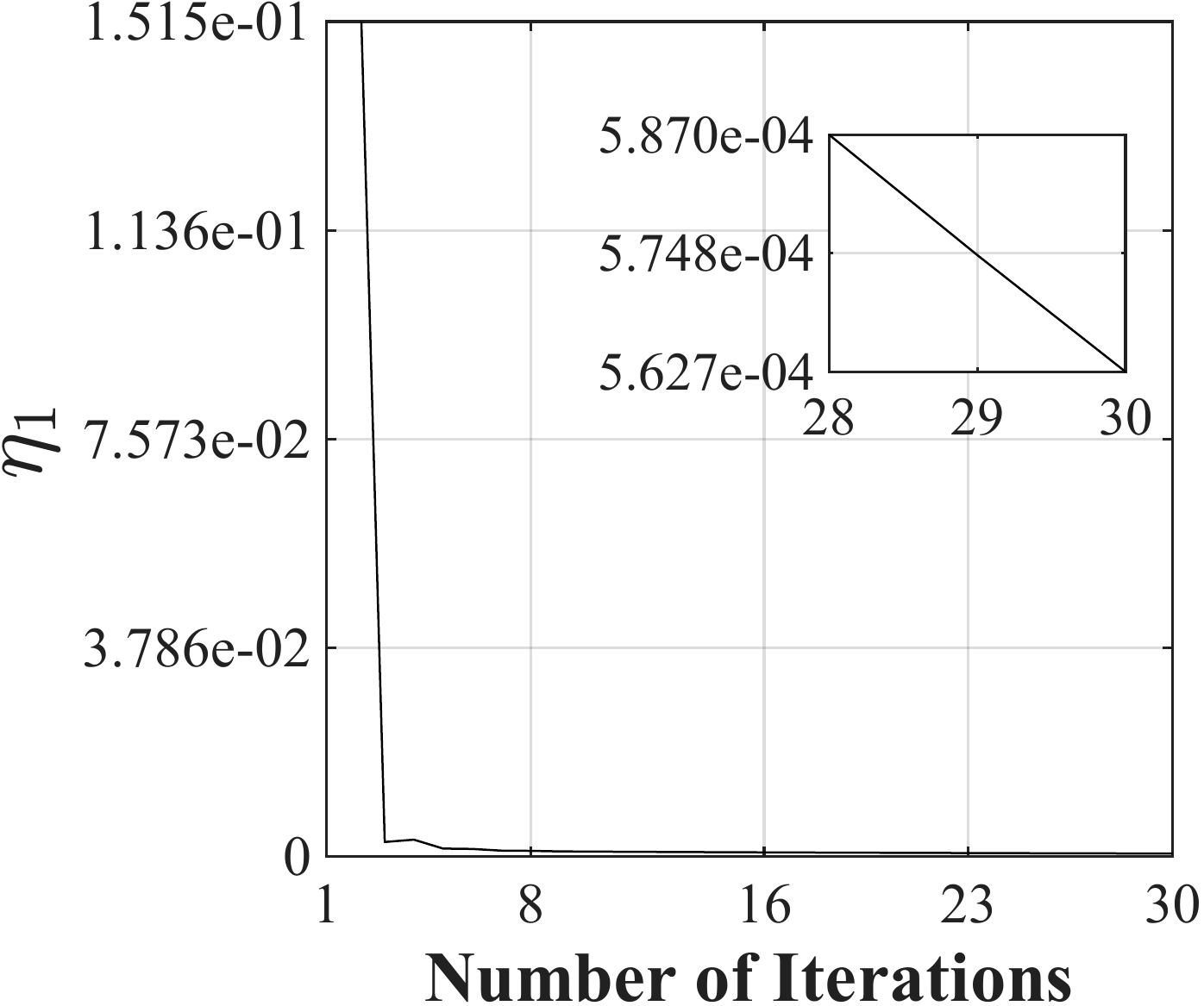}
    \caption{The convergence of $\eta_{1}$ with $h=1/80$}
    \label{subfig:sq-convergence}
\end{subfigure}

\vspace{2mm}

\begin{subfigure}{0.49\textwidth}
    \centering
    \includegraphics[width=0.7\linewidth]{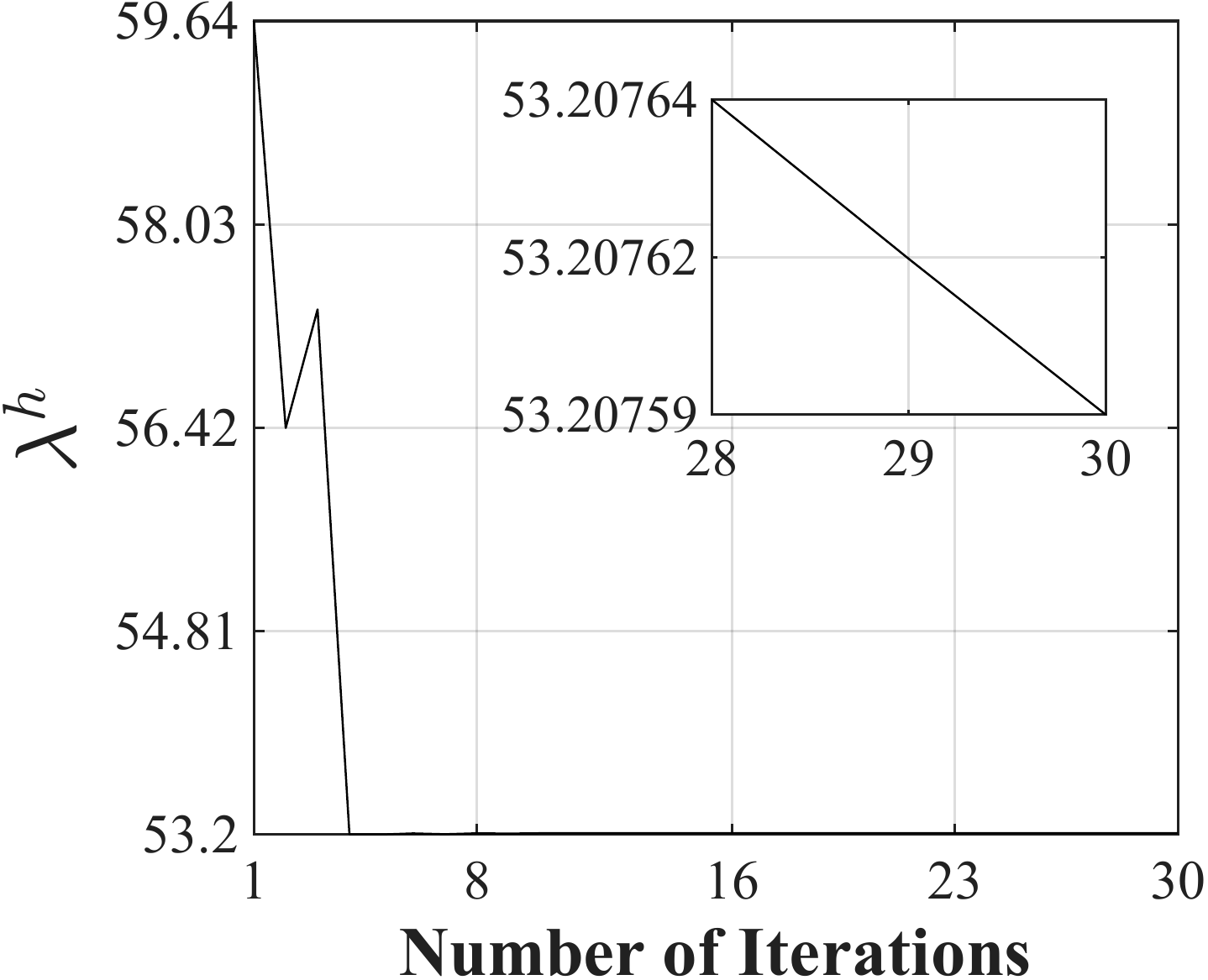}
    \caption{The value of $\lambda^h$ versus iteration $k$ with $h=1/80$}
    \label{subfig:sq-lambda}
\end{subfigure}
\hfill
\begin{subfigure}{0.49\textwidth}
    \centering
    \includegraphics[width=0.7\linewidth]{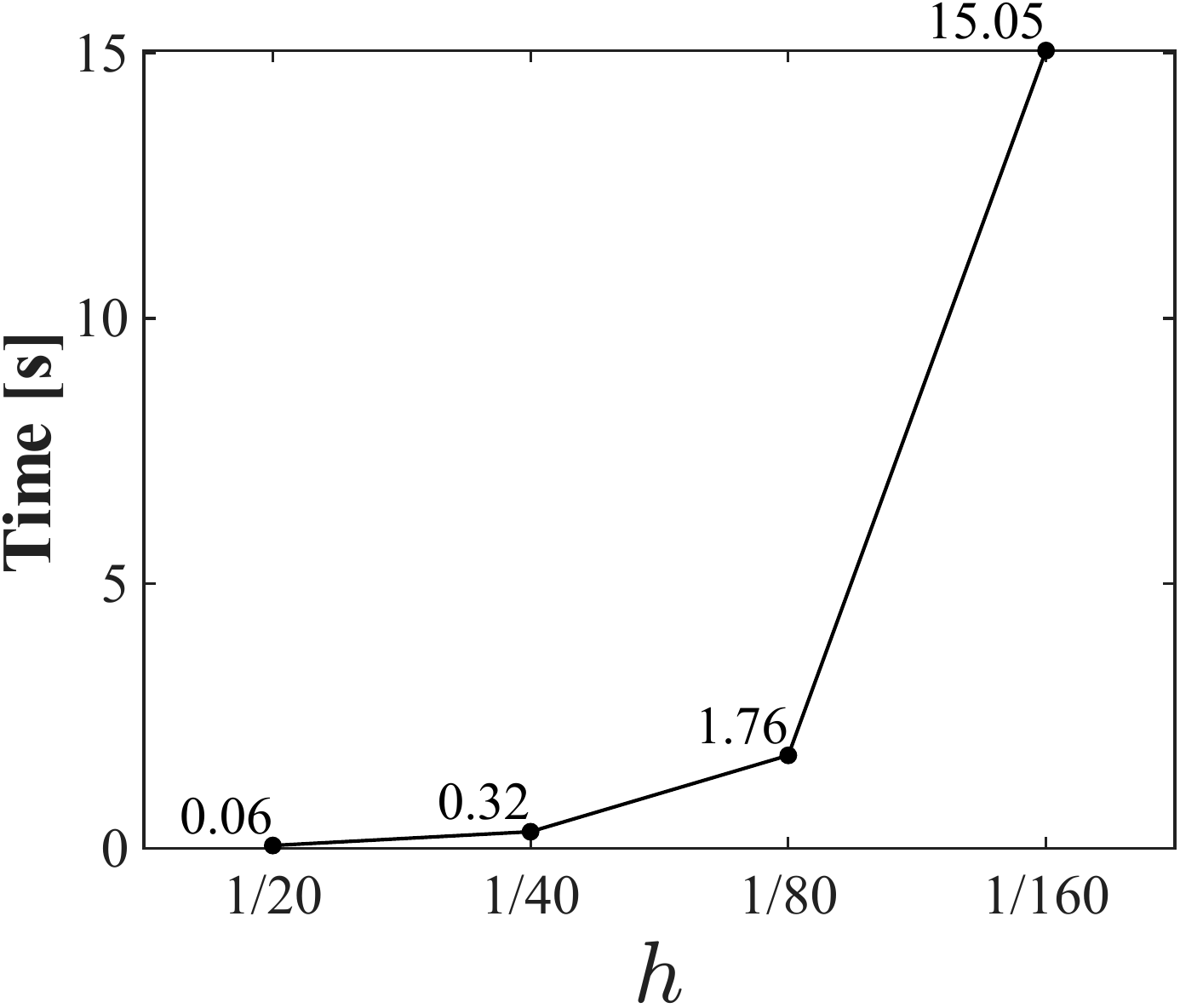}
    \caption{Computation time for different values of $h$}
    \label{subfig:sq-time}
\end{subfigure}

\caption{ 
 Performance of the inexact AKI-FP method on the unit square domain \eqref{eq.sq}}
\label{fig.sq}
\end{figure}

\begin{table}
\centering
\caption{Numerical performances of the inexact AKI-FP method and the original AKI method for solving the MAE problem  \eqref{eq_MAeig} on the 
unit square domain \eqref{eq.sq}. The notation is identical to Table \ref{Tab.disk1}.}
\label{Tab.sq}
\begin{tabular}{@{}ccccccccc@{}}
\toprule
Algorithm & $h$ & Iter & sub-Iter & \multicolumn{2}{c}{$\eta_1$} & $\lambda^h$ & $\min(\bm{u}^h)$ & Time [s]\\
\midrule
iAKI-FP & 1/20 & 144 & 144 & \multicolumn{2}{c}{$9.72\text{E-}{07}$} & 57.300303054 & -1.799393375 & $\bm{0.06}$ \\
AKI-$10^{-6}$ & 1/20 & 8 & 370 & \multicolumn{2}{c}{$6.48\text{E-}{07}$} & 57.300310360 & -1.799393630 & 0.08 \\
AKI-$10^{-10}$ & 1/20 & 8 & 1177 & \multicolumn{2}{c}{$5.39\text{E-}{07}$} & 57.300316507 & -1.799393782 & 0.22 \\
\midrule
iAKI-FP & 1/40 & 164&  173 &  \multicolumn{2}{c}{$9.87\text{E-}{07}$} & 54.154562523 & -1.780351125 & $\bm{0.32}$ \\
AKI-$10^{-6}$ & 1/40 & 9 & 905 & \multicolumn{2}{c}{$6.37\text{E-}{07}$} & 54.154733468 & -1.780378288 &  0.78 \\
AKI-$10^{-10}$ & 1/40 & 9 & 3584 & \multicolumn{2}{c}{$6.13\text{E-}{07}$} & 54.154708634 & -1.780377855 &  2.66 \\
\midrule
iAKI-FP & 1/80 &  241 & 306 & \multicolumn{2}{c}{$9.99\text{E-}{07}$} & 53.206715182 & -1.774914213 & $\bm{1.76}$ \\
AKI-$10^{-6}$ & 1/80 & 9 & 2747 & \multicolumn{2}{c}{$3.84\text{E-}{07}$} & 53.206711116 & -1.774948410 &  10.06 \\
AKI-$10^{-10}$ & 1/80 & 9 & 12693 & \multicolumn{2}{c}{$3.75\text{E-}{07}$} & 53.206703014 & -1.774948095 & 46.15 \\
\midrule
iAKI-FP & 1/160 & 243 &  542 &  \multicolumn{2}{c}{$9.96\text{E-}{07}$} & 52.823026411 & -1.772222292 & $\bm{15.05}$ \\
AKI-$10^{-6}$ & 1/160 & 9 & 7408 & \multicolumn{2}{c}{$2.85\text{E-}{07}$} & 52.823025681 & -1.772261303 & 166.35 \\
AKI-$10^{-10}$ & 1/160 & 9 & 25619 & \multicolumn{2}{c}{$2.77\text{E-}{07}$} & 52.823025571 & -1.772261136 & 574.81\\
\bottomrule
\end{tabular}
\end{table}

\subsection{Numerical experiments in \texorpdfstring{$\mathbb{R}^3$}{num R3}}

In this subsection, we consider the inexact AKI-FP method (Algorithm \ref{alg:fix}) for 3D MAE problems on a smoothed cube, a unit ball, and an ellipsoid, and set $g = 0.25$ in \eqref{initial condition} and $(\zeta_1, \zeta_2) = (0.99, 0.9)$ in \eqref{eq:dynamic-adj}.
For the original AKI method, we choose $\texttt{tol-inner} \equiv 10^{-6}$ for comparison. 
It is worth noting that even with this tolerance, achieving such accuracy can be challenging for the subproblems of certain examples in the 3D cases.
Therefore, we set an extra condition that the maximum number of inner iterations is no greater than $500$.

\paragraph{Example 5.} Consider the MAE problem \eqref{eq_MAeig} on the unit ball domain with
\begin{equation}
   	\Omega=\{(x,y,z)\mid x^2+y^2+z^2 < 1\}.
	\label{eq.ball}
\end{equation}
Figure \ref{subfig:ball} illustrates the mesh we used.
We demonstrate the performance of the proposed inexact AKI-FP method using the finite element method on tetrahedral meshes with characteristic lengths $h = 1/16, 1/24, 1/32$, as reported in Figure  \ref{fig.ball} and Table \ref{Tab.ball}.

Figure \ref{fig.ball} presents the detailed results for the inexact AKI-FP method for $h=1/24$. Figures \ref{subfig:ball-convergence} and \ref{subfig:ball-lambda} show the convergence of the relative error $\eta_1$ and the computed eigenvalue estimate $\lambda^h$ versus the iteration count $k$, respectively. 
One can see that $\eta_1$ decreases to the tolerance of $10^{-6}$, while the sequence $\lambda^h$ approaches the corresponding eigenvalue. 

Table \ref{Tab.ball} shows the computational results of the inexact AKI-FP method with a dynamically adjusted tolerance \eqref{eq:dynamic-adj} and the original AKI method on the unit ball domain \eqref{eq.ball}. 
The relative residual $\eta_1$ remains of the order $10^{-7}$, indicating that the fixed-point iteration achieves stable accuracy. 
On the finest mesh, the inexact AKI-FP method is about $34$ times faster than the AKI method with $\texttt{tol-inner}\equiv10^{-6}$. 
Moreover, the computed eigenvalue aligns with the theoretical bounds derived in the paper \cite{le2025large}. 
Specifically, it was shown that for the unit ball in $\mathbb{R}^3$, the exact value of the eigenvalue must lie within the range $[20,35]$.

\begin{figure}[H]
\centering
\begin{subfigure}{0.49\textwidth}
    \centering
    \includegraphics[width=0.78\linewidth]{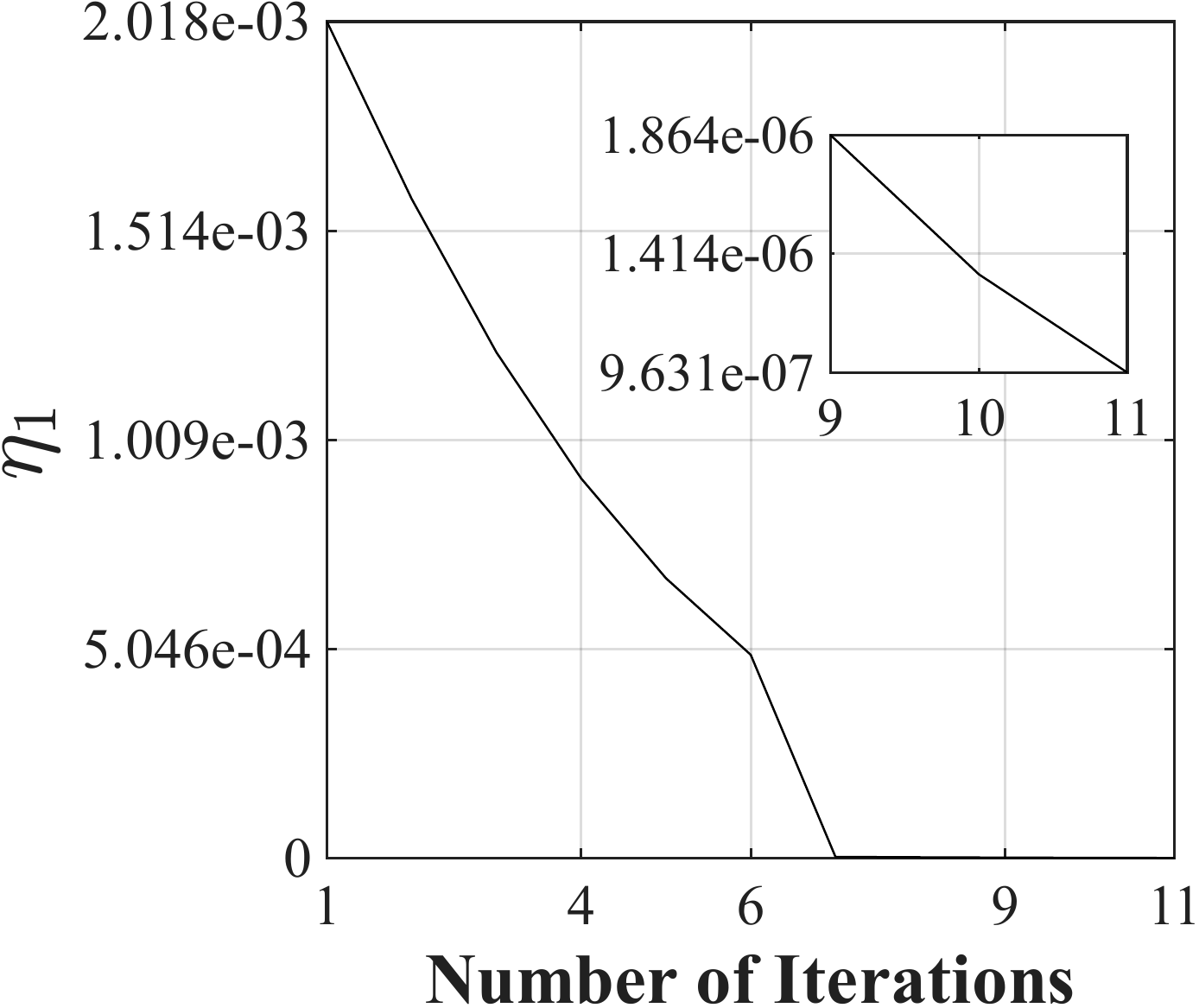}
    \caption{The convergence of $\eta_{1}$ with $h=1/24$}
    \label{subfig:ball-convergence}
\end{subfigure}
\hfill
\begin{subfigure}{0.49\textwidth}
    \centering
    \includegraphics[width=0.8\linewidth]{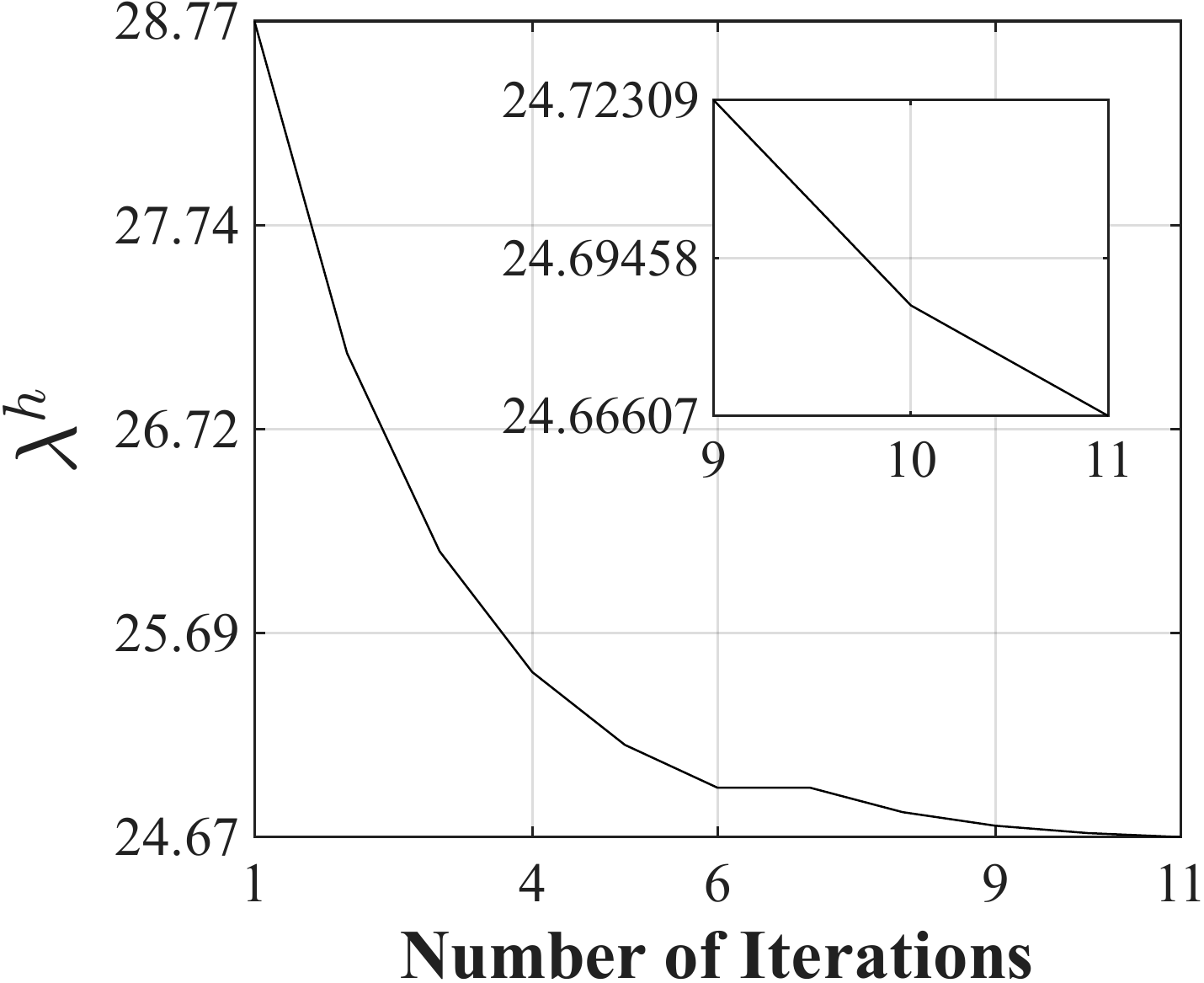}
    \caption{The value of $\lambda^h$ versus iteration $k$ with $h=1/24$}
    \label{subfig:ball-lambda}
\end{subfigure}

\caption{Performance of the inexact AKI-FP method on the unit ball domain \eqref{eq.ball}}
\label{fig.ball}
\end{figure}

\begin{table}[H]
\centering
\caption{Numerical performances of the inexact AKI-FP method and the original AKI method for solving the MAE problem  \eqref{eq_MAeig} on the 
 unit ball domain  \eqref{eq.ball}. The notation is identical to Table \ref{Tab.disk1}.}
\label{Tab.ball}
\begin{tabular}{@{}ccccccccc@{}}
\toprule
Algorithm & $h$ & Iter & sub-Iter & \multicolumn{2}{c}{$\eta_1$} & $\lambda^h$ & $\min(\bm{u}^h)$ & Time [s]\\
\midrule
iAKI-FP & 1/16 & 18 & 60  & \multicolumn{2}{c}{$6.54\text{E-}{07}$} & 24.623565100 & -1.166148725 & $\bm{4.03}$ \\
AKI-$10^{-6}$ & 1/16 & 8  & 178 & \multicolumn{2}{c}{$6.85\text{E-}{07}$} & 24.623594185 & -1.166107699 & 8.65 \\
\midrule
iAKI-FP & 1/24 & 11 & 10  & \multicolumn{2}{c}{$9.63\text{E-}{07}$} &  24.666071872 & -1.155181630 & $\bm{7.15}$ \\
AKI-$10^{-6}$ & 1/24 & 6  & 134  & \multicolumn{2}{c}{$9.74\text{E-}{07}$} & 24.660132098 & -1.156046554 & 36.43 \\
\midrule
iAKI-FP & 1/32 & 12 & 11 & \multicolumn{2}{c}{$8.89\text{E-}{07}$} & 24.662240036 &  -1.158627146 & $\bm{27.20}$ \\
AKI-$10^{-6}$ & 1/32 & 10  & 1007 & \multicolumn{2}{c}{$8.00\text{E-}{07}$} & 24.649171102 & -1.166158550 & 923.39  \\
\bottomrule
\end{tabular}
\end{table}

\paragraph{Example 6.} Consider the MAE problem \eqref{eq_MAeig} on the ellipsoid domain
\begin{equation}
   	\Omega=\{(x,y,z)\mid x^2+1.5y^2+2z^2 < 1\}.
	\label{eq.ellipsoid}
\end{equation}
 Figure \ref{subfig:ellipsoid} illustrates the mesh we used.  
The algorithms are tested with varying characteristic lengths $h = 1/16, 1/24, 1/32$. 

The results shown in Figure \ref{fig.ellipsoid} are similar to the unit ball domain example. 
Figure \ref{subfig:ellipsoid-convergence} illustrates the downward trajectory of the relative error $\eta_1$, while Figure \ref{subfig:ellipsoid-lambda} presents the converging path of the computed eigenvalues $\lambda^h$.

Table \ref{Tab.ellipsoid} presents the computational results of the inexact AKI-FP method and the original AKI method on the ellipsoid domain \eqref{eq.ellipsoid}. 
The residual $\eta_1$ is consistently of the order $10^{-7}$, confirming the robustness of the fixed-point iteration. 
On the finest mesh, the inexact AKI-FP method is about $5$ times faster than the AKI method with $\texttt{tol-inner}\equiv10^{-6}$.

\begin{figure}[H]
\centering
\begin{subfigure}{0.49\textwidth}
    \centering
    \includegraphics[width=0.78\linewidth]{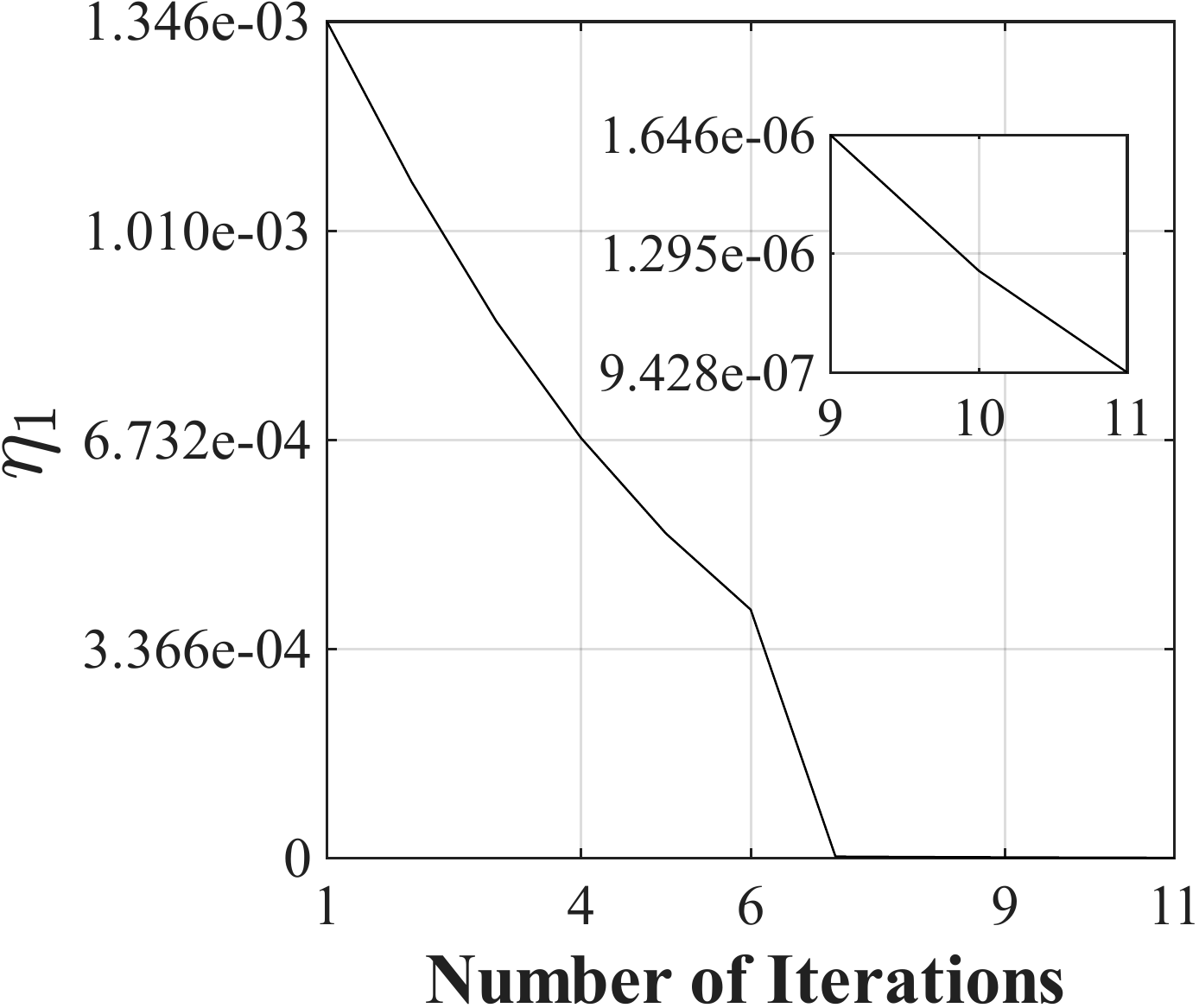}
    \caption{The convergence of $\eta_{1}$ with $h=1/24$}
    \label{subfig:ellipsoid-convergence}
\end{subfigure}
\hfill
\begin{subfigure}{0.49\textwidth}
    \centering
    \includegraphics[width=0.8\linewidth]{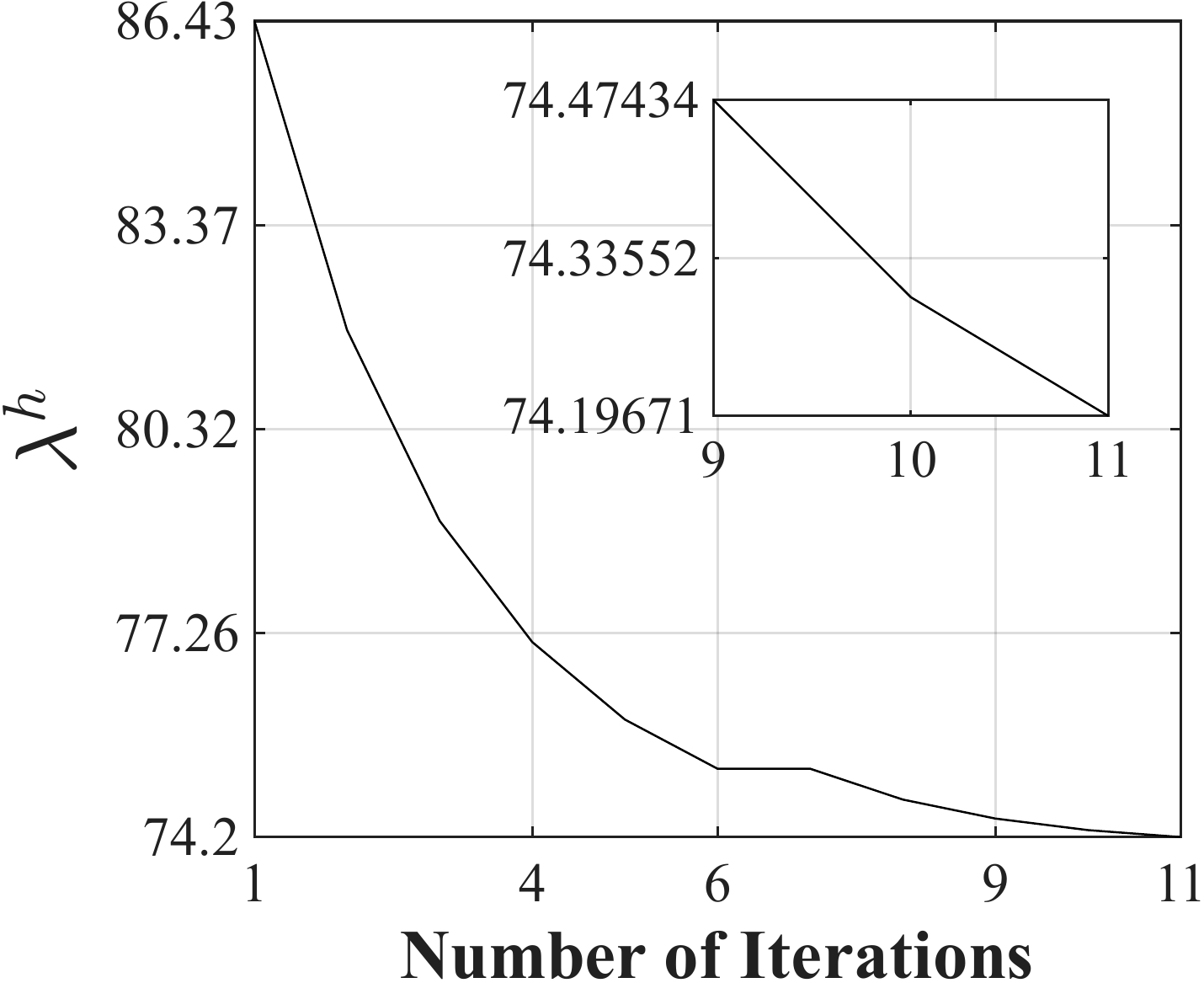}
    \caption{The value of $\lambda^h$ versus iteration $k$ with $h=1/24$}
    \label{subfig:ellipsoid-lambda}
\end{subfigure}

\caption{Performance of the inexact AKI-FP method on the ellipsoid domain \eqref{eq.ellipsoid}}
\label{fig.ellipsoid}
\end{figure}

\begin{table}[H]
\centering
\caption{Numerical performances  of the inexact AKI-FP method (denoted ``iAKI-FP'') and the original AKI method (denoted ``AKI-$10^{-6}$'') for solving the MAE problem \eqref{eq_MAeig} on the ellipsoid domain  \eqref{eq.ellipsoid}.  
The notation is identical to Table \ref{Tab.disk1}.}
\label{Tab.ellipsoid}
\begin{tabular}{@{}ccccccccc@{}}
\toprule
Algorithm & $h$ & Iter & sub-Iter & \multicolumn{2}{c}{$\eta_1$} & $\lambda^h$ & $\min(\bm{u}^h)$ & Time [s]\\
\midrule
iAKI-FP & 1/16 & 43 & 43 & \multicolumn{2}{c}{$9.98\text{E-}{07}$} &  74.045923981 & -1.536002186 & $\bm{2.26}$ \\
AKI-$10^{-6}$ & 1/16 & 8  & 124 & \multicolumn{2}{c}{$5.33\text{E-}{07}$} & 74.046149885 & -1.535957548 & 3.03\\
\midrule
iAKI-FP & 1/24 & 11 &10 &  \multicolumn{2}{c}{$9.43\text{E-}{07}$} & 74.196706807 & -1.516238370 & $\bm{3.32}$ \\
AKI-$10^{-6}$ & 1/24 & 6  & 128 &  \multicolumn{2}{c}{$7.23\text{E-}{07}$} & 74.098853083 & -1.521981439 & 15.19 \\
\midrule
iAKI-FP & 1/32 & 11 & 10 & \multicolumn{2}{c}{$9.47\text{E-}{07}$} & 74.155816346 &-1.516792104  & $\bm{10.18}$ \\
AKI-$10^{-6}$ & 1/32 & 6  & 122 & \multicolumn{2}{c}{$7.20\text{E-}{07}$} & 74.053202968 & -1.522402340 & 48.02  \\
\bottomrule
\end{tabular}
\end{table}

\paragraph{Example 7.} Consider the MAE problem \eqref{eq_MAeig} on the smoothed cube domain
\begin{equation}
   	\Omega=\{(x,y,z)\mid (|x|^3+|y|^3+|z|^3)^{{1}/{3}} < 0.75\}.
	\label{eq.smoothedcube}
\end{equation}
Figure \ref{subfig:smoothedcube} illustrates the mesh we used.

Figures  \ref{subfig:smoothedcube-convergence} and \ref{subfig:smoothedcube-lambda} display the convergence of the relative error $\eta_1$ and the computed eigenvalue estimate $\lambda^h$ versus the iteration count $k$, respectively, obtained with $h=1/24$. 

Table \ref{Tab.smoothedcube} summarizes the computational results of the inexact AKI-FP method and the original AKI method on the smoothed cube domain \eqref{eq.smoothedcube}. The residual $\eta_1$ remains at the level of $10^{-7}$, which illustrates the stable accuracy of the fixed-point scheme. 
On the finest mesh, the inexact AKI-FP method is about $9$ times faster than the AKI method with $\texttt{tol-inner}\equiv10^{-6}$.

In summary, the numerical results demonstrate that the proposed inexact AKI method offers comparable robustness to the original AKI method while achieving the same level of accuracy several times faster.

\begin{figure}[ht]
\centering
\begin{subfigure}{0.49\textwidth}
    \centering
    \includegraphics[width=0.78\linewidth]{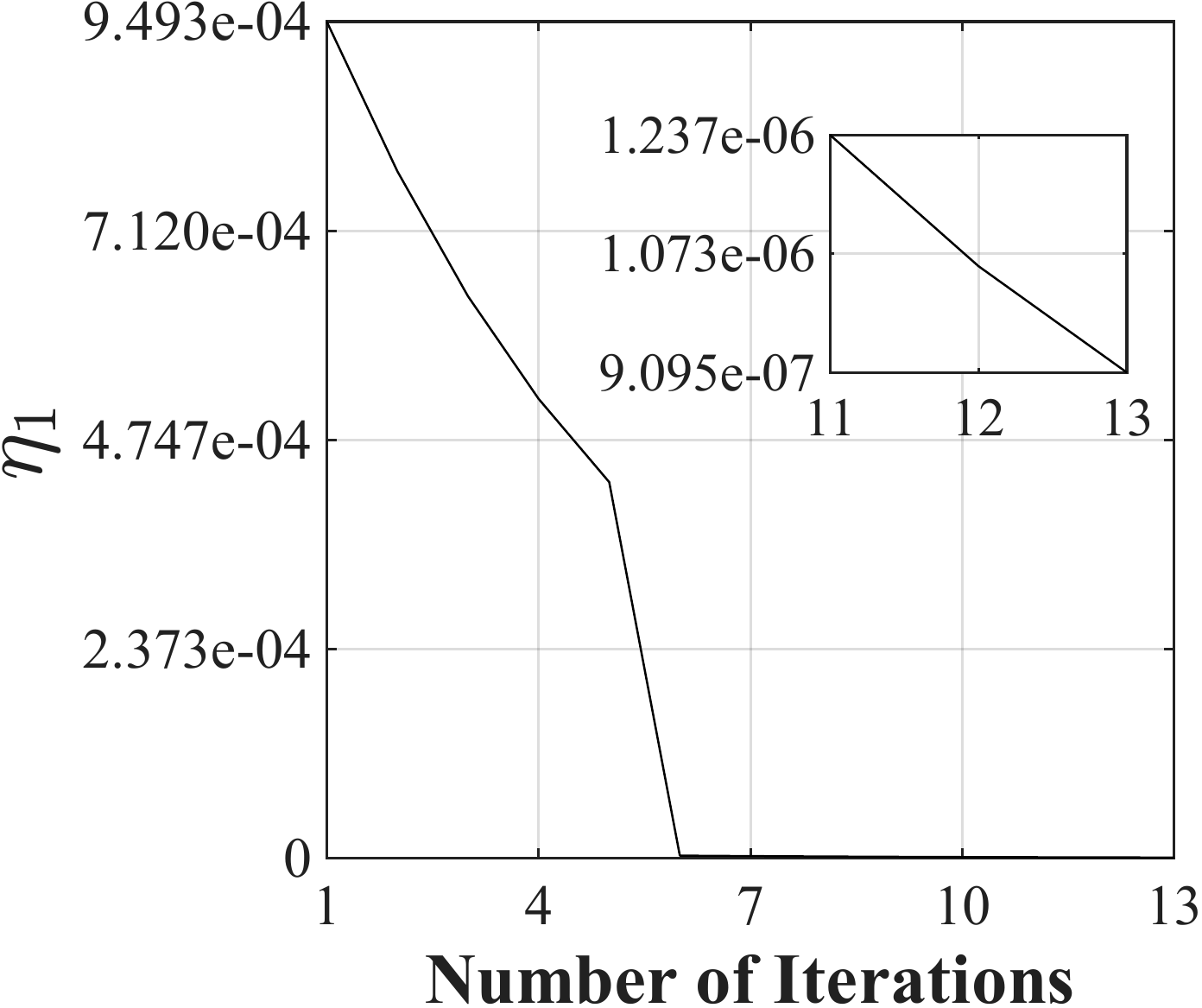}
    \caption{The convergence of $\eta_{1}$ with $h=1/24$}
    \label{subfig:smoothedcube-convergence}
\end{subfigure}
\hfill
\begin{subfigure}{0.49\textwidth}
    \centering
    \includegraphics[width=0.8\linewidth]{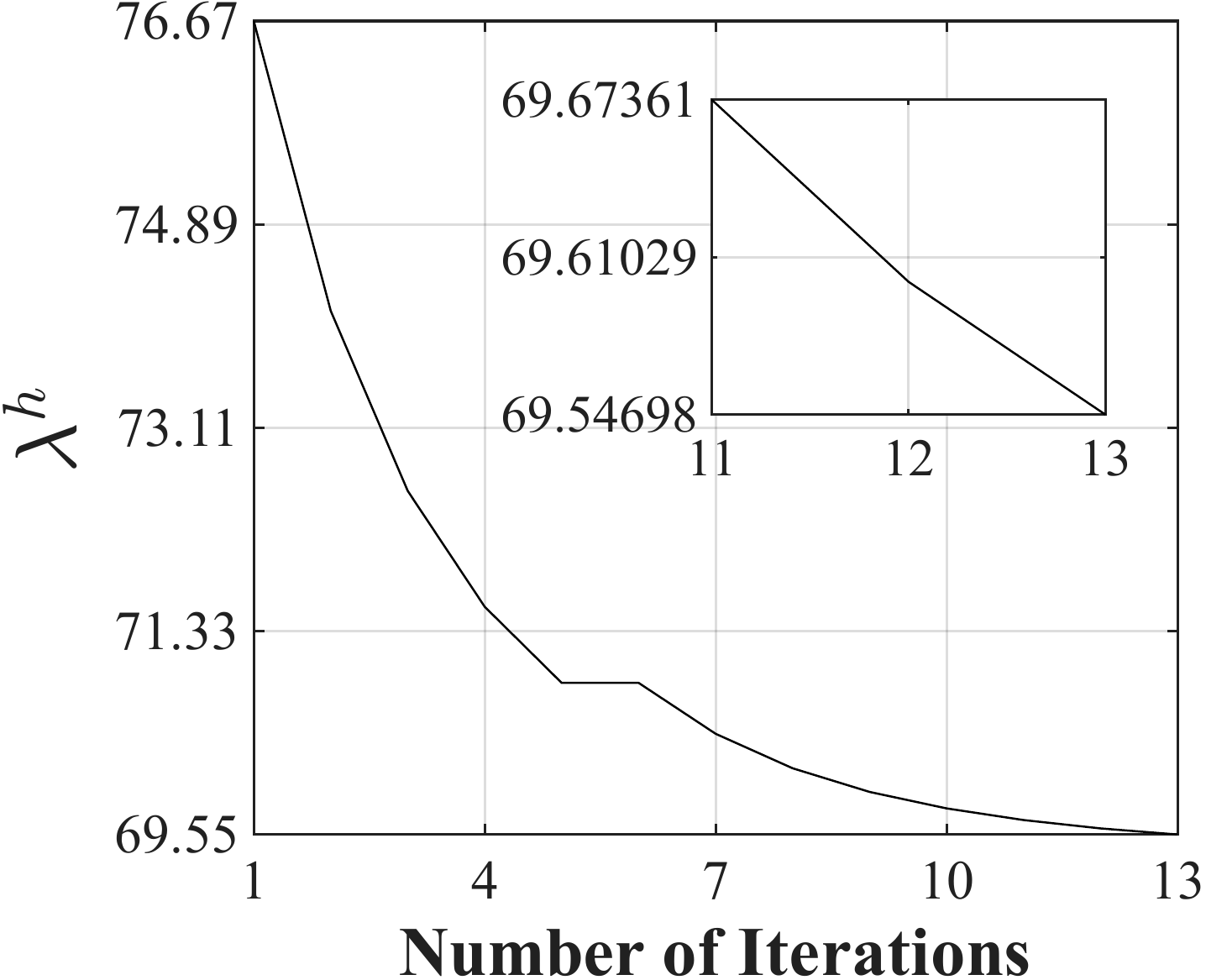}
    \caption{The value of $\lambda^h$ versus iteration $k$ with $h=1/24$}
    \label{subfig:smoothedcube-lambda}
\end{subfigure}

\caption{Numerical results obtained using the inexact AKI-FP method on the smoothed cube domain \eqref{eq.smoothedcube}}
\label{fig.smoothedcube}
\end{figure}

\begin{table}[ht]
\centering
\caption{Numerical performances of the inexact AKI-FP method and the original AKI method for solving the MAE problem  \eqref{eq_MAeig} on the 
 smoothed cube domain  \eqref{eq.smoothedcube}. The notation is identical to Table \ref{Tab.disk1}.}
\label{Tab.smoothedcube}
\begin{tabular}{@{}ccccccccc@{}}
\toprule
Algorithm & $h$ & Iter & sub-Iter & \multicolumn{2}{c}{$\eta_1$} & $\lambda^h$ & $\min(\bm{u}^h)$ & Time [s]\\
\midrule
iAKI-FP & 1/16 & 44 & 69 & \multicolumn{2}{c}{$8.96\text{E-}{07}$} & 69.245540951 &-1.523322925& $\bm{2.84}$ \\
AKI-$10^{-6}$ & 1/16 & 8 & 255 & \multicolumn{2}{c}{$3.48\text{E-}{07}$} & 69.245276351 & -1.523261475 & 5.89 \\
\midrule
iAKI-FP & 1/24 & 13 & 12 &   \multicolumn{2}{c}{$9.09\text{E-}{07}$}& 69.546975024 & -1.509177099 & $\bm{3.85}$ \\
AKI-$10^{-6}$ & 1/24 & 6 & 177 & \multicolumn{2}{c}{$6.07\text{E-}{07}$} & 69.427400110 & -1.509410659 & 22.14 \\
\midrule
iAKI-FP & 1/32 & 12 & 37 &  \multicolumn{2}{c}{$8.22\text{E-}{07}$} &  69.416716416 & -1.519558826& $\bm{21.92}$ \\
AKI-$10^{-6}$ & 1/32 & 7 & 505 & \multicolumn{2}{c}{$8.77\text{E-}{07}$} & 69.450358108 & -1.512183397 & 207.50  \\
\bottomrule
\end{tabular}
\end{table}

\section{Conclusions}
\label{sec:conclusions}
In this paper, we proposed an efficient inexact variant of the Abedin-Kitagawa iteration method for solving the Monge-Amp{\`e}re eigenvalue problem (Eq. \eqref{eq_MAeig}) on bounded convex domains. Unlike the original method, which requires high-accuracy solutions to the subproblems at each iteration, the proposed approach asymptotically increases the accuracy of the subproblems, significantly enhancing computational efficiency. 
The convergence analysis for the proposed inexact method is provided.
Moreover, for the 2D and 3D cases, a fixed-point method is introduced
for the subproblems, with the convergence of the fixed-point method established under $\mathcal{C}^{2,\alpha}$ boundary conditions.
Numerical experiments across various cases demonstrate that the proposed method is highly efficient and effective, achieving significant improvements in computational efficiency over the original Abedin-Kitagawa iteration.
As a future research direction, it would be valuable to investigate how to rigorously establish the convergence of the fixed-point method in higher dimensions. 
Another possible direction is to extend the convergence analysis to situations with possibly negative error functions, while using the weaker assumptions that suffice for nonnegative error functions.
Besides, it is of interest to examine the applicability of such inexact fixed-point computation methods to the $k$-Hessian eigenvalue problems.

\normalsize
\appendix
\section{Proofs of technical results in Section \ref{sec:An Inexact AKI method}}
\label{appendix}
\subsection{Proof of Proposition \ref{inequality1}}
\label{appendixineq}
\begin{proof}
Since $u_{k+1}$ satisfies the perturbed subproblem \eqref{app sub}, multiplying both sides by $|u_{k+1}|$ and integrating over $\Omega$ gives
\begin{equation*}
    \int_{\Omega} |u_{k+1}|\det D^2 u_{k+1}\,\mathrm{d}\bm{x}
    = R(u_k)\int_{\Omega} |u_k|^d |u_{k+1}|\,\mathrm{d}\bm{x} 
    + \int_{\Omega}\varepsilon_k |u_{k+1}|\,\mathrm{d}\bm{x}.
\end{equation*}
By the definition of $R(u_{k+1})$ in \eqref{rayli}, this becomes
\begin{equation}
\label{lemma3.2.2}
   R(u_{k+1}) \|u_{k+1}\|_{L^{d+1}(\Omega)}^{d+1} 
   = R(u_k)\int_{\Omega} |u_k|^d |u_{k+1}|\,\mathrm{d}\bm{x} 
   + \int_{\Omega}\varepsilon_k |u_{k+1}|\,\mathrm{d}\bm{x}.
\end{equation}
For the first term on the right-hand side, the H\"older's inequality yields
\begin{equation*}
\int_{\Omega} |u_k|^d |u_{k+1}|\,\mathrm{d}\bm{x}
\le \|u_k\|_{L^{d+1}(\Omega)}^d \|u_{k+1}\|_{L^{d+1}(\Omega)}.
\end{equation*}
For the second term, noticing the error criterion  in Algorithm \ref{alg:inexact_AK}, applying the H\"older's inequality gives
\begin{equation}
\label{lemma3.2.4}
\int_{\Omega}\varepsilon_k |u_{k+1}|\,\mathrm{d}\bm{x} 
\le\xi_k\int_{\Omega}   |u_{k+1}|\,\mathrm{d}\bm{x} 
\le \left(\mathcal{L}^d(\Omega)\right)^{\frac{d}{d+1}}\xi_k \|u_{k+1}\|_{L^{d+1}(\Omega)}.
\end{equation}
Combining \eqref{lemma3.2.2} and \eqref{lemma3.2.4} gives
\begin{equation}
\label{lemma3.2.5}
   R(u_{k+1}) \|u_{k+1}\|_{L^{d+1}(\Omega)}^{d+1} 
   \le R(u_k) \|u_k\|_{L^{d+1}(\Omega)}^d \|u_{k+1}\|_{L^{d+1}(\Omega)} 
   + \left(\mathcal{L}^d(\Omega)\right)^{\frac{d}{d+1}}\xi_k \|u_{k+1}\|_{L^{d+1}(\Omega)}.
\end{equation}
From Proposition \ref{inftyinf}, $u_k \not\equiv 0$ for all $k\ge0$.
Then dividing both sides of \eqref{lemma3.2.5} by $\|u_{k+1}\|_{L^{d+1}(\Omega)}$ yields \eqref{inequalitye1}, completing the proof.
\end{proof}

\subsection{Proof of Proposition \ref{eventual smoothness1}}
\label{appendixsmoothness}
\begin{proof}
According to the perturbed subproblem \eqref{app sub} and the Aleksandrov maximum principle (Lemma \ref{Aleksandrov maximum principle}), for any $k\ge1$ and $\bm{x}\in\Omega$,
\begin{align*}
|u_{k}(\bm{x})|^d 
&\le C_d \diam(\Omega)^{d-1} \dist(\bm{x},\partial\Omega) 
   \int_{\Omega}\det D^2 u_{k}\,\mathrm{d}\bm{x} \\
&= C_d \diam(\Omega)^{d-1} \dist(\bm{x},\partial\Omega) 
   \Big(R(u_{k-1})\int_{\Omega}|u_{k-1}|^d\,\mathrm{d}\bm{x} + \int_{\Omega}\varepsilon_{k-1}\,\mathrm{d}\bm{x}\Big).
\end{align*}
Applying H\"older's inequality and Proposition \ref{inequality1}, we obtain
\[
|u_{k}(\bm{x})|^d \le 
C_d \diam(\Omega)^{d-1}\dist(\bm{x},\partial\Omega)\left(\mathcal{L}^d(\Omega)\right)^{\frac{1}{d+1}}
\Big(R(u_0)\|u_0\|_{L^{d+1}(\Omega)}^d + \left(\mathcal{L}^d(\Omega)\right)^{\frac{d}{d+1}}\sum\limits_{i=0}^{k-1} \xi_i\Big).
\]
Since $\{\xi_k\}_{k\geq0}$ is summable and $0<R(u_0)<\infty$ from Assumption \ref{ass_blanket1}, there exists $C_1(d,\Omega,u_0,M)>0$ such that
$\sup_\Omega |u_{k}| \le C_1(d,\Omega,u_0,M)$ for all $k\ge1$.
Thus, $u_{k}$ is uniformly bounded for all $k\geq1$. 
By the interior gradient estimate (c.f. \cite[Lemma 3.2.1]{gutierrez2016monge}) and the boundary condition, $u_{k}$ is uniformly Lipschitz on compact subsets and uniformly $\tfrac{1}{d}$-H\"older near $\partial\Omega$. 
Hence, there exists $C(d,\Omega,u_0,M)>0$ such that, for all $k\geq 1$,
\[
u_{k}\in \mathcal{C}^{0,{1}/{d}}(\overline{\Omega})\quad\mbox{and}\quad  \|u_{k}\|_{\mathcal{C}^{0,{1}/{d}}(\overline{\Omega})}\le C(d,\Omega,u_0,M).
\]

We now prove by induction that $u_{k}$ is strictly convex and belongs to $\mathcal{C}^{2(k-1),{1}/{d}}(\Omega)$ for all $k\geq 2$ by induction.  
Define $M_{k}:=\|u_{k}\|_{L^\infty(\Omega)}$. 
For the base case $k = 2$, by Proposition \ref{inftyinf}, we have $u_1\not\equiv0$. 
Convexity implies $u_1<0$ in $\Omega$. 
For $\tau_1\in(0,M_2)$, the level set $\Omega_1'=\{x\in\Omega\mid u_2(x)\le -\tau_1\}$ is convex with nonempty interior, and the continuity gives $|u_1|\ge m_1>0$ in $\overline{\Omega_1'}$. 
From \eqref{inf},
$$
\begin{cases}
\lambda_{\rm  MA} m_1^d\leq \det D^2 u_{2}=R(u_1)|u_1|^d + \varepsilon_1  \leq R(u_1) M_1^d+\xi_1 &\text{ in }\Omega_1',
\\
u_{2}= -\tau_1 &\text{ on }\partial\Omega_1'. 
\end{cases}
$$ 
According to the Caffarelli’s localization theorem \cite{caffarelli1990localization} (see also \cite[Theorem 4.10]{figalli2017monge} and \cite[Corollary 5.2.2]{gutierrez2016monge}), $u_2$ is strictly convex in $\Omega_1'$.
Since $u_1\in \mathcal{C}^{0,{1}/{d}}(\Omega_1')$, using the Caffarelli's $\mathcal{C}^{2,\alpha}$ estimates \cite{caffarelli1990interior}, we have $u_2\in \mathcal{C}^{2,{1}/{d}}_{\rm loc}(\Omega_1')$. 
In addition, due to the arbitrary choice of $\tau_1\in (0, M_2)$, we conclude that $u_2\in \mathcal{C}^{2,{1}/{d}}(\Omega)$, and $u_2$ is strictly convex in $\Omega$.

Suppose that $u_{k}$ is strictly convex in $\Omega$ and $u_{k}\in \mathcal{C}^{2k, {1}/{d}}(\Omega)$ for all $k \ge2$ holds up to $k=n-1$ where $n\geq 3$ is a certain integer. 
It is then sufficient to prove the proposition by showing that the above condition also holds for $k=n$. 
For each $\tau_{n-1}\in (0, M_{n})$, the level set $\Omega_{n-1}'=\{x\in\Omega\mid u_{n}(x)\leq -\tau_{n-1}\}$ is convex with nonempty interior, and continuity gives $|u_{n-1}|\geq m_{n-1}>0$ in $\overline{\Omega_n'}$.
Note that
$$
\begin{cases}
\lambda_{\rm MA} m_{n-1}^d\leq \det D^2 u_{n}=R(u_{n-1})|u_{n-1}|^d+\varepsilon_{n-1}  \leq R(u_{n-1}) M_{n-1}^d+\xi_{n-1}&\text{ in }\Omega_n',
\\
u_{n}= -\tau_{n-1} & \text{ on }\partial\Omega_n'. 
\end{cases}
$$ 
Thus, the function $u_{n}$ is strictly convex in $\Omega_{n-1}'$ by the Caffarelli’s localization theorem \cite{caffarelli1990localization}. 
Furthermore, by the induction hypothesis, $u_{n-1}\in \mathcal{C}^{2((n-1)-1),{1}/{d}}(\Omega_{n-1}')$.
In the interior of $\Omega_{n-1}'$, the equation $\det D^2 u_{n}= R(u_{n-1})|u_{n-1}|^d+\varepsilon_{n-1}$ now becomes uniformly elliptic with $\mathcal{C}^{2((n-1)-1),{1}/{d}}$ right-hand side. 
Therefore, we have $u_{n}\in \mathcal{C}^{2(n-1),{1}/{d}}_{\rm loc}(\Omega_{n-1}')$.
Finally, since $\tau_{n-1}\in (0, M_{n})$ is arbitrary, we conclude that $u_{n}\in \mathcal{C}^{2(n-1), {1}/{d}}(\Omega)$, and $u_{n}$ is strictly convex in $\Omega$.
This completes the proof.  
\end{proof}

\end{document}